\documentclass[reqno,a4paper,11pt]{amsart}
\pdfoutput=1

\usepackage{mystyle}
\usepackage{etoolbox}

\usepackage[dvipsnames]{xcolor}
\definecolor{highlightArrow}{rgb}{0.33, 0.25, 0.85}
\definecolor{globjects}{rgb}{0.0,0.6,0.2}
\definecolor{functorL}{rgb}{0.7,0.0,0.0}
\definecolor{functorM}{rgb}{0.0,0.2,0.4}
\definecolor{functorF1}{rgb}{0.0,0.5,0.2}
\definecolor{functorF2}{rgb}{0.4,0.0,0.5}
\definecolor{functorF3}{rgb}{0.75,0.6,0.0}

\usepackage[headings]{fullpage}
\usepackage{parskip}

\usepackage[utf8]{inputenc}
\usepackage[T1]{fontenc}
\usepackage[british]{babel}
\usepackage[pdfencoding=auto]{hyperref}

\usepackage{microtype}
\usepackage{booktabs}

\usepackage{amsfonts,amssymb,bbm,mathrsfs,stmaryrd}
\usepackage{amsmath}

\usepackage{mathtools}
\usepackage[noabbrev,capitalize]{cleveref}
\usepackage{centernot}
\usepackage[shortlabels]{enumitem}
\usepackage{url}
\usepackage{tikz}
\usepackage{pict2e}
\usepackage{framed}

\usetikzlibrary{matrix}
\usetikzlibrary{positioning}
\usetikzlibrary{calc}
\usetikzlibrary{arrows}
\tikzset{> =stealth}
\usetikzlibrary{arrows.meta}
\tikzset{normalHead/.tip={Triangle[open,angle=60:4pt]},}
\tikzset{normalTail/.tip={Triangle[reversed,open,angle=60:4pt]},}

\newcommand{\addQEDstyle}[2]{\AtBeginEnvironment{#1}{\pushQED{\qed}\renewcommand{\qedsymbol}{#2}}\AtEndEnvironment{#1}{\popQED}}

\theoremstyle{plain}
\newtheorem{theorem}{Theorem}[section]
\newtheorem{lemma}[theorem]{Lemma}
\newtheorem{proposition}[theorem]{Proposition}
\newtheorem{corollary}[theorem]{Corollary}

\theoremstyle{definition}
\newtheorem{definition}[theorem]{Definition}

\newtheorem{example}[theorem]{Example}
\addQEDstyle{example}{$\triangle$}

\theoremstyle{remark}
\newtheorem{remark}[theorem]{Remark}

\renewcommand{\epsilon}{\varepsilon}
\renewcommand{\phi}{\varphi}

\renewcommand{\C}{\mathcal{C}}
\newcommand{\B}{\mathcal{B}}
\newcommand{\D}{\mathcal{D}}

\newcommand{\Nvar}{\mathcal{N}}
\renewcommand{\G}{\mathcal{G}}
\renewcommand{\H}{\mathcal{H}}
\newcommand{\op}{^\mathrm{op}}
\newcommand{\co}{^\mathrm{co}}
\newcommand{\coop}{^\mathrm{co\,op}}

\newcommand{\inv}{^{-1}}

\newcommand{\id}{\mathrm{id}}
\newcommand{\Id}{\mathrm{Id}}

\newcommand{\coker}{\mathrm{coker}}
\newcommand{\Ker}{\mathrm{Ker}}

\newcommand{\Cat}{\mathrm{Cat}}

\newcommand{\Hom}{\mathrm{Hom}}

\newcommand{\ExtTwo}{{\mathbb{E}\mathrm{xt}}} % The 2-category of extensions
\newcommand{\ExtOne}{\mathrm{Ext}} % The 1-category of extensions

\newcommand{\FLTop}{\mathrm{Top}_\mathrm{lex}}
\newcommand{\FLCat}{\mathrm{Cat}_\mathrm{lex}}

\newcommand{\Homop}{\Hom_\mathrm{op}}

\DeclareMathOperator{\Gl}{Gl}

\newcommand{\splitext}[6]{%
\tikz[baseline]{
\newdimen{\mylabelwidth}
\newdimen{\skipwidth}
\node[anchor=base] (A) {\hspace*{\dimexpr0.5pt-\pgfkeysvalueof{/pgf/inner xsep}}${#1}$};
\settowidth{\mylabelwidth}{\pgfinterruptpicture {$#2$} \endpgfinterruptpicture}
\pgfmathsetlength{\skipwidth}{max(\mylabelwidth,10pt)}
;\node[right] (B) at ([xshift=\skipwidth+12pt]A.east) {${#3}$};
\settowidth{\mylabelwidth}{\pgfinterruptpicture {$#4$} \endpgfinterruptpicture}
\settowidth{\skipwidth}{\pgfinterruptpicture {$#5$} \endpgfinterruptpicture}
\pgfmathsetlength{\skipwidth}{max(\skipwidth,\mylabelwidth,10pt)}
\node[right] (C) at ([xshift=\skipwidth+12pt]B.east) {${#6}$\hspace*{\dimexpr0.5pt-\pgfkeysvalueof{/pgf/inner xsep}}};
\draw[normalTail->] (A) to node [above] {${#2}$} (B);
\draw[transform canvas={yshift=0.5ex},-normalHead] (B) to node [above] {${#4}$} (C);
\draw[transform canvas={yshift=-0.5ex},->] (C) to node [below] {${#5}$} (B);
}}

\pgfdeclarelayer{bg}
\pgfdeclarelayer{over}
\pgfsetlayers{bg,main,over}

\tikzset{dot/.style={circle,draw=black,fill=black,minimum size=1mm,inner sep=0mm}}
\newcommand{\colorN}{red!20}
\newcommand{\colorGone}{green!20}
\newcommand{\colorGtwo}{yellow!30}
\newcommand{\colorH}{blue!20}

\makeatletter
\pgfkeys{%
  /tikz/on layer/.code={
    \pgfonlayer{#1}\begingroup
    \aftergroup\endpgfonlayer
    \aftergroup\endgroup
  },
  /tikz/node on layer/.code={
    \gdef\node@@on@layer{%
      \setbox\tikz@tempbox=\hbox\bgroup\pgfonlayer{#1}\unhbox\tikz@tempbox\endpgfonlayer\egroup}
    \aftergroup\node@on@layer
  },
  /tikz/end node on layer/.code={
    \endpgfonlayer\endgroup\endgroup
  }
}
\def\node@on@layer{\aftergroup\node@@on@layer}
\makeatother

\title{Artin glueings of toposes as adjoint split extensions}

\author[P. F. Faul]{Peter F. Faul}
\address{Department of Pure Mathematics and Mathematical Statistics\\ University of Cambridge}
\email{peter@faul.io}

\author[G. Manuell]{Graham Manuell}
\address{Centre for Mathematics, University of Coimbra, Coimbra, Portugal}
\email{graham@manuell.me}

\author[J. Siqueira]{Jos\'e Siqueira \thanks{This study was financed in part by the Coordenação de Aperfeiçoamento de Pessoal de Nível Superior - Brasil (CAPES), which supported the CAPES scholar Jos\'e Siqueira (process n$^\circ$ 88881.128278/2016-01). }}
\address{Department of Pure Mathematics and Mathematical Statistics\\ University of Cambridge}
\email{jose.siqueira@cantab.net}

\date{\today}

\subjclass[2010]{18B25, 18G50, 54B99}
\keywords{gluing, Artin-Wraith glueing, topoi, left-exact functor, weakly Schreier, Baer sum}

\begin{document}

\maketitle

\begin{abstract}
    Artin glueings of frames correspond to adjoint split extensions in the category of frames and finite-meet-preserving maps. We extend these ideas to the setting of toposes and show that Artin glueings of toposes correspond to a 2-categorical notion of adjoint split extension in the 2-category of toposes, finite-limit-preserving functors and natural transformations. A notion of morphism between these split extensions is introduced, which allows the category $\ExtOne(\H,\Nvar)$ to be constructed. We show that $\ExtOne(\H,\Nvar)$ is equivalent to $\Hom(\H,\Nvar)\op$, and moreover, that this can be extended to a 2-natural contravariant equivalence between the $\Hom$ 2-functor and a naturally defined $\ExtOne$ 2-functor.
\end{abstract}

\section{Introduction}

Artin glueings of toposes were introduced in \cite{sga4vol1} and provide a way to view a topos $\G$ as a combination of an open subtopos $\G_{\mathfrak{o}(U)}$ and its closed complement $\G_{\mathfrak{c}(U)}$. This situation may be described as the `internal' view, but we might instead look at it externally. Here we have that Artin glueings of two toposes $\H$ and $\Nvar$ correspond to solutions to the problem of which toposes $\G$ does $\H$ embed in as an open subtopos and $\Nvar$ as its closed complement. 

There is an analogy to be made with semidirect products of groups. We may either view a group as being generated in a natural way from two complemented subgroups (one of which is normal), or externally, view a semidirect product as a solution to the problem of how to embed groups $H$ and $N$ as complemented subobjects so that $N$ is normal. Of particular importance to us is that semidirect products correspond to split extensions of groups (up to isomorphism).

Artin glueings of Grothendieck toposes decategorify to the setting of frames and in this algebraic setting the analogy to semidirect products has been made precise. In \cite{faul2019artin} it was shown that Artin glueings correspond to certain split extensions in the category of frames with finite-limit-preserving maps. While these results were proved in the setting of frames, it is not hard to see that the arguments carry over to Heyting algebras. It is this view that we now extend back to the elementary topos setting.
We now recall the main results of \cite{faul2019artin}.

In the category of frames with finite-meet-preserving maps, there exist zero morphisms given by the constant `top' maps. This allows us to consider kernels and cokernels. Cokernels always exist and the cokernel of $f\colon N \to G$ is given by $e\colon G \to {\downarrow} f(0)$ where $e(g) = f(0) \wedge g$. This map has a right adjoint splitting $e_*$ sending $h$ to $h^{f(0)}$. Kernels do not always exist, but kernels of cokernels always do, and the kernel of $e\colon G \to {\downarrow} u$ is the inclusion of ${\uparrow} u \subseteq G$. The cokernel is readily seen to be the open sublocale corresponding to $u$ and the kernel the corresponding closed sublocale. This immediately gives that the split extensions whose splittings are adjoint to the cokernel correspond to Artin glueings.

A notion of protomodularity for categories equipped with a distinguished class of split extensions was first introduced in \cite{bourn2013schreier}.
In a similar way, in \cite{faul2019artin} (as well as the current paper) a distinguished role is played by those split extensions whose splittings are right adjoint to the cokernel.

With this correspondence established, the corresponding $\ExtOne$ functor was shown to be naturally isomorphic to the $\Hom$ functor. Each hom-set $\Hom(H,N)$ has an order structure and this order structure was shown to correspond contravariantly in $\ExtOne(H,N)$ to morphisms of split extensions. Finally, it was demonstrated how the meet operation in $\Hom(H,N)$ naturally induces a kind of `Baer sum' in $\ExtOne(H,N)$.

In this paper all of the above results find natural generalisation to the topos setting after we provide definitions for the analogous 2-categorical concepts.
(Also see the paper \cite{niefield2021extensions} by Niefield which appeared after we wrote this paper and discusses the relationship between (non-split) extensions and glueing in a very general context. The current paper uses split extensions and discusses the functorial nature of the construction for toposes in more detail.)

\section{Background}

\subsection{2-categorical preliminaries}

There are a number of 2-categories and 2-categorical constructions considered in this paper and so we provide a brief overview of these here. All the 2-categories we consider shall be strict. 

Briefly, a (strict) 2-category consists of objects, 1-morphisms between objects and 2-morphisms between 1-morphisms. Phrased another way, instead of hom-sets between objects as is the case with 1-categories, for any two objects $A$ and $B$ we have an associated hom-category $\Hom(A,B)$.

Both 1-morphisms and 2-morphisms may be composed under the right conditions. If $F\colon A \to B$ and $G\colon B \to C$ are 1-morphisms, then we may compose them to yield $GF\colon A \to C$. We usually represent this with juxtaposition, though if an expression is particularly complicated we may use $G \circ F$. As with natural transformations, there are two ways to compose 2-morphisms --- vertically and horizontally. We may compose $\alpha \colon F \to G$ and $\beta \colon G \to H$ vertically to give $\beta\alpha \colon F \to H$, sometimes written $\beta \circ \alpha$. Orthogonally, if we have $F_2F_1 \colon A \to C$, $G_2G_1 \colon A \to C$, $\alpha\colon F_1 \to G_1$ and $\beta\colon F_2 \to G_2$, then we may compose $\alpha$ and $\beta$ horizontally to form $\beta \ast \alpha \colon F_2F_1 \to G_2G_1$. Vertical and horizontal composition are related by the so-called interchange law. For each object $A$ there exists an identity 1-morphism $\id_A$ and for each 1-morphism $F$ there exists an identity 2-morphism $\id_F$.

Just as one can reverse the arrows of a category $\mathcal{B}$ to give $\mathcal{B}\op$, one can reverse the 1-morphisms of a 2-category $\C$ to give $\C\op$. It is also possible to reverse the directions of the 2-morphisms yielding $\C\co$ and when both the 1-morphisms and 2-morphisms are reversed we obtain $\C\coop$.

In this paper we will make extensive use of string diagrams. For an introduction to string diagrams for 2-categories see \cite{marsden2014category}. We will use the convention that vertical composition is read from bottom to top and horizontal composition runs diagrammatically from left to right.

We consider 2-functors between 2-categories defined as follows. (We follow the convention that 2-functors are not necessarily strict.)

\begin{definition}
  A \emph{2-functor} $\mathcal{F}$ between 2-categories $\C$ and $\D$ consists of a function $\mathcal{F}$ sending objects $C$ in $\C$ to objects $\mathcal{F}(C)$ in $\D$ and for each pair of objects $C_1$ and $C_2$ in $\C$ a functor $\mathcal{F}_{C_1,C_2} \colon \Hom(C_1,C_2) \to \Hom(\mathcal{F}(C_1),\mathcal{F}(C_2))$, for which we use the same name.
  Additionally, for each pair of composable 1-morphisms $(F,G)$ we have an invertible 2-morphism $\omega_{G,F} \colon \mathcal{F}(G) \circ \mathcal{F}(F) \to \mathcal{F}(G \circ F)$ called the compositor, and for each object $C$ in $\C$ we have an invertible 2-morphism $\kappa_A \colon \Id_{\mathcal{F}(A)} \to \mathcal{F}(\Id_A)$ called the unitor.
  This data satisfies the following constraints.
  \begin{enumerate}
    \item Let $\alpha\colon F_1 \to F_2$ and $\beta\colon G_1 \to G_2$ be horizontally composable 2-morphisms. Then the compositors must satisfy the naturality condition $\omega_{G_2,F_2} \left( \mathcal{F}(\beta) \ast \mathcal{F}(\alpha) \right) = \mathcal{F}(\beta \ast \alpha)\,\omega_{G_1,F_1}$.
    
    \item The compositors must be associative in the sense that if $F \colon X \to Y$, $G\colon Y \to Z$ and $H\colon Z \to W$ are 1-morphisms, then $\omega_{H,GF} (\id_{\mathcal{F}(H)} \ast \omega_{G,F}) = \omega_{HG,F} (\omega_{H,G} \ast \id_{\mathcal{F}(F)})$.
    
    \item If $F \colon X \to Y$ is a 1-morphism, then we have the unit axiom $\omega_{F,\Id_X} (\id_{\mathcal{F}(F)} \ast \kappa_X) = \id_{\mathcal{F}(F)} =  \omega_{\Id_Y,F} (\kappa_Y \ast \id_{\mathcal{F}(F)})$.
  \end{enumerate}
\end{definition}

There is a notion of 2-natural transformation between 2-functors defined as follows. 

\begin{definition}
Let $(\mathcal{F}_1,\omega_1,\kappa_1), (\mathcal{F}_2,\omega_2,\kappa_2)\colon \mathcal{X} \to \mathcal{Y}$ be two 2-functors. A 2-natural transformation $\rho \colon \mathcal{F}_1 \to \mathcal{F}_2$ is given by two families:

\begin{enumerate}
    \item A 1-morphism $\rho_X \colon \mathcal{F}_1(X) \to \mathcal{F}_2(X)$ for each object $X$ in $\mathcal{X}$.
    \item An invertible 2-morphism $\rho_F \colon \mathcal{F}_2(F) \rho_X \to \rho_Y \mathcal{F}_1(F)$ for each 1-morphism $F\colon X \to Y$ in $\mathcal{X}$.
\end{enumerate}

They must satisfy the following coherence conditions. First if $F\colon X \to Y$ and $G\colon Y \to Z$ are 1-morphisms in $\mathcal{X}$, then $\rho$ must respect composition, so that the following diagram commutes.
\begin{center}
  \begin{tikzpicture}[node distance=3.0cm, auto]
    \node (C) {$\rho_Z \circ \mathcal{F}_1(G) \circ \mathcal{F}_1(F)$};
    \node (A) [below of=C] {$\rho_Z \circ \mathcal{F}_1(GF)$};
    \node (D) [left=2.2cm of C] {$\mathcal{F}_2(G) \circ \rho_Y \circ \mathcal{F}_1(F)$};
    \node (E) [left=2.2cm of D] {$\mathcal{F}_2(G) \circ \mathcal{F}_2(F) \circ \rho_X$};
    \node (B) [below of=E] {$\mathcal{F}_2(GF) \circ \rho_X$};
    \draw[<-] (A) to node {$\rho_{GF}$} (B);
    \draw[->] (C) to node {$\rho_Z \kappa_1$} (A);
    \draw[<-] (D) to node [swap] {$\mathcal{F}_2(G) \rho_F$} (E);
    \draw[<-] (C) to node [swap] {$\rho_G \mathcal{F}_1(F)$} (D);
    \draw[->] (E) to node [swap] {$\kappa_2 \rho_X$} (B);
  \end{tikzpicture}
\end{center}

Next for each object $X \in \mathcal{X}$, $\rho$ must respect the identity $\Id_X$. For this we need the following diagram to commute.

\begin{center}
  \begin{tikzpicture}[node distance=3.1cm, auto]
    \node (A) {$\rho_X \circ \Id_{\mathcal{F}_1(X)}$};
    \node (B) [below of=A] {$\rho_X \circ \mathcal{F}_1(\Id_X)$};
    \node (C) [left=1.3cm of A] {$\Id_{\mathcal{F}_2(X)} \circ \rho_X$};
    \node (D) [below of=C] {$\mathcal{F}_2(\Id_X) \circ \rho_X$};
    \draw[->] (A) to node {$\rho_X \omega_1$} (B);
    \draw[double equal sign distance] (A) to node {} (C);
    \draw[<-] (B) to node {$\rho_{\Id_X}$} (D);
    \draw[->] (C) to node [swap] {$\omega_2 \rho_X$} (D);
  \end{tikzpicture}
\end{center}

Finally, they must satisfy the following `naturality' condition for $F,F'\colon X \to Y$ and $\alpha\colon F \to F'$.
\begin{center}
  \begin{tikzpicture}[node distance=3.1cm, auto]
    \node (A) {$\rho_Y \circ \mathcal{F}_1(F)$};
    \node (B) [below of=A] {$\rho_Y \circ \mathcal{F}_1(F')$};
    \node (C) [left=1.3cm of A] {$\mathcal{F}_2(F) \circ \rho_X$};
    \node (D) [below of=C] {$\mathcal{F}_2(F') \circ \rho_X$};
    \draw[->] (A) to node {$\rho_Y \mathcal{F}_1(\alpha)$} (B);
    \draw[<-] (A) to node [swap] {$\rho_F$} (C);
    \draw[<-] (B) to node {$\rho_{F'}$} (D);
    \draw[->] (C) to node [swap] {$\mathcal{F}_2(\alpha) \rho_X$} (D);
  \end{tikzpicture}
\end{center}

When each component $\rho_X$ is an equivalence, we call $\rho$ a \emph{2-natural equivalence}.
\end{definition}

One 2-functor of note is the 2-functor $\mathrm{Op}\colon \Cat\co \to \Cat$ which sends a category $\C$ to its opposite category $\C\op$. We will use this 2-functor in \cref{sec:ext_functor} to help compare two 2-functors of different variances.

Limits and colimits have 2-categorical analogues, which will be used extensively throughout this paper. A more complete introduction to these concepts can be found in \cite{lack2cats}. In particular, we will make use of 2-pullbacks and 2-pushouts, as well as comma and cocomma objects, which we describe concretely below.

\begin{definition}
Given two 1-morphisms $F\colon \B \to \D$ and $G\colon \C \to \D$ their \emph{comma object} is shown in the following diagram
\begin{center}
  \begin{tikzpicture}[node distance=3.5cm, auto]
    \node (A) {$\mathcal{P}$};
    \node (B) [below of=A] {$\B$};
    \node (C) [right of=A] {$\C$};
    \node (D) [right of=B] {$\D$};
    \draw[->] (A) to node [swap] {$P_{G}$} (B);
    \draw[->] (A) to node {$P_{F}$} (C);
    \draw[->] (B) to node [swap] {$F$} (D);
    \draw[->] (C) to node {$G$} (D);
    \begin{scope}[shift=({A})]
        \draw[dashed] +(0.25,-0.75) -- +(0.75,-0.75) -- +(0.75,-0.25);
    \end{scope}
    \draw[double distance=4pt,-implies] ($(B)!0.38!(C)$) to node {$\phi$} ($(B)!0.62!(C)$);
 \end{tikzpicture}
\end{center}
and satisfies the following two conditions.
\begin{enumerate}
    \item Let $T\colon \mathcal{X} \to \B$ and $S\colon \mathcal{X} \to \C$ be 1-morphisms and let $\psi\colon FT \to GS$ be a 2-morphism. Then there exists a 1-morphism $H\colon \mathcal{X} \to \mathcal{P}$ and invertible 2-morphisms $\nu\colon P_GH \to T$ and $\mu\colon P_FH \to S$ satisfying that $G \mu \circ \phi H \circ F \nu\inv = \psi$.
    \begin{center}
  \begin{tikzpicture}[node distance=3.5cm, auto]
    \node (A) {$\mathcal{P}$};
    \node (B) [below of=A] {$\B$};
    \node (C) [right of=A] {$\C$};
    \node (D) [right of=B] {$\D$};
    \draw[->] (A) to node [swap] {$P_G$} (B);
    \draw[->] (A) to node {$P_F$} (C);
    \draw[->] (B) to node {$F$} (D);
    \draw[->] (C) to node {$G$} (D);
    
    \begin{scope}[shift=({A})]
        \draw[dashed] +(0.25,-0.75) -- +(0.75,-0.75) -- +(0.75,-0.25);
    \end{scope}
    
    \draw[double distance=4pt,-implies] ($(B)!0.38!(C)$) to node {$\phi$} ($(B)!0.62!(C)$);
    
    \node (X) [above left=1.2cm and 1.2cm of A.center] {$\mathcal{X}$};
    \draw[out=-90,->] (X) to node [swap] {$T$} node [anchor=center] (T') {} (B);
    \draw[out=0,->] (X) to node {$S$} node [anchor=center] (T) {} (C);
    \draw[dashed, ->, pos=0.55] (X) to node {$H$} (A);
    
    \draw[double distance=4pt,-implies,shorten <=2.5pt,shorten >=1pt] (A) to node [swap] {$\mu$} (T);
    \draw[double distance=4pt,-implies,shorten <=2.5pt,shorten >=1pt] (T') to node [pos=0.75] {$\nu\inv$} (A);
  \end{tikzpicture}
\end{center}
\item If $H,K\colon \mathcal{X} \to \mathcal{P}$ are 1-morphisms and $\alpha\colon P_GH \to P_GK$ and $\beta\colon P_FH \to P_FK$ are 2-morphisms satisfying that $\phi K \circ F \alpha = G \beta \circ \phi H$, then there exists a unique 2-morphism $\gamma\colon H \to K$ such that $P_G\gamma = \alpha$ and $P_F\gamma = \beta$.
\end{enumerate}

\medskip
A 2-pullback is defined similarly, except both $\phi$ and $\psi$ are required to be invertible, and is represented as follows.
\begin{center}
  \begin{tikzpicture}[node distance=3.25cm, auto]
    \node (A) {$\mathcal{P}$};
    \node (B) [below of=A] {$\B$};
    \node (C) [right of=A] {$\C$};
    \node (D) [right of=B] {$\D$};
    \draw[->] (A) to node [swap] {$P_{G}$} (B);
    \draw[->] (A) to node {$P_{F}$} (C);
    \draw[->] (B) to node [swap] {$F$} (D);
    \draw[->] (C) to node {$G$} (D);
    \begin{scope}[shift=({A})]
        \draw +(0.25,-0.75) -- +(0.75,-0.75) -- +(0.75,-0.25);
    \end{scope}
    \draw[double distance=4pt,-implies] ($(B)!0.38!(C)$) to node {$\phi$} ($(B)!0.62!(C)$);
 \end{tikzpicture}
\end{center}

Cocomma objects and 2-pushouts may be defined dually.
\end{definition}

We now recall the definition of a fibration of categories.
\begin{definition}
  Let $F\colon \mathcal{X} \to \mathcal{Y}$ be a functor. A morphism $\overline{f}\colon \overline{A} \to \overline{B}$ in $\mathcal{X}$ is \emph{cartesian} with respect to $F$ if for any $\overline{g}\colon \overline{C} \to \overline{B}$ in $\mathcal{X}$ and $h\colon F(\overline{C}) \to F(\overline{B})$ in $\mathcal{Y}$ with $F(\overline{g}) = F(\overline{f})h$, there exists a unique map $\overline{h}\colon \overline{C} \to \overline{A}$ with $F(\overline{h}) = h$ and $\overline{f}\,\overline{h} = \overline{g}$.
  We say $F\colon \mathcal{X} \to \mathcal{Y}$ is a \emph{(Street) fibration} in $\Cat$ if for any morphism $f\colon A \to F(\overline{B})$ in $\mathcal{Y}$ there exists a cartesian lifting $\overline{f}\colon \overline{A} \to \overline{B}$ and an isomorphism $j\colon F(\overline{A}) \cong A$ with $F(\overline{f}) = f j$.
\end{definition}

In fact, we will also need the notion of a fibration in other 2-categories, such as the 2-category $\FLCat$ of finitely-complete categories and finite-limit-preserving functors. The general definitions of fibrations, morphisms of fibrations and 2-morphisms of fibrations can be found, for example, in \cite[Definitions 3.4.3--3.4.5]{buckley2014fibred}. However, it is not hard to see that the fibrations in $\FLCat$ are simply the finite-limit-preserving functors which are fibrations in $\Cat$.

\subsection{Elementary toposes and Artin glueings}

By a \emph{topos} we mean an \emph{elementary topos} --- that is, a cartesian-closed category admitting finite limits and containing a subobject classifier. The usual 2-category of toposes has 1-morphisms given by \emph{geometric morphisms} and 2-morphisms given by natural transformations. 

For Grothendieck toposes, the subobjects of the terminal object may be imbued with the structure of a frame. Moreover, a geometric morphism between two Grothendieck toposes induces a locale homomorphism between their locales of subterminals. This induces a functor from the category of Grothendieck toposes into the category of locales, which is in fact a reflector.

A \emph{subtopos} of a topos $\mathcal{E}$ is a fully faithful geometric morphism $i\colon \mathcal{S} \hookrightarrow \mathcal{E}$. Subtoposes when acted upon by the localic reflection may sometimes be sent to open or closed sublocales. Those sent to open sublocales we call open subtoposes and those sent to closed sublocales we call closed subtoposes. Since open (or closed) sublocales correspond to elements of the frame, it follows that any open (or closed) subtopos corresponds to a particular subterminal $U$.
This suggests a similar notion of open/closed subtopos corresponding to a particular subterminal $U$ even in the elementary setting.

The open subtopos $\mathcal{E}_{\mathfrak{o}(U)}$ corresponding to a subterminal $U$ has a reflector given by the exponential functor $(-)^U$. It is not hard to see that this topos is equivalent to the slice topos $\mathcal{E} \slash U$, which in turn can be thought as the full subcategory of the objects in $\mathcal{E}$ admitting a map into $U$. From this point of view, the reflector $E\colon \mathcal{E} \to \mathcal{E}/U$ maps an object $X$ to the product $X \times U$. We denote its right adjoint by $E_* = (-)^U$ and write $\theta$ and $\epsilon$ for the unit and counit respectively.
Note that $E_*E(G) = (G \times U)^U \cong G^U$.
In addition to a right adjoint, $E$ also has a left adjoint $E_!$, which is simply the inclusion of $\mathcal{E}/U$ into $\mathcal{E}$.

The closed subtopos $\mathcal{E}_{\mathfrak{c}(U)}$ has reflector $K^*\colon \mathcal{E} \to \mathcal{E}_{\mathfrak{c}(U)}$ given on objects by the following pushout.
\begin{center}
      \begin{tikzpicture}[node distance=3.0cm, auto]
        \node (A) {$G\times U$};
        \node (B) [below of=A] {$U$};
        \node (C) [right of=A] {$G$};
        \node (D) [right of=B] {$K^*(G)$};
        \draw[->] (A) to node [swap] {$\pi_U$} (B);
        \draw[->] (A) to node {$\pi_G$} (C);
        \draw[->] (B) to node [swap] {$p_1^G$} (D);
        \draw[->] (C) to node {$p_2^G$} (D);
        \begin{scope}[shift=({D})]
            \draw +(-0.3,0.7) -- +(-0.7,0.7) -- +(-0.7,0.3);
            \end{scope}
      \end{tikzpicture}
\end{center}

On morphisms $f\colon G \to G'$ in $\mathcal{E}$, $K^*(f)$ is given by the universal property of the pushout in the following diagram.
Here the left, front and top faces commute and so a diagram chase determines that $p_1^{G'}f$ and $p_2^{G'} \id_U$ indeed form a cocone.
\begin{center}
  \begin{tikzpicture}[node distance=4.0cm, auto, on grid,
    cross line/.style={preaction={draw=white, -,line width=6pt}}
    ]
    
    \node (A) {$G' \times U$};
    \node [right of=A] (B) {$U$};
    \node [below of=A] (C) {$G'$};
    \node [right of=C] (D) {$K^*(G')$};
    
    \begin{scope}[shift=({D})]
        \draw [draw=black] +(-0.25,0.75) coordinate (X) -- +(-0.75,0.75) coordinate (Y) -- +(-0.75,0.25) coordinate (Z);
    \end{scope}
    
    \node [above left=1.4cm and 1.9cm of A.center] (A1) {$G \times U$};
    \node [right of=A1] (B1) {$U$};
    \node [below of=A1] (C1) {$G$};
    \node [right of=C1] (D1) {$K^*(G)$};
    
    \begin{scope}[shift=({D1})]
        \draw [draw=black!70] +(-0.22,0.78) coordinate (X1) -- +(-0.78,0.78) coordinate (Y1) -- +(-0.78,0.22) coordinate (Z1);
    \end{scope}
    
    \draw[->] (A1) to node {$\pi_U$} (B1);
    \draw[->] (A1) to node [swap] {$\pi_G$} (C1);
    \draw[->] (B1) to node [pos=0.65] {} (D1);
    \draw[->] (C1) to node [node on layer=over,fill=white,inner sep=2pt,outer sep=0pt,yshift=1pt] (K) {} (D1);
    
    \draw[cross line, ->] (A) to node [pos=0.5] {$\pi_U$} (B);
    \draw[cross line, ->] (A) to node [swap, pos=0.4] {$\pi_{G'}$} (C);
    \draw[cross line, ->] (B) to node {$p_2^{G'}$} (D);
    \draw[cross line, ->] (C) to node [swap] {$p_1^{G'}$} (D);
    
    \begin{scope}
    \begin{pgfonlayer}{over}
    \clip (K.south west) rectangle (K.north east);
    \draw[->, opacity=0.7] (A) to (C);
    \end{pgfonlayer}
    \end{scope}
    
    \draw[->] (A1) to node [yshift=-2pt,pos=0.6] {$f \times \id$} (A);
    \draw[double equal sign distance] (B1) to (B);
    \draw[->] (C1) to node [swap, yshift=2pt] {$f$} (C);
    \draw[dashed, ->] (D1) to node [swap, pos=0.2, xshift=5pt] {$K^*(f)$} (D);
  \end{tikzpicture}
\end{center}

We denote the right adjoint of $K^*$ by $K$ and write $\zeta$ and $\delta$ for the unit and counit respectively.

As expected, $\mathcal{E}_{\mathfrak{o}(U)}$ and $\mathcal{E}_{\mathfrak{c}(U)}$ are complemented subobjects. Given toposes $\H$ and $\Nvar$ we can ask for which toposes $\G$ may $\H$ be embedded as an open subtopos and $\Nvar$ its closed complement. This is solved completely by the Artin glueing construction. For any finite-limit-preserving functor $F\colon \H \to \Nvar$ we may construct the category $\Gl(F)$ whose objects are triples $(N,H,\ell)$ in which $N \in \Nvar$, $H \in \H$ and $\ell\colon N \to F(H)$ and whose morphisms are pairs $(f,g)$ making the following diagram commute.

\begin{center}
  \begin{tikzpicture}[node distance=2.75cm, auto]
    \node (A) {$N$};
    \node (B) [below of=A] {$N'$};
    \node (C) [right of=A] {$F(H)$};
    \node (D) [right of=B] {$F(H')$};
    \draw[->] (A) to node [swap] {$f$} (B);
    \draw[->] (A) to node {$\ell$} (C);
    \draw[->] (B) to node [swap] {$\ell'$} (D);
    \draw[->] (C) to node {$F(g)$} (D);
  \end{tikzpicture}
\end{center}

The category $\Gl(F)$ is a topos, and moreover, the obvious projections $\pi_1 \colon \Gl(F) \to \Nvar$ and $\pi_2\colon \Gl(F) \to \H$ are finite-limit preserving. The projection $\pi_2$ has a right adjoint $\pi_{2*}$ sending objects $H$ to $(F(H),H,\id_{F(H)})$ and morphisms $f$ to $(F(f),f)$. This map $\pi_{2*}$ is a geometric morphism and, in particular, an open subtopos inclusion. Similarly, $\pi_1$ has a right adjoint sending objects $N$ to $(N,1,!)$ and morphisms $f$ to $(f,!)$ where $!\colon N \to 1$ is the unique map to the terminal object. This is itself a geometric morphism, and indeed, a closed subtopos inclusion.

Remarkably, Artin glueings may be viewed as both comma and cocomma objects in the category of toposes with finite-limit-preserving functors. We provide a proof of the latter in \cref{sec:extension_category}.

One sees immediately that $\pi_1\pi_{2*} = F$. This suggests a way to view any open or closed subtopos as corresponding to one in glueing form. If $K\colon \G_{\mathfrak{c}(U)} \to \G$ and $E_*\colon \G/U \to \G$ are respectively the inclusions of open and closed subtoposes, then there is a natural sense in which these maps correspond to $\pi_{1*}\colon \G_{\mathfrak{c}(U)} \to \Gl(K^*E_*)$ and $\pi_{2*}\colon \G/U \to \Gl(K^*E_*)$ respectively. This fact is well known, though a new proof will be provided in \cref{sec:adjoint_extensions}.

We now note that the maps $\pi_1\colon \Gl(F) \to \Nvar$ and $\pi_2\colon \Gl(F) \to \H$ are fibrations. By the above argument, these results apply equally to the inverse image maps of open and closed subtoposes. This is likely well known, though we were unable to find explicit mention of this in the literature.

\begin{proposition}\label{prp:pi2fib}
    Let $F\colon \H \to \Nvar$ be a finite-limit-preserving map between toposes. Then the projection $\pi_2\colon \Gl(F) \to \H$ is a fibration.
\end{proposition}

\begin{proof}
We must show that if $f\colon H' \to H$ is a morphism in $\H$, then for every object $(N,H,\ell)$ in $\Gl(F)$ there exists a cartesian lifting. This lifting is given by $(P_f,f)\colon (\overline{N},H',\overline{\ell}) \to (N,H,\ell)$ where $\overline{N}$, $\overline{\ell}$ and $P_f$ are defined by the following pullback.

\begin{center}
 \begin{tikzpicture}[node distance=3.0cm, auto]
    \node (A) {$\overline{N}$};
    \node (B) [right of=A] {$N$};
    \node (C) [below of=A] {$F(H')$};
    \node (D) [below of=B] {$F(H)$};
    \draw[->] (A) to node {$P_f$} (B);
    \draw[->] (A) to node [swap] {$\overline{\ell}$} (C);
    \draw[->] (B) to node {$\ell$} (D);
    \draw[->] (C) to node [swap] {$F(f)$} (D);
    \begin{scope}[shift=({A})]
        \draw +(0.25,-0.75) -- +(0.75,-0.75) -- +(0.75,-0.25);
    \end{scope}
 \end{tikzpicture}
\end{center}

The cartesian property of $(P_f,f)$ follows from the universal property of the pullback.
\end{proof}

\begin{proposition}\label{prp:pi1fib}
    Let $F\colon \H \to \Nvar$ be a finite-limit-preserving map between toposes. Then the projection $\pi_1\colon \Gl(F) \to \Nvar$ is a fibration.
\end{proposition}
\begin{proof}
    We must show that if $f\colon N' \to N$ is a morphism in $\Nvar$ then for every object $(N,H,\ell)$ in $\Gl(F)$ there exists a cartesian lifting. This map is given by $(f,\id_H)\colon (N',H,\ell f) \to (N,H,\ell)$.
    
    To see that this map satisfies the universal property, suppose that we have a morphism $(g_1,g_2)\colon (A,B,k) \to (N,H,\ell)$ which is mapped by $\pi_1$ to $fh$.
    We must show there is a unique map $\overline{h}$ such that $(g_1,g_2) = (f,\id_H) \overline{h}$ and $\pi_1(\overline{h}) = h$. These constraints imply that $\overline{h} = (h,g_2)$.
    
    To see that $(h,g_2)$ is a morphism in $\Gl(F)$, we consider the following diagram.
    The left-hand square commutes as $\pi_1(g_1,g_2) = fh$ and the right-hand square commutes since $(g_1,g_2)$ is a morphism.
    \begin{center}
      \begin{tikzpicture}[node distance=2.75cm, auto]
        \node (A) {$A$};
        \node (B) [right of=A] {$A$};
        \node (C) [right of=B] {$F(B)$};
        \node (D) [below of=A] {$N$};
        \node (E) [right of=D] {$N'$};
        \node (F) [right of=E] {$F(H)$};
        \draw[double equal sign distance] (A) to (B);
        \draw[->] (B) to node {$k$} (C);
        \draw[->] (D) to node [swap] {$f$} (E);
        \draw[->] (E) to node [swap] {$\ell$} (F);
        \draw[->] (B) to node [swap] {$g_1$} (E);
        \draw[->] (A) to node [swap] {$h$} (D);
        \draw[->] (C) to node {$F(g_2)$} (F);
     \end{tikzpicture}
    \end{center}
    Finally, we immediately see that $(f,\id_H)(h,g_2) = (g_1,g_2)$, as required.
\end{proof}

\section{Adjoint extensions}

In generalising the frame results to the topos setting, it is clear that the appropriate 2-category to consider is $\FLTop$, the 2-category of toposes, finite-limit-preserving functors and natural transformations. (This is the horizontal bicategory of the double category of toposes considered in \cite{niefield2012glueing}.) For convenience, we will assume that $1$ always refers to a distinguished terminal object in a topos, and $0$ a distinguished initial object.

We will now introduce the necessary concepts in order to discuss extensions of toposes and show how Artin glueings can be viewed as adjoint extensions. In particular, the definition of extension will require notions of kernel and cokernel.

\subsection{Zero morphisms, kernels and cokernels}

The definition of extensions requires a notion of zero morphisms. Let us now define these in the 2-categorical context.

\begin{definition}\label{def:pointed_2_cat}
A \emph{pointed 2-category} is a 2-category equipped with a class $\mathcal{Z}$ of 1-morphisms (called \emph{zero morphisms)} satisfying the following conditions:
\begin{itemize}
    \item $\mathcal{Z}$ contains an object of each hom-category,
    \item $\mathcal{Z}$ is an ideal with respect to composition (as in \cite{ehresmann}) --- that is, $g \in \mathcal{Z} \implies fgh \in \mathcal{Z}$,
    \item $\mathcal{Z}$ is closed under 2-isomorphism in the sense that if $f \in \mathcal{Z}$ and $f' \cong f$ then $f' \in \mathcal{Z}$,
    \item for any parallel pair $f_1,f_2$ of morphisms in $\mathcal{Z}$, there is a unique 2-morphism $\xi\colon f_1 \to f_2$.
\end{itemize}
\end{definition}

\begin{definition}
A \emph{zero object} in a 2-category is an object which is both 2-initial and 2-terminal.
\end{definition}

\begin{lemma}
A 2-terminal or 2-initial object in a pointed 2-category is always a zero object. Furthermore, any 2-category with a zero object has a unique pointed 2-category structure where the zero morphisms are those 1-morphisms which factor through the zero object up to 2-isomorphism. 
\end{lemma}

\begin{remark}
 \Cref{def:pointed_2_cat} is a categorification of pointed-set-enriched categories.
 Pointed categories are often instead defined as those having a zero object. The previous lemma shows that our definition agrees with a definition in terms of zero objects when the category has a terminal or initial object.
\end{remark}

\begin{definition}
    The \emph{2-cokernel} of a morphism $f\colon A \to B$ in a pointed 2-category is an object $C$ equipped with a morphism $c\colon B \to C$ such that $cf$ is a zero morphism and which is the universal such in the following sense.
    \begin{enumerate}
        \item If $t \colon B \to X$ is such that $tf$ is a zero morphism, then there exists a morphism $h \colon C \to X$ such that $hc$ is isomorphic to $t$.
        \item Given $h,h'\colon C \to X$ and $\alpha \colon hc \to h'c$, there is a unique $\gamma\colon h \to h'$ such that $\gamma c = \alpha$.
    \end{enumerate}
    
    The \emph{2-kernel} of a morphism in a pointed 2-category $\C$ is simply the 2-cokernel in $\C\op$.
\end{definition}
A similar notion was defined for groupoid-enriched categories in \cite{abelian2d}.

Note that these may also be defined in terms of 2-pushouts or 2-coequalisers involving the zero morphism (in the same way kernels and cokernels can be defined in the 1-categorical setting). Although there would usually be coherence conditions that need to be satisfied, the uniqueness of natural isomorphisms between zero morphisms eliminates all of them in the case of 2-cokernels.

\begin{remark}
Note that the condition (2) for 2-kernels is simply the statement that the 2-kernel map is a fully faithful 1-morphism.
Moreover, since an adjoint of a fully faithful morphism is fully faithful in the opposite 2-category,
we have that when a putative 2-cokernel $c \colon B \to C$ has a (left or right) adjoint $d$, then condition (2) for 2-cokernels is equivalent to $d$ being fully faithful.
\end{remark}

We can now consider how these concepts behave in our case of interest.
Note that $\FLTop$ has a zero object, the trivial topos. Then zero morphisms in $\FLTop$ are precisely those functors which send every object to a terminal object.

In $\FLTop$, 2-cokernels of morphisms $F\colon \Nvar \to \G$ always exist and are given by the open subtopos corresponding to $F(0)$.

\begin{proposition}\label{prop:cokernelsexist}
    The 2-cokernel of $F\colon \Nvar \to \G$ is given by $E\colon \G \to \G/F(0)$ sending objects $G$ to $G \times F(0)$ and morphisms $f\colon G \to G'$ to $(f,id_{F(0)}) \colon G\times F(0) \to G' \times F(0)$.
\end{proposition}

\begin{proof}
    We know that $E$ lies in $\FLTop$ and so we begin by showing that that $EF$ is a zero morphism. The terminal object in $\G/F(0)$ is $F(0)$ and so consider the following calculation.
    \begin{align*}
        EF(N)   &= F(N) \times F(0) \\
                &\cong F(N \times 0) \\
                &\cong F(0).
    \end{align*}
    
    Next suppose that $T\colon \G \to \mathcal{X}$ is such that $TF$ is a zero morphism. We claim that $TE_*\colon \H \to \mathcal{X}$ when composed with $E$ is naturally isomorphic to $T$. 
     
    Observe that
    \begin{align*}
        T(G)        &\cong T(G) \times 1 \\
                    &\cong T(G) \times TF(0) \\
                    &\cong T(G \times F(0)) \\
                    &\cong T(E_!E(G))
    \end{align*}
    where each isomorphism is natural in $G$.
    Hence, $T \cong TE_!E \cong (TE_!E)E_*E \cong TE_*E$ where the central isomorphism comes from $E\theta\colon E \xrightarrow{\sim} EE_*E$.
    
    The final condition for the 2-cokernel holds immediately because $E$ has a full and faithful adjoint.
\end{proof}

Unfortunately, 2-kernels do not always exist in $\FLTop$. However, they do exist in the larger 2-category $\FLCat$ of finitely-complete categories and finite-limit-preserving functors.

\begin{proposition}\label{prp:kernelfullsubcat}
Let $F \colon \G \to \H$ be a morphism in $\FLCat$. The 2-kernel of $F$, which we write as $\Ker(F)$, is given by the inclusion into $\G$ of the full subcategory of objects sent by $F$ to a terminal object.
\end{proposition}

\begin{proof}
Since $F$ preserves finite limits and sends each object in $\Ker(F)$ to a terminal object, it is clear that $\Ker(F)$ is closed under finite limits. Naturally, the inclusion is a finite-limit-preserving functor.

It is clear that $FK$ is a zero morphism. We must check that if $T\colon \mathcal{X} \to \G$ is such that $FT$ is a zero morphism, then it factors through $\Ker(F)$. Note that since $FT$ is a zero morphism, all objects (and morphisms) in its image lie in $\Ker(F)$. Thus, it is easy to see that $T$ factors through $\Ker(F)$.
The uniqueness condition of the universal property is immediate, as the inclusion of $\Ker(F)$ is full and faithful.
\end{proof}

We will only be concerned with 2-kernels of 2-cokernels. The following proposition shows that these do always exist in $\FLTop$.

\begin{proposition}\label{prop:kernelofcokernel}
    Let $U$ be a subterminal of a topos $\G$ and consider $E\colon \G \to \G/U$ defined as in \cref{prop:cokernelsexist}. Then the kernel of $E$ is given by $K \colon \G_{\mathfrak{c}(U)} \hookrightarrow \G$, the inclusion of the closed subtopos corresponding to $U$.
\end{proposition}
\begin{proof}
    Since $\FLTop$ is a full sub-2-category of $\FLCat$,
    it suffices to show that the closed subtopos $\G_{\mathfrak{c}(U)}$ is equivalent to $\Ker(E)$, the full subcategory of objects sent by $E$ to a terminal object.
    
    The reflector $K^*\colon \G \to \G_{\mathfrak{c}(U)}$ sends an object $G$ to the following pushout. 
    \begin{center}
      \begin{tikzpicture}[node distance=3.2cm, auto]
        \node (A) {$G \times U$};
        \node (B) [below of=A] {$G$};
        \node (C) [right of=A] {$U$};
        \node (D) [right of=B] {$K^*(G)$};
        \draw[->] (A) to node [swap] {$\pi_G$} (B);
        \draw[->] (A) to node {$\pi_U$} (C);
        \draw[->] (B) to node {} (D);
        \draw[->] (C) to node [swap] {} (D);
        \begin{scope}[shift=({D})]
            \draw +(-0.3,0.7) -- +(-0.7,0.7) -- +(-0.7,0.3);
        \end{scope}
      \end{tikzpicture}
    \end{center}
    First we show that $K^*(G)$ lies in $\Ker(E)$.
    We know that $E$ preserves colimits and so we obtain the following pushout in $\G/U$.
    
    \begin{center}
      \begin{tikzpicture}[node distance=3.2cm, auto]
        \node (A) {$G \times U$};
        \node (B) [below of=A] {$G \times U$};
        \node (C) [right of=A] {$U$};
        \node (D) [right of=B] {$EK^*(G)$};
        \draw[->] (A) to node [swap] {$\id_{G \times U}$} (B);
        \draw[->] (A) to node {$!$} (C);
        \draw[->] (B) to node {} (D);
        \draw[->] (C) to node {$p$} (D);
        \begin{scope}[shift=({D})]
            \draw +(-0.3,0.7) -- +(-0.7,0.7) -- +(-0.7,0.3);
        \end{scope}
      \end{tikzpicture}
    \end{center}
    But note that $p$ is an isomorphism, since it is the pushout of an identity morphism and thus $EK^*(G) \cong U$ and $K^*(G)$ lies in $\Ker(E)$.
    
    Finally, we must show that $K^*$ fixes the objects of $\Ker(E)$. First observe that $U$ is the initial object in $\Ker(E)$, since if $X$ is an object in $\Ker(E)$ then we have $\Hom_\G(U,X) = \Hom_\G(E_!(U),X) \cong \Hom_{\G/U}(U,E(X)) \cong \Hom_{\G/U}(U,U)$. There is precisely one morphism in $\Hom_{\G/U}(U,U)$, since $U$ is the terminal object in $\G/U$.
    
    Now consider the following candidate pushout diagram where $G$ lies in $\Ker(E)$.
    \begin{center}
      \begin{tikzpicture}[node distance=3.0cm, auto]
        \node (A) {$G \times U$};
        \node (B) [below of=A] {$G$};
        \node (C) [right of=A] {$U$};
        \node (D) [right of=B] {$G$};
        \draw[->] (A) to node [swap] {$\pi_G$} (B);
        \draw[->] (A) to node [swap] {$\pi_U$} (C);
        \draw[bend right, dashed, ->] (C) to node [swap] {$!_{G \times U}$} (A);
        \draw[->] (B) to node {$\id_G$} (D);
        \draw[->] (C) to node [swap] {$!_G$} (D);
        \node (X) [below right=1.2cm and 1.2cm of D.center] {$X$};
        \draw[bend right=15,->] (B) to node [swap] {$f$} (X);
        \draw[bend left=15,->] (C) to node {$g$} (X);
        \draw[dashed,->] (D) to node [swap,pos=0.4] {$h$} (X);
      \end{tikzpicture}
    \end{center}
    To see that the square commutes, note that by assumption $G$ lies in $\Ker(E)$ and so $G \times U \cong U$. Therefore, $G \times U$ is initial in $\Ker(E)$ and there is a unique map into $G$.
    
    Now suppose $(f,g)$ is a cocone in $\G$. It is clear that the candidate morphism $h\colon G \to X$ must equal $f$ and so we must just show that $f\circ{!}_G = g$.
    Since $G \times U$ and $U$ are both initial in $\Ker(E)$, $\pi_U$ has an inverse $!_{G\times U}\colon U \to G \times U$.
    
    We now have
    \begin{align*}
        f\circ{!}_G &= f !_G \pi_U !_{G\times U} \\
                    &= f \pi_G !_{G\times U} \\
                    &= g \pi_U !_{G\times U} \\
                    &= g.
    \end{align*}
    
    This gives that $G$ is the pushout and hence fixed by $K^*$.
\end{proof}

\subsection{Adjoint extensions and Artin glueings}\label{sec:adjoint_extensions}

We are now in a position to define our main object of study: adjoint split extensions.

\begin{definition}
    A diagram in $\FLTop$ of the form
    \[
        \splitext{\Nvar}{K}{\G}{E}{E_*}{\H}
    \]
    equipped with a natural isomorphism $\epsilon\colon E E_* \to \Id_\H$
    is called an \emph{adjoint split extension} if $K$ is the 2-kernel of $E$, $E$ is the 2-cokernel of $K$, $E_*$ is the right adjoint of $E$
    and $\epsilon$ is the counit of the adjunction.
\end{definition}

\begin{remark}
  \Cref{prop:cokernelsexist,prop:kernelofcokernel} suggest that every adjoint split extension is equivalent to an extension arising from a closed subtopos and its open complement (in a sense that will be made precise in \cref{def:morphism_of_extensions}).
  Here $U = K(0)$ is a subterminal in $\G$.
  \begin{center}
   \begin{tikzpicture}[node distance=2.7cm, auto]
    \node (A) {$\Nvar$};
    \node (B) [right of=A] {$\G$};
    \node (C) [right of=B] {$\H$};
    \node (D) [below of=A] {$\G_{\mathfrak{c}(U)}$};
    \node (E) [right of=D] {$\G$};
    \node (F) [right of=E] {$\G/U$};
    \draw[->] (A) to node {$K$} (B);
    \draw[transform canvas={yshift=0.5ex},->] (B) to node {$E$} (C);
    \draw[transform canvas={yshift=-0.5ex},->] (C) to node {$E_*$} (B);
    \draw[right hook->] (D) to (E);
    \draw[transform canvas={yshift=0.5ex},->] (E) to node {$(-)\times U$} (F);
    \draw[transform canvas={yshift=-0.5ex},->] (F) to node {$(-)^U$} (E);
    \draw[double equal sign distance] (B) to (E);
    \draw[->] (A) to node [swap, sloped] {$\sim$} (D);
    \draw[->] (C) to node [sloped] {$\sim$} (F);
   \end{tikzpicture}
  \end{center}
\end{remark}

The above situation is precisely the setting in which Artin glueings are studied and it is well known that in this case $\G$ is equivalent to an Artin glueing $\Gl(K^*E_*)$.
We will present an alternative proof of this result from the perspective of extensions.

We begin by showing that Artin glueings can be viewed as adjoint split extensions in a natural way.

\begin{proposition}\label{prop:glueingisadjoint}
Let $F\colon \H \to \Nvar$ be a finite-limit-preserving functor. Then the diagram \[\splitext{\Nvar}{\pi_{1*}}{\Gl(F)}{\pi_2}{\pi_{2*}}{\H}\]
is an adjoint split extension in $\FLTop$.
\end{proposition}

\begin{proof}
  We first note that $\Gl(F)$ is a topos. This is a fundamental result in the theory of Artin glueings of toposes and a proof can be found in \cite{wraith1974glueing}.
  
  By \cref{prp:kernelfullsubcat}, it is immediate that $\pi_{1*}$ is the 2-kernel of $\pi_2$.
  To see that $\pi_2$ is the 2-cokernel of $\pi_{1*}$, we first observe that the slice category of $\Gl(F)$ by the subterminal $(0,1,!) = \pi_{1*}(0)$ is equivalent to $\H$. The objects of $\Gl(F)/(0,1,!)$ are isomorphic to those the form $(0,H,!)$ (since every morphism into an initial object in $\Nvar$ is an isomorphism) and its morphisms of the form $(!,f)$. If $L\colon \Gl(F)/(0,1,!) \to \H$ is this isomorphism sending $(0,H,!)$ to $H$ and $(!,f)$ to $f$ and $E\colon \Gl(F) \to \Gl(F)/(0,1,!)$ is the cokernel map, then it is clear that $LE \cong \pi_2$.
\end{proof}

The following proposition is shown for Grothendieck toposes in \cite{sga4vol1}, but deserves to be more well known. Here we prove it for general elementary toposes. (It also follows easily from the theory of Artin glueings, but here we will use it to develop that theory.)

Recall that we use $\theta$ for the unit of the open subtopos adjunction $E \dashv E_*$ and $\zeta$ for the unit of the closed subtopos adjunction $K^* \dashv K$.

\begin{proposition}\label{prop:pullback_representation_objects}
 Let $\G$ be a topos and consider an open subtopos $E_*\colon \H \hookrightarrow \G$ with closed complement $K\colon \Nvar \hookrightarrow \G$.
 Then each object $G$ in $\G$ can be expressed as the following pullback in $\G$ of objects from $\Nvar$ and $\H$.
 \begin{center}
  \begin{tikzpicture}[node distance=3.5cm, auto]
    \node (A) {$G$};
    \node (B) [below of=A] {$K K^*(G)$};
    \node (C) [right of=A] {$E_*E(G)$};
    \node (D) [right of=B] {$K K^*E_*E(G)$};
    \draw[->] (A) to node [swap] {$\zeta_G$} (B);
    \draw[->] (A) to node {$\theta_G$} (C);
    \draw[->] (B) to node [swap] {$K K^*\theta_G$} (D);
    \draw[->] (C) to node {$\zeta_{E_*E(G)}$} (D);
    \begin{scope}[shift=({A})]
        \draw +(0.25,-0.75) -- +(0.75,-0.75) -- +(0.75,-0.25);
    \end{scope}
  \end{tikzpicture}
 \end{center}
\end{proposition}
\begin{proof}
First note that the diagram commutes by the naturality of $\zeta$.
Recall that, setting $U = KK^*(0) \cong K(0)$,
we have $E_*E(G) = (G \times U)^U \cong G^U$
and that $KK^*(G)$ is the pushout of $G$ and $U$ along the projections $\pi_1\colon G \times U \to G$ and $\pi_2\colon G \times U \to U$.
Now we can rewrite the relevant pullback diagram as follows.
\begin{center}
  \begin{tikzpicture}[node distance=3.3cm, auto]
    \node (A) {$P$};
    \node (B) [below of=A] {$G +_{G \times U} U$};
    \node (C) [right=3.2 of A] {$G^U$};
    \node (D) [below of=C] {$G^U +_{G^U \times U} U$};
    \draw[->] (A) to node [swap] {} (B);
    \draw[->] (A) to node {} (C);
    \draw[->] (B) to node [swap] {$\begin{pmatrix}\iota_{G^U} \circ c \\ \iota_U\end{pmatrix}$} (D);
    \draw[->] (C) to node {$\iota_{G^U}$} (D);
    
    \begin{scope}[shift=({A})]
        \draw +(0.25,-0.75) -- +(0.75,-0.75) -- +(0.75,-0.25);
    \end{scope}
    
    \node (X) [above left=1.2cm and 1.2cm of A.center] {$G$};
    \draw[out=-90,->] (X) to node [swap] {$\iota_G$} (B);
    \draw[out=0,->] (X) to node {$c$} (C);
    \draw[->, pos=0.55, dashed] (X) to node {$r$} (A);
  \end{tikzpicture}
\end{center}
Here the $\iota$ maps are injections into the pushout and $c$ is the unit of the exponential adjunction, which intuitively maps elements of $G$ to their associated constant functions.

Let us express $P$ in the internal logic.
We have \[P = \{(f,[g]) \mid f \sim c(g) \} \cup \{(f,[*]) \mid f \sim {*} \} \subseteq G^U \times (G+U)/\sim\] where $\sim$ denotes the equivalence relation generated by $f \sim {*}$ for ${*} \in U$. Explicitly, we find that $f \sim f' \iff f = f' \lor {*} \in U$. 
Thus, we find
\begin{align*}
    P &= \{(f,[g]) \mid f = c(g) \lor {*} \in U \} \cup \{(f,[*]) \mid f \in G^U,\, {*} \in U \} \\
      &= \{(c(g),[g]) \mid g \in G \} \cup \{(f,[g]) \mid {*} \in U \} \cup \{(f,[*]) \mid {*} \in U \}.
\end{align*}
Now observe that if ${*} \in U$, then $[g] = [*]$ and hence $\{(f,[g]) \mid {*} \in U \} \subseteq \{(f,[*]) \mid {*} \in U \}$.
Finally, commuting the subobject and the quotient we arrive at
\[P = \left(\{(c(g),g) \mid g \in G \} \sqcup \{(f,*) \mid {*} \in U \}\right)/\sim\]
where the equivalence relation is generated by $(f,g) \sim (f,{*})$ for ${*} \in U$.
Note that the union is now disjoint.

The map $r\colon G \to P$ sends $g \in G$ to $[(c(g), g)]$.
We can define a candidate inverse by $s\colon P \to G$ by $(f,g) \mapsto g$ and $(f,{*}) \mapsto f({*})$,
which can seen to be well-defined, since if $(f,g)$ and $(f',{*})$ are elements of the disjoint union with $f = f'$ and ${*} \in U$, then $f'({*}) = c(g)({*}) = g$.

We clearly have $sr = \id_G$. Now if $[(c(g),g)] \in P$ then
$rs([(c(g),g)]) = r(g) = [(c(g), g)]$. On the other hand, if $[(f,{*})] \in P$ then $* \in U$ and $rs([f,{*}]) = r(f({*})) = [(c(f({*})), {*})]$, which equals $[(f,{*})]$ since $f({*}) = c(f({*}))(*)$ so that $f = c(f({*}))$. Thus, $r$ and $s$ are inverses as required.
\end{proof}

\begin{remark}
 The non-classical logic in the above proof can be hard to make sense of. It can help to consider the cases where $U = 0$ and $U = 1$. In the former case, $G^U$ contains no information and $G +_{G \times U} U \cong G$, while in the latter case the opposite is true. The general case `interpolates' between these.
\end{remark}

The next result will play a central role in this paper.

\begin{proposition}\label{prop:pullback_representation_functors}
 Let $\G$ be a topos and consider an open subtopos $E_*\colon \H \hookrightarrow \G$ with closed complement $K\colon \Nvar \hookrightarrow \G$.
 We have the following pullback in the category of (finite-limit-preserving) endofunctors on $\G$.
 \begin{center}
  \begin{tikzpicture}[node distance=3.0cm, auto]
    \node (A) {$\Id_\G$};
    \node (B) [below of=A] {$K K^*$};
    \node (C) [right of=A] {$E_* E$};
    \node (D) [right of=B] {$K K^* E_* E$};
    \draw[->] (A) to node [swap] {$\zeta$} (B);
    \draw[->] (A) to node {$\theta$} (C);
    \draw[->] (B) to node [swap] {$K K^* \theta$} (D);
    \draw[->] (C) to node {$\zeta E_* E$} (D);
    \begin{scope}[shift=({A})]
        \draw +(0.25,-0.75) -- +(0.75,-0.75) -- +(0.75,-0.25);
    \end{scope}
  \end{tikzpicture}
 \end{center}
\end{proposition}
\begin{proof}
We proved this on objects in \cref{prop:pullback_representation_objects}.
Now consider a morphism $p\colon G \to G'$ in $\G$. We obtain the following commutative cube.
\begin{center}
  \begin{tikzpicture}[node distance=4.5cm, auto, on grid,
    cross line/.style={preaction={draw=white, -,line width=6pt}}
    ]
    
    \node (A) {$G'$};
    \node [right of=A] (B) {$E_*E(G')$};
    \node [below of=A] (C) {$KK^*(G')$};
    \node [right of=C] (D) {$KK^*E_*E(G')$};
    
    \begin{scope}[shift=({A})]
        \draw [draw=black] +(0.25,-0.75) coordinate (X) -- +(0.75,-0.75) coordinate (Y) -- +(0.75,-0.25) coordinate (Z);
    \end{scope}
    
    \node [above left=1.4cm and 1.9cm of A.center] (A1) {$G$};
    \node [right of=A1] (B1) {$E_*E(G)$};
    \node [below of=A1] (C1) {$KK^*(G)$};
    \node [right of=C1] (D1) {$KK^*E_*E(G)$};
    
    \begin{scope}[shift=({A1})]
        \draw [draw=black!70] +(0.22,-0.78) coordinate (X1) -- +(0.78,-0.78) coordinate (Y1) -- +(0.78,-0.22) coordinate (Z1);
    \end{scope}
    
    \draw[->] (A1) to node {$\theta_G$} (B1);
    \draw[->] (A1) to node [swap] {$\zeta_G$} (C1);
    \draw[->] (B1) to node [pos=0.65] {$\zeta_{E_*E(G)}$} (D1);
    \draw[->] (C1) to node [node on layer=over,fill=white,inner sep=2pt,outer sep=0pt,yshift=1pt] (K) {\small  $KK^*\theta_G$} (D1);
    
    \draw[cross line, ->] (A) to node [pos=0.35] {$\theta_{G'}$} (B);
    \draw[cross line, ->] (A) to node [swap, pos=0.25] {$\zeta_{G'}$} (C);
    \draw[cross line, ->] (B) to node {$\zeta_{E_*E(G')}$} (D);
    \draw[cross line, ->] (C) to node [swap] {\small $KK^*\theta_{G'}$} (D);
    
    \begin{scope}
    \begin{pgfonlayer}{over}
    \clip (K.south west) rectangle (K.north east);
    \draw[->, opacity=0.7] (A) to (C);
    \end{pgfonlayer}
    \end{scope}
    
    \draw[->] (A1) to node [yshift=-2pt] {$p$} (A);
    \draw[->] (B1) to node [xshift=-1pt, yshift=-3pt] {\small $E_*E(p)$} (B);
    \draw[->] (C1) to node [swap, yshift=2pt] {\small $KK^*(p)$} (C);
    \draw[->] (D1) to node [swap, pos=0.2, xshift=5pt] {\small $KK^*E_*E(p)$} (D);
   \end{tikzpicture}
 \end{center}
 An enjoyable diagram chase around the cube shows that $p$ is the unique morphism making the cube commute by the pullback property of the front face.
 Since pullbacks in the functor category are computed pointwise, this yields the desired result.
\end{proof}

\begin{remark}
It is remarked in \cite{faul2019artin} that adjoint extensions of frames can be viewed as weakly Schreier split extensions of monoids as defined in \cite{bourn2015partialMaltsev}. \Cref{prop:pullback_representation_objects,prop:pullback_representation_functors} can be viewed as a categorified version of the weakly Schreier condition for the topos setting, though it is as yet unclear how the weakly Schreier condition might be categorified more generally. It would also be interesting to see how the general theory of ($\mathcal{S}$-)protomodular categories \cite{borceux2004protomodular,bourn2015Sprotomodular} might be categorified. Another potential example of 2-dimensional protomodularity can be found in \cite{kasangian2006split}.
\end{remark}

We can now prove the main result of this section.
For now we will treat equivalences of adjoint extensions without worrying too much about coherence, which we will discuss in more detail when we define morphisms of extensions in \cref{sec:extension_category}.
\begin{theorem}\label{thm:glueingequivalence}
  Let $\splitext{\Nvar}{K}{\G}{E}{E_*}{\H}$ be an adjoint split extension. Then it is equivalent to the split extension $\splitext{\Nvar}{\pi_{1*}}{\Gl(K^*E_*)}{\pi_2}{\pi_{2*}}{\H}$.
\end{theorem}
\begin{proof}
We denote the unit of $\pi_2 \dashv \pi_{2*}$ by $\theta'\colon \Id_{\Gl(K^*E_*)} \to \pi_{2*}\pi_2$ and observe that its component at $(N,H,\ell)$ can be given explicitly by $(\ell,\id_{H})\colon (N,H,\ell) \to (K^*E_*(H),H,\id_H)$.
We denote the unit of $\pi_1 \dashv \pi_{1*}$ by $\zeta'\colon \Id_{\Gl(K^*E_*)} \to \pi_{1*}\pi_1$ and its components are given by $(\id_N,!)\colon (N,H,\ell) \to (N,1,!)$.

Consider the functor $\Phi\colon \G \to \Gl(K^*E_*)$ which sends objects $G$ to the triple $(K^*(G),E(G),K^*\theta_G)$ and morphisms $f\colon G \to G'$ to $(K^*(f),E(f))$. This is a morphism of extensions in the sense that we have natural isomorphisms $\alpha\colon \pi_{1*} \to \Phi K$, $\beta\colon \pi_2 \Phi \to E$ and $\gamma\colon \pi_{2*} \to \Phi E_*$
given by $\alpha = (\id,!) \circ \pi_{1*}\delta\inv$ (where $\delta\colon K^* K \to \Id_\Nvar$ is the counit of $K^* \dashv K$ and $(\id,!)\colon \pi_{1*} K^* K \cong \Phi K$ is simply to ensure the different choices of terminal object agree),
$\beta = \id$ (using $\pi_2 \Phi = E$) and $\gamma = (\id_{K^*E_*}, \epsilon\inv)$ (where $\epsilon\colon EE_* \to \Id_\H$ is the counit of $E \dashv E_*$ and this is a morphism in $\Gl(K^*E_*)$ by the triangle identity). This can be seen to be a morphism of extensions as defined in \cref{def:morphism_of_extensions} (see \cref{rem:glueingequivalence_coherence}).

We must show that $\Phi$ is an equivalence.
We claim that the following pullback in $\Hom(\Gl(K^*E_*),\G)$ is the inverse of $\Phi$. (Here the bottom equality comes from $K^*E_* = \pi_1\pi_{2*}$.)

\begin{center}
  \begin{tikzpicture}[node distance=3.3cm, auto]
    \node (A) {$\Phi'$};
    \node (B) [below of=A] {$K \pi_{1}$};
    \node (D) [right=1.5cm of B] {$K \pi_1\pi_{2*} \pi_2$};
    \node (D') [right=0.7cm of D] {$K K^* E_* \pi_2$};
    \node (C) [above of=D'] {$E_* \pi_2$};
    \draw[->] (A) to node [swap] {} (B);
    \draw[->] (A) to node {} (C);
    \draw[->] (B) to node [swap] {$K \pi_1 \theta'$} (D);
    \draw[->] (C) to node {$\zeta E_* \pi_2$} (D');
    \draw[double equal sign distance] (D) to (D');
    \begin{scope}[shift=({A})]
        \draw +(0.25,-0.75) -- +(0.75,-0.75) -- +(0.75,-0.25);
    \end{scope}
  \end{tikzpicture}
\end{center}

We shall make extensive use of \cref{prop:pullback_representation_functors} in order to prove this.

To see that $\Phi' \Phi \cong \id_\G$ note that composition with $\Phi$ on the right preserves limits and thus can be represented as the following pullback in $\Hom(\G,\G)$.

\begin{center}
  \begin{tikzpicture}[node distance=3.3cm, auto]
    \node (A) {$\Phi'\Phi$};
    \node (B) [below of=A] {$K \pi_{1}\Phi$};
    \node (C) [right of=A] {$E_* \pi_2\Phi$};
    \node (D) [right of=B] {$K K^* E_* \pi_2\Phi$};
    \draw[->] (A) to node [swap] {} (B);
    \draw[->] (A) to node {} (C);
    \draw[->] (B) to node [swap] {$K \pi_1 \theta'\Phi$} (D);
    \draw[->] (C) to node {$\zeta E_* \pi_2\Phi$} (D);
    \begin{scope}[shift=({A})]
        \draw +(0.25,-0.75) -- +(0.75,-0.75) -- +(0.75,-0.25);
    \end{scope}
  \end{tikzpicture}
\end{center}

Note that $\pi_1\Phi = K^*$, $\pi_2\Phi = E$ and that $\theta'_{\Phi(G)} = (K^*\theta_G,\id_{E(G)})$ which of course gives that $K\pi_1\theta'\Phi = KK^*\theta$. After making these substitutions into the diagram above, we have the pullback square occurring in \cref{prop:pullback_representation_functors}, which by the universal property gives that $\Phi\Phi'$ is naturally isomorphic to $\Id_\G$.

The same idea works for $\Phi\Phi'$. We consider the following pullback square in the category of endofunctors on $\Gl(K^*E_*)$.
\begin{center}
  \begin{tikzpicture}[node distance=3.5cm, auto]
    \node (A) {$\Phi\Phi'$};
    \node (B) [below of=A] {$\Phi K \pi_{1}$};
    \node (C) [right of=A] {$\Phi E_* \pi_2$};
    \node (D) [right of=B] {$\Phi KK^* E_* \pi_2$};
    \node (B') [below=1.5cm of B] {$\pi_{1*}\pi_1$};
    \node (C') [right=1.5cm of C] {$\pi_{2*}\pi_2$};
    \node (D') at (B' -| C') {$\pi_{1*}\pi_1\pi_{2*}\pi_2$};
    \draw[->] (A) to node [swap] {} (B);
    \draw[->] (A) to node {} (C);
    \draw[->] (B) to node [swap] {$\Phi K \pi_1 \theta'$} (D);
    \draw[->] (C) to node {$\Phi\zeta E_* \pi_2$} (D);
    \draw[->] (B) to node [swap] {$\alpha\inv \pi_1$} node [sloped] {$\sim$} (B');
    \draw[->] (C) to node {$\gamma\inv \pi_2$} node [swap,sloped] {$\sim$} (C');
    \draw[->] (D) to node [swap,yshift=5pt] {$\alpha\inv \pi_1 \pi_{2*}\pi_2$} node [sloped] {$\sim$} (D');
    \draw[->] (B') to node [swap] {$\pi_{1*}\pi_1 \theta'$} (D');
    \draw[->] (C') to node {$\zeta'\pi_{2*}\pi_2$} (D');
    \begin{scope}[shift=({A})]
        \draw +(0.25,-0.75) -- +(0.75,-0.75) -- +(0.75,-0.25);
    \end{scope}
  \end{tikzpicture}
\end{center}

The bottom trapezium commutes by naturality of $\alpha\inv$ (and using $K^*E_* = \pi_1\pi_{2*}$), whereas both paths around the right-hand trapezium can be seen to compose to $(\id_{K^*E_*(H)}, !)$.
Thus this pullback diagram has the same form as that in \cref{prop:pullback_representation_functors} and so \cref{prop:glueingisadjoint} allows us to deduce that $\Phi\Phi'$ is naturally isomorphic to the identity, completing the proof.
\end{proof}
Together with \cref{prop:glueingisadjoint} this shows that adjoint split extensions and Artin glueings are essentially the same. The equivalence $\Phi$ so defined is natural in a sense that will become clear later in \cref{sec:equivalence}.

\section{The category of extensions}\label{sec:extension_category}

It follows from \cref{thm:glueingequivalence} that equivalence classes of adjoint extensions between $\H$ and $\Nvar$ are in bijection with those of $\Hom(\H,\Nvar)$. However, the extensions have a natural 2-categorical structure and so we would like to know how this relates to the categorical structure of $\Hom(\H,\Nvar)$.

\subsection{Morphisms of extensions}

\begin{definition}\label{def:morphism_of_extensions}
Suppose we have two adjoint extensions,
\begin{align*}
    \splitext{\Nvar}{K_1}{\G_1}{E_1}{E_{1*}}{\H} &\quad\text{and}\quad \splitext{\Nvar}{K_2}{\G_2}{E_2}{E_{2*}}{\H},
\end{align*}
with the same kernel and cokernel objects and associated isomorphisms $\epsilon_1$ and $\epsilon_2$ respectively. Consider the following diagram
where $\alpha$, $\beta$ and $\gamma$ are natural isomorphisms and $\Psi$ is a finite-limit-preserving functor.
\begin{center}
    \begin{tikzpicture}[node distance=2.5cm, auto]
    \node (A) {$\Nvar$};
    \node (B) [right of=A] {$\G_1$};
    \node (C) [right of=B] {$\H$};
    \node (C2) [right of=C] {$\G_1$};
    \node (C3) [right of=C2] {$\H$};
    \node (D) [below of=A] {$\Nvar$};
    \node (E) [right of=D] {$\G_2$};
    \node (F) [right of=E] {$\H$};
    \node (F2) [right of=F] {$\G_2$};
    \node (F3) [right of=F2] {$\H$};
    
    \draw[->] (A) to node {$K_1$} (B);
    \draw[->] (B) to node {$E_1$} (C);
    \draw[->] (C) to node {$E_{1*}$} (C2);
    \draw[->] (C2) to node {$E_1$} (C3);
    \draw[->] (D) to node [swap] {$K_2$} (E);
    \draw[->] (E) to node [swap] {$E_2$} (F);
    \draw[->] (F) to node [swap] {$E_{2*}$} (F2);
    \draw[->] (F2) to node [swap] {$E_2$} (F3);
    \draw[->] (B) to node [swap] {$\Psi$} (E);
    \draw[double equal sign distance] (A) to (D);
    \draw[double equal sign distance] (C) to (F);
    \draw[->] (C2) to node [swap] {$\Psi$} (F2);
    \draw[double equal sign distance] (C3) to (F3);
    
    \draw[double distance=4pt,-implies] ($(D)!0.38!(B)$) to node {$\alpha$} ($(D)!0.62!(B)$);
    \draw[double distance=4pt,-implies] ($(E)!0.38!(C)$) to node {$\beta$} ($(E)!0.62!(C)$);
    \draw[double distance=4pt,-implies] ($(F)!0.38!(C2)$) to node {$\gamma$} ($(F)!0.62!(C2)$);
    \draw[double distance=4pt,-implies] ($(F2)!0.38!(C3)$) to node {$\beta$} ($(F2)!0.62!(C3)$);
   \end{tikzpicture}
\end{center}
We say the functor $\Psi\colon \G_1 \to \G_2$ together with $\alpha$, $\beta$ and $\gamma$ is a \emph{morphism of adjoint extensions (of $\H$ by $\Nvar$)} if $\epsilon_2 = \epsilon_1 (\beta E_{1*})  (E_2 \gamma)$.

Given two such morphisms $(\Psi, \alpha, \beta, \gamma)$ and $(\Psi', \alpha', \beta', \gamma')$,
a \emph{2-morphism of adjoint extensions} is a natural transformation $\tau\colon \Psi \to \Psi'$ such that $\alpha' = (\tau K_1) \alpha$, $\beta = \beta' (E_2 \tau)$ and $\gamma' = (\tau E_{1*}) \gamma$.

The morphisms compose in the obvious way: by composing the functors and pasting the natural transformations together by juxtaposing the squares from the diagram above. Horizontal and vertical composition of 2-morphisms is given by the corresponding operations of the natural transformations $\tau$. It is not hard to see that this gives a strict 2-category $\ExtTwo(\H,\Nvar)$.
\end{definition}

\begin{remark}\label{rem:glueingequivalence_coherence}
 Note that the `equivalence of extensions' $(\Phi, \alpha, \beta, \gamma)$ defined in \cref{thm:glueingequivalence} is indeed a morphism of extensions in the above sense, since there $\epsilon_2$ and $\beta$ are identities and $\epsilon (\pi_2 \gamma) = \epsilon \epsilon\inv = \id$, as required.
\end{remark}

\begin{lemma}\label{lmm:betareduction}
 Given two adjoint extensions as above, a functor $\Psi\colon \G_1 \to \G_2$ and natural transformations $\alpha\colon K_2 \to \Psi K_1$ and $\gamma\colon E_{2*} \to \Psi E_{1*}$, there is a unique natural isomorphism $\beta\colon E_2 \Psi \to E_1$ making $(\Psi, \alpha, \beta, \gamma)$ a morphism of adjoint extensions.
 
 Furthermore, given two morphisms of adjoint extensions as above, a natural transformation $\tau\colon \Psi \to \Psi'$ is a 2-morphism of adjoint extensions if and only if $\alpha' = (\tau K_1) \alpha$ and $\gamma' = (\tau E_{1*}) \gamma$.
\end{lemma}
\begin{proof}
  Any such $\beta$ must satisfy $\epsilon_2 = \epsilon_1 (\beta E_{1*})  (E_2 \gamma)$.
  But this can be rewritten as $\epsilon_2 (E_2 \gamma\inv) = \epsilon_1 (\beta E_{1*})$, which shows that $\beta$ and $\gamma\inv$ are mates with respect to the adjunctions $E_1 \dashv E_{1*}$ and $E_2 \dashv E_{2*}$ and hence determine each other.
  
  We now show that $\beta$ so defined is an isomorphism.
  As the mate of $\gamma\inv$, we can express $\beta$ as $(\epsilon_2 E_1) (E_2 \gamma\inv E_1) (E_2\Psi\theta_1)$. Now since $\epsilon_2$ and $\gamma\inv$ are isomorphisms, we need only show $E_2\Psi\theta_1$ is an isomorphism. This map occurs in the pullback obtained by applying $E_2 \Psi$ to the pullback square of \cref{prop:pullback_representation_functors}.
  \begin{center}
   \begin{tikzpicture}[node distance=3.6cm, auto]
    \node (A) {$E_2 \Psi$};
    \node (B) [below of=A] {$E_2 \Psi K_1 K_1^*$};
    \node (C) [right=3.4cm of A] {$E_2 \Psi E_{1*} E_1$};
    \node (D) [below of=C] {$E_2 \Psi K_1 K_1^* E_{1*} E_1$};
    \draw[->] (A) to node [swap] {$E_2 \Psi \zeta_1$} (B);
    \draw[->] (A) to node {$E_2 \Psi \theta_1$} (C);
    \draw[->] (B) to node [swap] {$E_2 \Psi K_1 K_1^* \theta_1$} (D);
    \draw[->] (C) to node {$E_2 \Psi \zeta_1 E_{1*} E_1$} (D);
    \begin{scope}[shift=({A})]
        \draw +(0.25,-0.75) -- +(0.75,-0.75) -- +(0.75,-0.25);
    \end{scope}
   \end{tikzpicture}
  \end{center}
  Now observe that $E_2\Psi K_1 \cong E_2 K_2 \cong 1$ is a zero morphism and hence so are $E_2\Psi K_1 K_1^*$ and $E_2\Psi K_1 K_1^* E_{1*} E_1$. Therefore, the bottom arrow of the above diagram is an isomorphism, and as the pullback of an isomorphism, $E_2\Psi\theta_1$ is an isomorphism too.
  
  Finally, we show that the condition on $\beta$ for 2-morphisms of extensions is automatic.
  Simply observe the following string diagrams. 
  \\ 
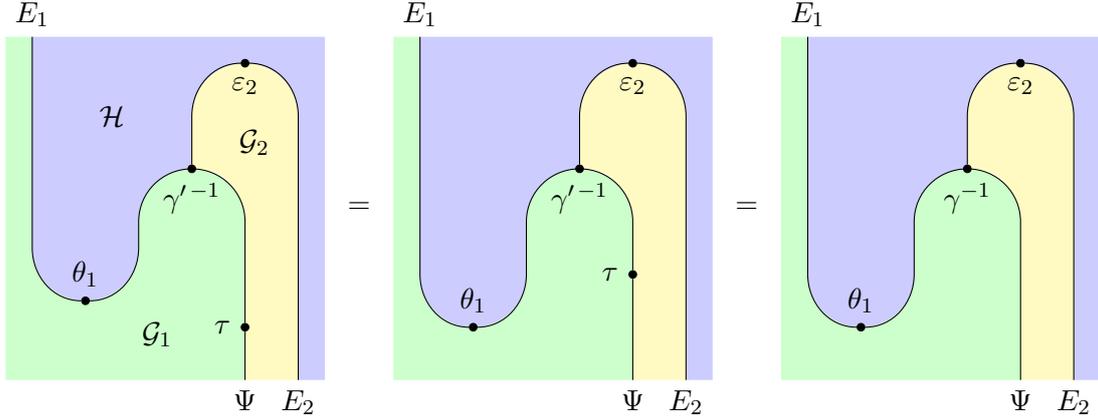
\vspace{-3pt}
  \begin{equation*}
   \begin{tikzpicture}[scale=0.7,baseline={([yshift=-0.5ex]current bounding box.center)}]
    \path coordinate[dot, label=below:$\gamma^{\prime\,-1}$] (mu)
    +(0,1) coordinate (d)
    +(-1,-1) coordinate (mbl)
    +(1,-1) coordinate (mbr);
    \path (d) ++(1,1) coordinate[dot, label=below:$\epsilon_2$] (epsilon) ++(1,-1) coordinate (e) ++(0,-5.0) coordinate[label=below:{\vphantom{$E$}\smash{$E_2$}}] (brr);
    \path (mbl) ++(0,-0.5) coordinate (a) ++(-1.0,-1) coordinate[dot, label=above:$\theta_1$] (eta) ++(-1.0,1) coordinate (c) ++(0,4.0) coordinate[label=above:$E_1$] (tl);
    \path (mbr) ++ (0,-2.0) coordinate[dot,label=left:$\tau$] (fr) ++ (0,-1.0) coordinate[label=below:$\Psi$] (br);
    \draw (tl) -- (c) to[out=-90, in=180] (eta) to[out=0, in=-90] (a)
               -- (mbl) to[out=90, in=180] (mu.center) to[out=0, in=90] (mbr)
               -- (fr) -- (br)
        (mu) -- (d) to[out=90, in=180] (epsilon) to[out=0, in=90] (e) -- (brr);
    \coordinate (cornerNW) at ($(tl) + (-0.5,0)$);
    \coordinate (cornerSE) at ($(brr) + (0.5,0)$);
    \coordinate (cornerNE) at (cornerNW -| cornerSE);
    \coordinate (cornerSW) at (cornerNW |- cornerSE);
    \begin{pgfonlayer}{bg}
    \fill[\colorGone] (cornerSW) -- (cornerNW)
                    -- (tl) -- (c) to[out=-90, in=180] (eta.center) to[out=0, in=-90] (a)
                    -- (mbl) to[out=90, in=180] (mu.center) to[out=0, in=90] (mbr)
                    -- (br) -- cycle;
    \fill[\colorGtwo] (mu.east) to[out=0, in=90] (mbr) -- (br)
                    -- (brr) -- (e) to[in=0, out=90] (epsilon.east) -- (epsilon.west) to[in=90, out=180] (d) -- (mu) -- cycle;
    \fill[\colorH] (tl) -- (c) to[out=-90, in=180] (eta.center) to[out=0, in=-90] (a)
                    -- (mbl) to[out=90, in=180] (mu.center)
                    -- (d) to[out=90, in=180] (epsilon.center) to[out=0, in=90] (e) -- (brr)
                    -- (cornerSE) -- (cornerNE) -- cycle;
    \end{pgfonlayer}
    \coordinate (tmp1) at ($(tl)!0.5!(mu)$);
    \coordinate (tmp2) at ($(mu)!0.5!(epsilon)$);
    \node (Hlabel) at (tmp1 |- tmp2) {$\H$};
    \coordinate (tmp3) at ($(eta)!0.45!(fr)$);
    \coordinate (tmp4) at ($(fr) + (0,-4pt)$);
    \node (Glabel1) at (tmp3 |- tmp4) {$\G_1$};
    \coordinate (tmp5) at ($(mu)!0.5!(brr) + (5pt,0)$);
    \coordinate (tmp6) at ($(epsilon)!0.75!(mu)$);
    \node (Glabel2) at (tmp5 |- tmp6) {$\G_2$};
   \end{tikzpicture}
   \enspace=\enspace
   \begin{tikzpicture}[scale=0.7,baseline={([yshift=-0.5ex]current bounding box.center)}]
    \path coordinate[dot, label=below:$\gamma^{\prime\,-1}$] (mu)
    +(0,1) coordinate (d)
    +(-1,-1) coordinate (mbl)
    +(1,-1) coordinate (mbr);
    \path (d) ++(1,1) coordinate[dot, label=below:$\epsilon_2$] (epsilon) ++(1,-1) coordinate (e) ++(0,-5.0) coordinate[label=below:{\vphantom{$E$}\smash{$E_2$}}] (brr);
    \path (mbl) ++(0,-1) coordinate (a) ++(-1.0,-1) coordinate[dot, label=above:$\theta_1$] (eta) ++(-1.0,1) coordinate (c) ++(0,4.5) coordinate[label=above:$E_1$] (tl);
    \path (mbr) ++ (0,-1.0) coordinate[dot,label=left:$\tau$] (fr) ++ (0,-2.0) coordinate[label=below:$\Psi$] (br);
    \draw (tl) -- (c) to[out=-90, in=180] (eta) to[out=0, in=-90] (a)
               -- (mbl) to[out=90, in=180] (mu.center) to[out=0, in=90] (mbr)
               -- (fr) -- (br)
        (mu) -- (d) to[out=90, in=180] (epsilon) to[out=0, in=90] (e) -- (brr);
    \coordinate (cornerNW) at ($(tl) + (-0.5,0)$);
    \coordinate (cornerSE) at ($(brr) + (0.5,0)$);
    \coordinate (cornerNE) at (cornerNW -| cornerSE);
    \coordinate (cornerSW) at (cornerNW |- cornerSE);
    \begin{pgfonlayer}{bg}
    \fill[\colorGone] (cornerSW) -- (cornerNW)
                    -- (tl) -- (c) to[out=-90, in=180] (eta.center) to[out=0, in=-90] (a)
                    -- (mbl) to[out=90, in=180] (mu.center) to[out=0, in=90] (mbr)
                    -- (br) -- cycle;
    \fill[\colorGtwo] (mu.east) to[out=0, in=90] (mbr) -- (br)
                    -- (brr) -- (e) to[in=0, out=90] (epsilon.east) -- (epsilon.west) to[in=90, out=180] (d) -- (mu) -- cycle;
    \fill[\colorH] (tl) -- (c) to[out=-90, in=180] (eta.center) to[out=0, in=-90] (a)
                    -- (mbl) to[out=90, in=180] (mu.center)
                    -- (d) to[out=90, in=180] (epsilon.center) to[out=0, in=90] (e) -- (brr)
                    -- (cornerSE) -- (cornerNE) -- cycle;
    \end{pgfonlayer}
   \end{tikzpicture}
   \enspace=\enspace
   \begin{tikzpicture}[scale=0.7,baseline={([yshift=-0.5ex]current bounding box.center)}]
    \path coordinate[dot, label=below:$\gamma\inv$] (mu)
    +(0,1) coordinate (d)
    +(-1,-1) coordinate (mbl)
    +(1,-1) coordinate (mbr);
    \path (d) ++(1,1) coordinate[dot, label=below:$\epsilon_2$] (epsilon) ++(1,-1) coordinate (e) ++(0,-5.0) coordinate[label=below:{\vphantom{$E$}\smash{$E_2$}}] (brr);
    \path (mbl) ++(0,-1) coordinate (a) ++(-1.0,-1) coordinate[dot, label=above:$\theta_1$] (eta) ++(-1.0,1) coordinate (c) ++(0,4.5) coordinate[label=above:$E_1$] (tl);
    \path (mbr) ++ (0,-1.0) coordinate (fr) ++ (0,-2.0) coordinate[label=below:$\Psi$] (br);
    \draw (tl) -- (c) to[out=-90, in=180] (eta) to[out=0, in=-90] (a)
               -- (mbl) to[out=90, in=180] (mu.center) to[out=0, in=90] (mbr)
               -- (fr) -- (br)
        (mu) -- (d) to[out=90, in=180] (epsilon) to[out=0, in=90] (e) -- (brr);
    \coordinate (cornerNW) at ($(tl) + (-0.5,0)$);
    \coordinate (cornerSE) at ($(brr) + (0.5,0)$);
    \coordinate (cornerNE) at (cornerNW -| cornerSE);
    \coordinate (cornerSW) at (cornerNW |- cornerSE);
    \begin{pgfonlayer}{bg}
    \fill[\colorGone] (cornerSW) -- (cornerNW)
                    -- (tl) -- (c) to[out=-90, in=180] (eta.center) to[out=0, in=-90] (a)
                    -- (mbl) to[out=90, in=180] (mu.center) to[out=0, in=90] (mbr)
                    -- (br) -- cycle;
    \fill[\colorGtwo] (mu.east) to[out=0, in=90] (mbr) -- (br)
                    -- (brr) -- (e) to[in=0, out=90] (epsilon.east) -- (epsilon.west) to[in=90, out=180] (d) -- (mu) -- cycle;
    \fill[\colorH] (tl) -- (c) to[out=-90, in=180] (eta.center) to[out=0, in=-90] (a)
                    -- (mbl) to[out=90, in=180] (mu.center)
                    -- (d) to[out=90, in=180] (epsilon.center) to[out=0, in=90] (e) -- (brr)
                    -- (cornerSE) -- (cornerNE) -- cycle;
    \end{pgfonlayer}
   \end{tikzpicture}
  \end{equation*}
  Here the first diagram represents $\beta' (E_2 \tau)$ and the last diagram represents $\beta$.
  In moving from the first diagram to the second we shift $\tau$ above $\theta_1$ and
  to move from the second diagram to the third we use $\gamma' = (\tau E_{1*}) \gamma$.
\end{proof}

\begin{lemma}\label{prop:2morphisms_of_extensions}
 Suppose $(\Psi, \alpha, \beta, \gamma)$ and $(\Psi', \alpha', \beta', \gamma')$ are parallel morphisms of extensions. Then any 2-morphism $\tau$ between them is unique and invertible.
 
 Moreover, such a 2-morphism exists if and only if $\alpha'\alpha\inv K_1^* E_{1*} \circ \Psi \zeta_1 E_{1*} = \Psi' \zeta_1 E_{1*} \circ \gamma'\gamma\inv$.
\end{lemma}
\begin{proof}
  Suppose $\tau\colon \Psi \to \Psi'$ is a 2-morphism of extensions. Then we have $\tau K_1 = \alpha'\alpha^{-1}$ and $\tau E_{1*} = \gamma' \gamma^{-1}$. Now by composing the pullback square of \cref{prop:pullback_representation_functors} with $\Psi$ and $\Psi'$ and using the naturality of $\tau$ we have the following commutative cube in $\Hom_{\FLTop}(\G_1, \G_2)$.
  \begin{center}
   \begin{tikzpicture}[node distance=4.5cm, auto, on grid,
    cross line/.style={preaction={draw=white, -,line width=6pt}}
    ]
    
    \node (A) {$\Psi'$};
    \node [right of=A] (B) {$\Psi' E_{1*}E_1$};
    \node [below of=A] (C) {$\Psi' K_1K_1^*$};
    \node [right of=C] (D) {$\Psi' K_1K_1^*E_{1*}E_1$};
    
    \begin{scope}[shift=({A})]
        \draw [draw=black] +(0.25,-0.75) coordinate (X) -- +(0.75,-0.75) coordinate (Y) -- +(0.75,-0.25) coordinate (Z);
    \end{scope}
    
    \node [above left=1.4cm and 1.9cm of A.center] (A1) {$\Psi$};
    \node [right of=A1] (B1) {$\Psi E_{1*}E_1$};
    \node [below of=A1] (C1) {$\Psi K_1K_1^*$};
    \node [right of=C1] (D1) {$\Psi K_1K_1^*E_{1*}E_1$};
    
    \begin{scope}[shift=({A1})]
        \draw [draw=black!70] +(0.22,-0.78) coordinate (X1) -- +(0.78,-0.78) coordinate (Y1) -- +(0.78,-0.22) coordinate (Z1);
    \end{scope}
    
    \draw[->] (A1) to node {$\Psi \theta_1$} (B1);
    \draw[->] (A1) to node [swap] {$\Psi \zeta_1$} (C1);
    \draw[->] (B1) to node [pos=0.65] {$\Psi \zeta_1 E_{1*}E_1$} (D1);
    \draw[->] (C1) to node [node on layer=over,fill=white,inner sep=2pt,outer sep=0pt,yshift=1pt] (K) {\small $\Psi K_1K_1^*\theta_1$} (D1);
    
    \draw[cross line, ->] (A) to node [pos=0.35] {$\Psi' \theta_1$} (B);
    \draw[cross line, ->] (A) to node [swap, pos=0.25] {$\Psi' \zeta_1$} (C);
    \draw[cross line, ->] (B) to node {$\Psi' \zeta_1 E_{1*}E_1$} (D);
    \draw[cross line, ->] (C) to node [swap] {\small $\Psi' K_1K_1^*\theta_1$} (D);
    
    \begin{scope}
    \begin{pgfonlayer}{over}
    \clip (K.south west) rectangle (K.north east);
    \draw[->, opacity=0.7] (A) to (C);
    \end{pgfonlayer}
    \end{scope}
    
    \draw[->] (A1) to node [yshift=-2pt] {$\tau$} (A);
    \draw[->] (B1) to node [xshift=-1pt, yshift=-3pt] {\small $\tau  E_{1*}E_1$} (B);
    \draw[->] (C1) to node [swap, yshift=2pt] {\small $\tau K_1K_1^*$} (C);
    \draw[->] (D1) to node [swap, pos=0.2, xshift=5pt] {\small $\tau K_1K_1^*E_{1*}E_1$} (D);
   \end{tikzpicture}
  \end{center}
  The universal property of the pullback on the front face then gives that $\tau$ is uniquely determined by $\tau K_1K_1^*$ and $\tau  E_{1*}E_1$, and hence by $\tau K_1 = \alpha'\alpha^{-1}$ and $\tau E_{1*} = \gamma'\gamma^{-1}$. Thus, the morphism $\tau$ is unique if it exists.
  
  We can also attempt to use a similar cube to construct $\tau$ without assuming it exists a priori by replacing $\tau K_1$ with $\alpha'\alpha\inv$ and $\tau E_{1*} = \gamma'\gamma\inv$ in the above diagram. However, in order to obtain a map $\tau$ from the universal property of the pullback, we require that the right-hand face commutes.
  \begin{center}
   \begin{tikzpicture}[node distance=3.0cm, auto]
    \node (A) {$\Psi E_{1*} E_1$};
    \node (B) [below of=A] {$\Psi K_1 K_1^* E_{1*} E_1$};
    \node (C) [right=3.4cm of A] {$\Psi' E_{1*} E_1$};
    \node (D) [below of=C] {$\Psi' K_1 K_1^* E_{1*} E_1$};
    \draw[->] (A) to node [swap] {$\Psi \zeta_1 E_{1*} E_1$} (B);
    \draw[->] (A) to node {$\gamma'\gamma\inv E_1$} (C);
    \draw[->] (B) to node [swap] {$\alpha'\alpha\inv K_1^* E_{1*} E_1$} (D);
    \draw[->] (C) to node {$\Psi' \zeta_1 E_{1*} E_1$} (D);
   \end{tikzpicture}
  \end{center}
  Note that this square commutes if and only if the similar square obtained by inverting the horizontal morphisms commutes. But this latter square is precisely the square we need to commute to obtain a 2-morphism in the opposite direction. Uniqueness then shows that these two 2-morphisms compose to give identities.
  
  Finally, observe that commutativity of this square is the required equality stated above whiskered with $E_1$. This is equivalent to the desired condition, since $E_1$ is essentially surjective.
\end{proof}

\begin{corollary}\label{cor:extensions_only_1_cat}
The 2-category of adjoint extensions $\ExtTwo(\H,\Nvar)$ is equivalent to the locally trivial 2-category $\ExtOne(\H,\Nvar)$ of adjoint extensions and isomorphism classes of morphisms.
\end{corollary}
This will justify working with $\ExtOne(\H,\Nvar)$ (viewed as a 1-category) going forward.

\begin{lemma}\label{lem:associatednat}
  From a morphism of extensions $(\Psi, \alpha, \beta, \gamma)$ we can form an associated natural transformation
  from $K_2^*E_{2*}$ to $K_1^*E_{1*}$ given by $(\delta_2 K_1^*E_{1*}) (K_2^* \alpha\inv K_1^*E_{1*}) (K_2^* \Psi \zeta_1 E_{1*}) (K_2^* \gamma)$ where $\delta_2$ is the counit of the $K_2^* \dashv K_2$ adjunction 
  and which is depicted below.
  Two parallel morphisms of extensions are isomorphic if and only if their corresponding natural transformations are equal.
  
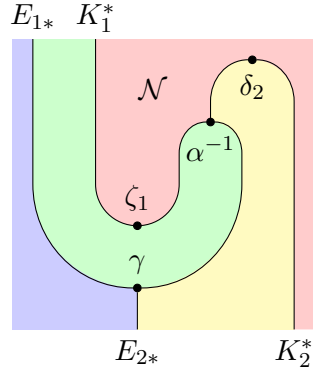
\begin{center}
   \begin{tikzpicture}[scale=0.55,baseline={([yshift=-0.5ex]current bounding box.center)}]
    \path coordinate[label=below:{\vphantom{$E_*$}\smash{$E_{2*}$}}] (b)
      ++(0,1) coordinate[dot, label=above:$\gamma$] (d)
       +(-2.5,2.5) coordinate (minvl)
       +(2.5,2.5) coordinate (minvr);
    \path (minvl) +(0,3.5) coordinate[label=above:{\vphantom{$E_1$}\smash{$E_{1*}$}}] (itl);
    \path (minvr) +(0,0.75) coordinate (itr);
    \path (d) ++(0,1.5) coordinate[dot, label=above:$\zeta_1$] (e)
              +(-1.0,1.0) coordinate (el)
              +(1.0,1.0) coordinate (er);
    \path (el) ++(0,3.5) coordinate[label=above:{\vphantom{$K_1$}\smash{$K_1^*$}}] (etl);
    \path (er) ++(0,0.75) coordinate (f)
               ++(0.75,0.75) coordinate[dot, label=below:$\alpha\inv$] (a)
               ++(0,0.5) coordinate (g1)
                +(1.0,1.0) coordinate[dot, label=below:$\delta_2$] (h)
                +(2.0,0) coordinate (g2);
    \coordinate[label=below:{\vphantom{$K_2$}\smash{$K_2^*$}}] (br) at (b -| g2);
    \draw (b) -- (d)
          (itl) -- (minvl) to[out=-90, in=180] (d) to[out=0, in=-90] (minvr) -- (itr)
          (etl) -- (el) to[out=-90, in=180] (e.center) to[out=0, in=-90] (er) -- (f)
          (f) to[out=90, in=180] (a) to[out=0, in=90] (itr)
          (a.center) -- (g1) to[out=90, in=180] (h) to[out=0, in=90] (g2) -- (br);
    \coordinate (cornerNW) at ($(itl) + (-0.5,0)$);
    \coordinate (cornerSE) at ($(br) + (0.5,0)$);
    \begin{pgfonlayer}{bg}
    \fill[\colorH] (cornerNW) rectangle (b);
    \fill[\colorGtwo] (b |- cornerNW) rectangle (br);
    \fill[\colorGone] (itl) -- (minvl) to[out=-90, in=180] (d.center) to[out=0, in=-90] (minvr) -- (itr) to[in=0, out=90] (a.center) -- (a.west) to[in=90, out=180] (f) -- (er) to[in=0, out=-90] (e.center) to[in=-90, out=180] (el) -- (etl) -- cycle;
    \fill[\colorN] (etl) -- (el) to[out=-90, in=180] (e.center) to[out=0, in=-90] (er) -- (f) to[out=90, in=180] (a.west) -- (a.center) -- (g1) to[out=90, in=180] (h.west) -- (h.center) to[out=0, in=90] (g2) -- (br) -- (cornerSE) -- (cornerSE |- cornerNW) -- cycle;
    \end{pgfonlayer}
    \coordinate (tmp1) at ($(etl)!0.5!(a)$);
    \coordinate (tmp2) at ($(etl)!0.6!(a)$);
    \node (Nlabel) at (tmp1 |- tmp2) {$\Nvar$};
   \end{tikzpicture}
  \end{center}
\end{lemma}
\begin{proof}
  There is a (necessarily invertible) 2-morphism from $(\Psi, \alpha, \beta, \gamma)$ to $(\Psi', \alpha', \beta', \gamma')$ if and only if $(\alpha'\alpha\inv K_1^*E_{1*}) (\Psi \zeta_1 E_{1*}) = (\Psi' \zeta_1 E_{1*}) \gamma'\gamma\inv$.
  We can now move all the unprimed variables to the left and primed variables to the right by
  multiplying both sides of this equation on the left by $\alpha^{\prime\,-1} K_1^*E_{1*}$ and on the right by $\gamma$
  to obtain $(\alpha\inv K_1^*E_{1*}) (\Psi \zeta_1 E_{1*}) \gamma = (\alpha^{\prime\,-1} K_1^*E_{1*}) (\Psi' \zeta_1 E_{1*}) \gamma'$.
  These are the mates of the desired natural transformations with respect to the adjunction $K_2^* \dashv K_2$.
\end{proof}

\subsection{The equivalence of categories}\label{sec:equivalence}

In this section we show that the categories $\ExtOne(\H,\Nvar)$ and $\Hom(\H,\Nvar)\op$ are equivalent. This requires showing that isomorphism classes of morphisms of extensions correspond to natural transformations. We have already seen that each isomorphism class has an associated natural transformation. We will now further explore this relationship, making use of the following folklore result.

\begin{proposition}\label{prp:cocomma}
 Let $\splitext{\Nvar}{K}{\G}{E}{E_*}{\H}$ be an adjoint extension. Then the following diagram is a cocomma square.
 \begin{center}
  \begin{tikzpicture}[node distance=3.0cm, auto]
    \node (A) {$\H$};
    \node (B) [below of=A] {$\Nvar$};
    \node (C) [right of=A] {$\H$};
    \node (D) [right of=B] {$\G$};
    \draw[->] (A) to node [swap] {$K^*E_*$} (B);
    \draw[double equal sign distance] (A) to node {} (C);
    \draw[->] (B) to node [swap] {$K$} (D);
    \draw[->] (C) to node {$E_*$} (D);
    \draw[double distance=4pt,-implies] ($(B)!0.62!(C)$) to node [swap, yshift=3pt, xshift=3pt] {$\zeta E_*$} ($(B)!0.38!(C)$);
    
    \begin{scope}[shift=({D})]
        \draw[dashed] +(-0.25,0.75) -- +(-0.75,0.75) -- +(-0.75,0.25);
    \end{scope}
  \end{tikzpicture}
 \end{center}
\end{proposition}
\begin{proof}
  We first check the 2-categorical condition. Consider two finite-limit-preserving functors $U,V \colon \G \to \mathcal{X}$ and natural transformations $\mu \colon UE_* \to VE_*$ and $\nu \colon UK \to VK$ such that $(V \zeta E_*) \mu = (\nu K^*E_*) (U \zeta E_*)$. We must find a unique $\omega\colon U \to V$ such that $\omega E_* = \mu$ and $\omega K = \nu$.
  
  We use \cref{prop:pullback_representation_functors} to express $U$ and $V$ as pullbacks and then as in \cref{prop:2morphisms_of_extensions} we find that there is a unique map $\omega\colon U \to V$ with $\omega E_* = \mu$ and $\omega K = \nu$ as long as the following diagram commutes.
  \begin{center}
   \begin{tikzpicture}[node distance=3.2cm, auto]
    \node (A) {$U E_* E$};
    \node (B) [below of=A] {$U K K^* E_* E$};
    \node (C) [right=2.7cm of A] {$V E_* E$};
    \node (D) [below of=C] {$V K K^* E_* E$};
    \draw[->] (A) to node [swap] {$U \zeta E_* E$} (B);
    \draw[->] (A) to node {$\mu E$} (C);
    \draw[->] (B) to node [swap] {$\nu K^* E_* E$} (D);
    \draw[->] (C) to node {$V \zeta E_* E$} (D);
   \end{tikzpicture}
  \end{center}
  But commutativity of this diagram is simply the assumed condition whiskered with $E$ on the right.
  
  Now we show the 1-categorical condition. Suppose we have finite-limit-preserving functors $T_1\colon \H \to \mathcal{X}$ and $T_2\colon \Nvar \to \mathcal{X}$ and a natural transformation $\phi\colon T_1 \to T_2 K^* E_*$. We must construct a finite-limit-preserving functor $L\colon \G \to \mathcal{X}$ and natural isomorphisms $\tau_1\colon LE_* \to T_1$ and $\tau_2\colon LK \to T_2$ depicted below such that $\phi = \tau_2K^*E_* \circ L\zeta E_* \circ \tau_1\inv$.
  \begin{center}
   \begin{tikzpicture}[node distance=2.8cm, auto]
    \node (A) {$\H$};
    \node (B) [below of=A] {$\Nvar$};
    \node (C) [right of=A] {$\H$};
    \node (D) [right of=B] {$\G$};
    \draw[->] (A) to node [swap] {$K^*E_*$} (B);
    \draw[double equal sign distance] (A) to node {} (C);
    \draw[->] (B) to node [swap] {$K$} (D);
    \draw[->] (C) to node {$E_*$} (D);
    \draw[double distance=4pt,-implies] ($(B)!0.62!(C)$) to node [swap, yshift=3pt, xshift=3pt] {$\zeta E_*$} ($(B)!0.38!(C)$);
    
    \node (X) [below right=1.2cm and 1.2cm of D.center] {$\mathcal{X}$};
    \draw[bend right,->] (B) to node [swap] {$T_2$} node [anchor=center] (T') {} (X);
    \draw[bend left,->] (C) to node {$T_1$} node [anchor=center] (T) {} (X);
    \draw[dashed, ->, pos=0.55] (D) to node [swap] {$L$} (X);
    
    \draw[double distance=4pt,implies-] ($(D)!0.3!(T)$) to node [swap,pos=0.5,yshift=2pt] {$\tau_1\inv$} ($(D)!0.8!(T)$);
    \draw[double distance=4pt,-implies] ($(D)!0.3!(T')$) to node [pos=0.33] {$\tau_2$} ($(D)!0.8!(T')$);
   \end{tikzpicture}
  \end{center}
  
  Suppose we are given such a functor $L$ and natural isomorphisms $\tau_1, \tau_2$ and consider the following pullback diagram.
  \begin{center}
   \begin{tikzpicture}[node distance=3.3cm, auto]
    \node (A) {$L$};
    \node (B) [below of=A] {$L K K^*$};
    \node (C) [right of=A] {$L E_* E$};
    \node (D) [right of=B] {$L K K^* E_* E$};
    \draw[->] (A) to node [swap] {$L \zeta$} (B);
    \draw[->] (A) to node {$L \theta$} (C);
    \draw[->] (B) to node [swap] {$L K K^* \theta$} (D);
    \draw[->] (C) to node {$L \zeta E_* E$} (D);
    \begin{scope}[shift=({A})]
        \draw +(0.25,-0.75) -- +(0.75,-0.75) -- +(0.75,-0.25);
    \end{scope}
    \node (B') [below=1.5cm of B] {$T_2 K^*$};
    \node (C') [right=1.5cm of C] {$T_1 E$};
    \node (D') at (B' -| C') {$T_2K^*E_*E$};
    \draw[->] (B) to node [swap] {$\tau_2 K^*$} node [sloped] {$\sim$} (B');
    \draw[->] (C) to node {$\tau_1 E$} node [swap] {$\sim$} (C');
    \draw[->] (B') to node [swap] {$T_2K^*\theta$} (D');
    \draw[->] (C') to node {$\phi E$} (D');
    \draw[->] (D) to node [swap,pos=0.4] {\small $\tau_2 K^*E_*E$} node [sloped] {$\sim$} (D');
   \end{tikzpicture}
  \end{center}
  Here the bottom trapezium commutes by the naturality of $\tau_2 K^*$ and the right trapezium commutes since $(\phi \tau_1) E = (\tau_2K^*E_* \circ L\zeta E_*)E$ by assumption. Note that the left edge of the large square is the mate of $\tau_2$ with respect to $K^* \dashv K$ and the top edge is the mate of $\tau_1$ with respect to $E \dashv E_*$.
  
  Now without assuming $L$ exists to start with, we can use the outer pullback diagram to define it
  and we may recover $\tau_1$ and $\tau_2$ as the mates of the resulting pullback projections.
  
  Observe that precomposing the pullback with $K$ turns the right-hand edge into an isomorphism between zero morphisms. Hence the left-hand morphism $(\tau_2 K^*)(L \zeta) K$ is an isomorphism as well. Since $\tau_2$ is given by composing this with the isomorphism $T_2 \delta$, we find that $\tau_2$ is an isomorphism.
  On the other hand, precomposing the pullback with $E_*$ turns the bottom edge into an isomorphism (as $\Nvar \xhookrightarrow{E_*} \G$ is a reflective subcategory). It follows that $\tau_1$ is also an isomorphism.
  
  Finally, we show that $\phi$ be can recovered in the appropriate way.
  The commutativity of the pullback square gives $\phi E \circ \tau_1 E \circ L\theta = T_2K^*\theta \circ \tau_2 K^* \circ L\zeta$.
  The result of whiskering this with $E_*$ on the right and composing with $T_2K^*E_* \epsilon$ is depicted in the string diagram below.
  \\ \vspace{-3pt}
  \begin{equation*}
   \begin{tikzpicture}[scale=0.7,baseline={([yshift=-0.5ex]current bounding box.center)}]
    \path coordinate[dot, label=below:$\tau_1$] (mu)
    +(0,1) coordinate[dot, label=below left:$\phi$] (d)
    +(-1,-1) coordinate (mbl)
    +(1,-1) coordinate (mbr);
    \path (d)
    +(-1.25,1.25) coordinate (minvl)
    +(1.25,1.25) coordinate (minvr)
    +(0,1.25) coordinate (minvm);
    \path (mbl) ++(0,-0.5) coordinate (a) ++(-0.75,-0.75) coordinate[dot, label=above:$\theta$] (eta) ++(-0.75,0.75) coordinate (c) ++(0,3.5) coordinate (tl);
    \path (mbr) ++ (0,-2.0) coordinate[label=below:$L$] (br);
    \path (minvl) +(0,1.0) coordinate[label=above:{\vphantom{$E$}\smash{$E_*$}}] (itl);
    \path (minvr) +(0,1.0) coordinate[label=above:{\vphantom{$T$}\smash{$T_2$}}] (itr);
    \path (minvm) +(0,1.0) coordinate[label=above:$K^*$] (itm);
    \path (tl) ++(-0.75,0.75) coordinate[dot, label=below:$\epsilon$] (eps) ++(-0.75,-0.75) coordinate (tll1);
    \coordinate[label=below:$E_*$] (bll1) at (tll1 |- br);
    \draw (tl) -- (c) to[out=-90, in=180] (eta.center) to[out=0, in=-90] (a)
              -- (mbl) to[out=90, in=180] (mu) to[out=0, in=90] (mbr) -- (br)
          (mu) -- (d.center) -- (itm)
          (itl) -- (minvl) to[out=-90, in=180] (d.center) to[out=0, in=-90] (minvr.center) -- (itr)
          (bll1) -- (tll1) to[out=90, in=180] (eps) to[out=0, in=90] (tl);
    \coordinate (cornerNW) at ($(tll1 |- itl) + (-0.5,0)$);
    \coordinate (cornerSE) at ($(br -| minvr) + (0.5,0)$);
    \begin{pgfonlayer}{bg}
    \fill[\colorH] (cornerNW) rectangle (mu |- cornerSE);
    \fill[\colorGtwo] (mu |- cornerNW) rectangle (cornerSE);
    \fill[\colorGone] (bll1) -- (tll1) to[out=90, in=180] (eps.center) to[out=0, in=90] (tl)
                    -- (c) to[out=-90, in=180] (eta.center) to[out=0, in=-90] (a)
                    -- (mbl) to[out=90, in=180] (mu.center) to[out=0, in=90] (mbr)
                    -- (br) -- cycle;
    \fill[\colorGone] (itl) -- (minvl) to[out=-90, in=180] (d.center) -- (itm) -- cycle;
    \fill[\colorN] (itm) -- (d.center) to[out=0, in=-90] (minvr.center) -- (itr) -- (itm);
    \end{pgfonlayer}
    \coordinate (tmp1) at ($(itm)!0.5!(itr) - (1pt,0)$);
    \coordinate (tmp2) at ($(itm)!0.45!(d)$);
    \node (Xlabel) at (tmp1 |- tmp2) {$\mathcal{X}$};
   \end{tikzpicture}
   \enspace=\enspace
   \begin{tikzpicture}[scale=0.7,baseline={([yshift=-0.5ex]current bounding box.center)}]
    \path coordinate[dot, label=below:$\tau_2$] (mu)
    +(0,3) coordinate[label=above:{\vphantom{$T$}\smash{$T_2$}}] (d)
    +(-1,-1) coordinate (mbl)
    +(1,-1) coordinate (mbr);
    \path (mbl) ++(0,-0.5) coordinate (a) ++(-1.0,-1.0) coordinate[dot, label=above:$\zeta$] (eta) ++(-1.0,1) coordinate (c) ++(0,4.5) coordinate[label=above:$K^*$] (tl);
    \path (mbr) ++ (0,-2.25) coordinate[label=below:$L$] (br);
    \path (tl) ++(-1,0) coordinate[label=above:{\vphantom{$E$}\smash{$E_*$}}] (tll1) ++(0,-1.25) coordinate (ettr) ++(0,-0.5) coordinate (etr) ++(-0.75,-0.75) coordinate[dot,label=above:$\theta$] (eta2) ++(-0.75,0.75) coordinate (etl) ++(0,0.5) coordinate (pr) ++(-0.75,0.75) coordinate[dot,label=below:$\epsilon$] (eps) ++(-0.75,-0.75) coordinate (pl);
    \coordinate[label=below:$E_*$] (bll) at (pl |- br);
    \draw (tl) -- (c) to[out=-90, in=180] (eta.center) to[out=0, in=-90] (a)
              -- (mbl) to[out=90, in=180] (mu.center) to[out=0, in=90] (mbr) -- (br)
          (mu.center) -- (d)
          (tll1) -- (etr) to[out=-90, in=0] (eta2.center) to[out=180, in=-90] (etl) -- (pr) to[out=90, in=0] (eps.center) to[out=180, in=90] (pl) -- (bll);
    \coordinate (cornerNW) at ($(d -| pl) + (-0.5,0)$);
    \coordinate (cornerSE) at ($(br) + (0.5,0)$);
    \begin{pgfonlayer}{bg}
    \fill[\colorN] (tl) rectangle (mu.center |- eta.center);
    \fill[\colorH] (cornerNW) rectangle (cornerSE -| tll1);
    \fill[\colorGtwo] (mu.center |- cornerNW) rectangle (cornerSE);
    \fill[\colorGone] (bll) -- (pl) to[out=90, in=180] (eps.center) to[out=0, in=90] (pr)
                    -- (etl) to[in=180, out=-90] (eta2.center) to[in=-90, out=0] (etr) -- (ettr) -- (tll1) 
                    -- (tl) -- (c) to[out=-90, in=180] (eta.center) to[out=0, in=-90] (a)
                    -- (mbl) to[out=90, in=180] (mu.center) to[out=0, in=90] (mbr)
                    -- (br) -- cycle;
    \end{pgfonlayer}
    \coordinate (tmp1) at ($(tl)!0.5!(d)$);
    \coordinate (tmp2) at ($(tl)!0.55!(mu)$);
    \node (Xlabel) at (tmp1 |- tmp2) {$\mathcal{X}$};
   \end{tikzpicture}
  \end{equation*}
  The desired equality follows after using the triangle identities to `pull the wires straight'.
\end{proof}

\begin{proposition}\label{prp:natmorphismassociation}
    Let $A$ and $B$ denote the adjoint extensions $\splitext{\Nvar}{K_1}{\G_1}{E_1}{E_{1*}}{\H}$ and $\splitext{\Nvar}{K_2}{\G_2}{E_2}{E_{2*}}{\H}$ respectively.
    There is a bijection between the equivalence classes of morphisms of extensions  $\Hom(A,B)$ and the natural transformations $\Hom(K_2^*E_{2*}, K_1^*E_{1*})$.
\end{proposition}

\begin{proof}
    Let $\psi\colon K_2^*E_{2*} \to K_1^*E_{1*}$ be a natural transformation and consider its mate natural transformation $\overline{\psi}\colon E_{2*} \to K_2K_1^*E_{1*}$ with respect to the adjunction $K_2^* \dashv K_2$.
    Notice that $\overline{\psi}$ determines a 2-cocone in the following cocomma object diagram.
    
    
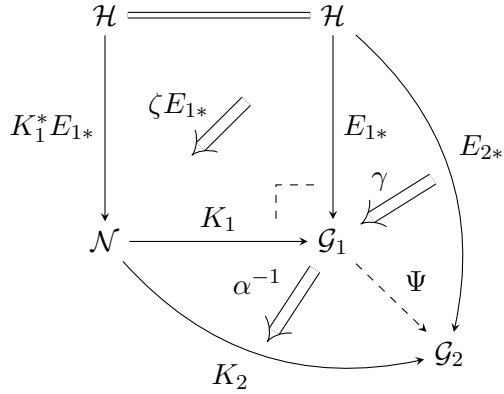
\begin{center} 
      \begin{tikzpicture}[node distance=3.0cm, auto]
        \node (A) {$\H$};
        \node (B) [below of=A] {$\Nvar$};
        \node (C) [right of=A] {$\H$};
        \node (D) [right of=B] {$\G_1$};
        \draw[->] (A) to node [swap] {$K_1^*E_{1*}$} (B);
        \draw[double equal sign distance] (A) to node {} (C);
        \draw[->] (B) to node {$K_1$} (D);
        \draw[->] (C) to node {$E_{1*}$} (D);
        
        \begin{scope}[shift=({D})]
            \draw[dashed] +(-0.25,0.75) -- +(-0.75,0.75) -- +(-0.75,0.25);
        \end{scope}
        
         \draw[double distance=4pt,-implies] ($(B)!0.62!(C)$) to node [swap] {$\zeta E_{1*}$} ($(B)!0.38!(C)$);
        
        \node (X) [below right=1.2cm and 1.2cm of D.center] {$\G_2$};
        \draw[bend right,->] (B) to node [swap] {$K_2$} node [anchor=center] (T') {} (X);
        \draw[bend left,->] (C) to node {$E_{2*}$} node [anchor=center] (T) {} (X);
        \draw[dashed, ->, pos=0.55] (D) to node {$\Psi$} (X);
        
        \draw[double distance=4pt,-implies,shorten <=2.5pt,shorten >=1pt] (T) to node [swap] {$\gamma$} (D);
        \draw[double distance=4pt,-implies,shorten <=2.5pt,shorten >=1pt] (D) to node [swap] {$\alpha\inv$} (T');
        \end{tikzpicture}
    \end{center}
    
    By the universal property of the cocomma, we get a map $\Psi \colon \G_1 \to \G_2$ and natural isomorphisms $\alpha \colon K_2 \to \Psi K_1$ and $\gamma \colon E_{2*} \to \Psi E_{1*}$. By \cref{lmm:betareduction} we can derive a unique natural isomorphism $\beta \colon E_2\Psi \to E_1$ such that $(\Psi,\alpha,\beta,\gamma)$ is a morphism of extensions.
    
    For the other direction we begin with a morphism of extensions $(\Psi,\alpha,\beta,\gamma)$ and form the pasting diagram above. We may consider the composite natural transformation \[\overline{\psi} = \alpha\inv K_1^*E_{1*} \circ \Psi\zeta E_{1*} \circ \gamma \colon E_{2*} \to K_2 K_1^*E_{1*}.\] Again we may use the adjunction $K_2^* \dashv K_2$ to arrive at the natural transformation \[\delta_2K^*_{1}E_{1*} \circ K_2^*\alpha\inv K_1^*E_{1*} \circ K_2^*\Psi\zeta E_{1*} \circ K_2^*\gamma \colon K_2^*E_{2*} \to K_1^*E_{1*},\] 
    where $\delta_2\colon K_2^*K_2 \to \Id_\Nvar$ is the counit of the adjunction.
    
    It is clear that these processes are inverses by the uniqueness of the universal property (bearing in mind that the morphisms in $\ExtOne(\H,\Nvar)$ are isomorphism classes).
\end{proof}

\begin{remark}
    Notice that the natural transformation associated to a morphism of extensions $(\Psi,\alpha,\beta,\gamma)$ in the above proof is precisely the one described in \cref{lem:associatednat}.
\end{remark}

\begin{corollary}\label{cor:hom_bijection2}
    Let $F_1, F_2\colon \H \to \Nvar$ be finite-limit-preserving functors and let \[\Gamma_1 = \splitext{\Nvar}{\pi_{1*}^{F_1}}{\Gl(F_1)}{\pi_2^{F_1}}{\pi_{2*}^{F_1}}{\H} \quad\text{and}\quad \Gamma_2 = \splitext{\Nvar}{\pi_{1*}^{F_2}}{\Gl(F_2)}{\pi_2^{F_2}}{\pi_{2*}^{F_2}}{\H}\] be the corresponding glueing extensions. Then $\Hom(F_2,F_1) \cong \Hom(\Gamma_1,\Gamma_2)$.
\end{corollary}

\begin{proof}
  Since $\pi_1^{F_1}\pi_{2*}^{F_1} = F_1$ and $\pi_1^{F_2}\pi_{2*}^{F_2} = F_2$, the above proposition implies that $\Hom(\Gamma_1,\Gamma_2) \cong \Hom(F_2,F_1)$, as required.
\end{proof}

We are now ready to show that the categories $\ExtOne(\H,\Nvar)$ and $\Hom(\H,\Nvar)\op$ are equivalent.
(A similar result can also be obtained from the results of \cite{rosebrugh1984cofibrations}.)

\begin{definition}
Let $\Gamma_{\H,\Nvar}\colon \Hom(\H,\Nvar)\op \to \ExtOne(\H,\Nvar)$ be the functor sending $F \colon \H \to \Nvar$ to the extension $\splitext{\Nvar}{\pi_{1*}}{\Gl(F)}{\pi_2}{\pi_{2*}}{\H}$ and sending natural transformations to the associated morphism of extensions described in \cref{cor:hom_bijection2}. 
\end{definition}

\begin{theorem}\label{thm:gammaequivalence}
The functor $\Gamma_{\H,\Nvar} \colon \Hom(\H,\Nvar)\op \to \ExtOne(\H,\Nvar)$ is a part of an equivalence.

An inverse $\Gamma_{\H,\Nvar}\inv$ sends an extension $\splitext{\Nvar}{K}{\G}{E}{E_*}{\H}$ to $K^*E_*$ and a morphism of extensions to the natural transformation described in \cref{lem:associatednat,prp:natmorphismassociation}. For the adjunction $\Gamma_{\H,\Nvar}\inv \dashv \Gamma_{\H,\Nvar}$ we take the counit to be the identity and the unit to be given by the isomorphisms described in \cref{thm:glueingequivalence}.
\end{theorem}

\begin{proof}
  Note that $\Gamma_{\H,\Nvar}\inv\Gamma_{\H,\Nvar} = \Id_{\Hom(\H,\Nvar)\op}$.
  We see that $\Gamma_{\H,\Nvar}\inv \dashv \Gamma_{\H,\Nvar}$ with the identity as the counit, since for each natural transformation $\psi\colon F \to \Gamma\inv(A)$, there is a unique map $\Psi\colon A \to \Gamma_{\H,\Nvar}(F)$ such that $\Gamma_{\H,\Nvar}\inv(\Psi) = \psi$, namely the image of $\psi$ under the inverse of
  the bijection \[\Hom(A,\Gamma_{\H,\Nvar}(F)) \cong \Hom(\Gamma_{\H,\Nvar}\inv\Gamma_{\H,\Nvar}(F),\Gamma_{\H,\Nvar}\inv(A)) = \Hom(F,\Gamma_{\H,\Nvar}\inv(A))\] from \cref{prp:natmorphismassociation}.
  
  Let $A$ denote the extension $\splitext{\Nvar}{K}{\G}{E}{E_*}{\H}$. It remains to show that the isomorphism of extensions $\Phi = (\Phi, \alpha, \beta, \gamma)\colon A \to \Gamma_{\H,\Nvar}(K^*E_*)$ described in \cref{thm:glueingequivalence} is the component of the unit at $A$.
  
  It suffices to show that the isomorphism $\Hom(A,\Gamma_{\H,\Nvar}\Gamma_{\H,\Nvar}\inv(A)) \cong \Hom(\Gamma_{\H,\Nvar}\inv(A),\Gamma_{\H,\Nvar}\inv(A))$ maps $\Phi$ to the identity --- that is, that $\Gamma_{\H,\Nvar}\inv(\Phi) = \id_{K^*E_*}$.
  As in \cref{lem:associatednat} we have $\Gamma_{\H,\Nvar}\inv(\Phi) = \delta' K^*E_* \circ \pi_1 \alpha\inv K^*E_* \circ \pi_1 \Phi \zeta E_* \circ \pi_1 \gamma$. Now recall that $\delta'\colon \pi_1 \pi_{1*} \to \Id_\Nvar$ is the identity, $\alpha = (\id,!) \circ \pi_{1*}\delta\inv$ and $\gamma = (\id_{K^*E_*}, \epsilon\inv)$. Thus the expression can be seen to simplify to $\delta K^*E_* \circ K^* \zeta E_*$, which in turn is $\id_{K^*E_*}$ by the triangle identity for $K^* \dashv K$.
  
  Finally, $\Gamma_{\H,\Nvar}$ and $\Gamma_{\H,\Nvar}\inv$ form an equivalence since the unit and counit are isomorphisms.
\end{proof}

With this in mind we may now consider the full subcategory of $\ExtOne(\H,\Nvar)$ whose objects are only those extensions of the form $\splitext{\Nvar}{\pi_{1*}}{\Gl(F)}{\pi_2}{\pi_{2*}}{\H}$ for some $F \colon \H \to \Nvar$. It is evident that this full subcategory is equivalent to $\ExtOne(\H,\Nvar)$ and for the remainder of the paper we will choose to perform calculations in this subcategory for simplicity. We will discuss how this can be done coherently when we investigate the $\ExtOne$ 2-functor in \cref{sec:ext_functor}.

We can now give a concrete description of the behaviour of morphisms of extensions.
Suppose that $(\Psi,\alpha,\beta,\gamma) \colon \Gamma_{\H,\Nvar}(F_1) \to \Gamma_{\H,\Nvar}(F_2)$ is a morphism of adjoint extensions as in the following diagram.

\begin{center}
    \begin{tikzpicture}[node distance=3.0cm, auto]
    \node (A) {$\Nvar$};
    \node (B) [right of=A] {$\Gl(F_1)$};
    \node (C) [right of=B] {$\H$};
    \node (D) [below of=A] {$\Nvar$};
    \node (E) [right of=D] {$\Gl(F_2)$};
    \node (F) [right of=E] {$\H$};
    \draw[->] (A) to node {$\pi_{1*}^{F_1}$} (B);
    \draw[transform canvas={yshift=0.5ex},->] (B) to node {$\pi_2^{F_1}$} (C);
    \draw[transform canvas={yshift=-0.5ex},->] (C) to node {$\pi_{2*}^{F_1}$} (B);
    \draw[->] (D) to node [swap] {$\pi_{1*}^{F_2}$} (E);
    \draw[transform canvas={yshift=0.5ex},->] (E) to node {$\pi_2^{F_2}$} (F);
    \draw[transform canvas={yshift=-0.5ex},->] (F) to node {$\pi_{2*}^{F_2}$} (E);
    \draw[->] (B) to node [swap] {$\Psi$} (E);
    \draw[double equal sign distance] (A) to (D);
    \draw[double equal sign distance] (C) to (F);
   \end{tikzpicture}
\end{center}

Since $\alpha_N \colon (N,1,!) \to \Psi(N,1,!)$ is an isomorphism, we have $\Psi(N,1,!) = (\overline{N},\overline{1},!)$ for some $\overline{N} \cong N$ and some terminal object $\overline{1}$. For simplicity, we will assume $\overline{1} = 1$ without any loss of generality.
Similarly $\gamma_H \colon (F_2(H),H,\id) \to \Psi(F_1(H),H,\id)$ is an isomorphism. So if $\Psi(F_1(H),H,\id) = (N_H,\overline{H},t_H)$ then $t_H\colon N_H \to F_2(\overline{H})$ is an isomorphism.

Since $\Psi$ preserves finite limits, we can use \cref{prop:pullback_representation_objects} to completely determine its behaviour. Every object in $\Gl(F_1)$ can be written as the following pullback diagram, where objects in the category are represented by the green arrows pointing out of the page and the pullback symbol has elongated into a wedge.

 \begin{center}
   \begin{tikzpicture}[node distance=3.5cm, auto, on grid,
    cross line/.style={preaction={draw=white, -,line width=6pt}},
    small cross line/.style={preaction={draw=white, -,line width=4pt}}
    ]
    
    \node [fill=white,inner sep=1pt, outer sep=2pt] (A) {$F_1(H)$};
    \node [right of=A] (B) {$F_1(H)$};
    \node [below of=A] (C) {$1$};
    \node [right of=C] (D) {$1$};
    
    \begin{scope}[shift=({A})]
        \draw [draw=black!50] +(0.25,-0.75) coordinate (X) -- +(0.75,-0.75) coordinate (Y) -- +(0.75,-0.25) coordinate (Z);
    \end{scope}
    
    \node [above left=1.1cm and 1.5cm of A.center] (A1) {$N$};
    \node [right of=A1] (B1) {$F_1(H)$};
    \node [below of=A1] (C1) {$N$};
    \node [right of=C1] (D1) {$F_1(H)$};
    
    \begin{scope}[shift=({A1})]
        \draw [draw=black!40] +(0.22,-0.78) coordinate (X1) -- +(0.78,-0.78) coordinate (Y1) -- +(0.78,-0.22) coordinate (Z1);
    \end{scope}
    
    \begin{pgfonlayer}{bg}
    \draw[draw=black!40] (X) to node {} (X1);
    \draw[draw=black!40,dashed] (Y) to node {} (Y1);
    \draw[draw=black!40] (Z) to node {} (Z1);
    \end{pgfonlayer}
    
    \draw[small cross line, ->] (A1) to node {$\ell$} (B1);
    \draw[double equal sign distance] (A1) to node [yshift=9pt, xshift=-2pt] {} (C1);
    \draw[double equal sign distance] (B1) to node {} (D1);
    \draw[->] (C1) to node [swap, pos=0.75] {$\ell$} (D1);
    
    \draw[cross line, double equal sign distance] (A) to node [swap, yshift=1pt, xshift=-5pt] {} (B);
    \draw[cross line, ->] (A) to node [swap, yshift=8pt] {$!$} (C);
    \draw[cross line, ->] (B) to node {$!$} (D);
    \draw[cross line, double equal sign distance] (C) to node {} (D);
    
    \draw[globjects, small cross line, ->] (A1) to node [xshift=2pt,yshift=-1pt, fill=white,inner sep=1.3pt] {$\ell$} (A);
    \draw[globjects, double equal sign distance] (B1) to node [swap] {} (B);
    \draw[globjects, ->] (C1) to node [xshift=-1pt,yshift=-1pt] {$!$} (C);
    \draw[globjects, ->] (D1) to node [xshift=-4pt,yshift=2pt] {$!$} (D);
   \end{tikzpicture}
 \end{center}
 
 Note that the front and back faces are pullback squares in $\Nvar$ and the other faces correspond to morphisms in $\Gl(F_1)$.
 
 We may now study how $\Psi$ acts on this pullback diagram.
 Observe that the bottom face corresponds to the morphism $\pi_{1*}^{F_1}(\ell)$ where $\ell\colon N \to F_1(H)$. It is then sent by $\Psi$ to the morphism represented in the diagram below.
 
  \begin{center}
  \begin{tikzpicture}[node distance=3.25cm, auto]
    \node (A) {$\overline{N}$};
    \node (B) [below of=A] {$1$};
    \node (C) [right of=A] {$\overline{F_1(H)}$};
    \node (D) [right of=B] {$1$};
    \draw[globjects,->] (A) to node [swap] {$!$} (B);
    \draw[->] (A) to node {$\pi_1\Psi \pi_{1*}^{F_1}(\ell)$} (C);
    \draw[double equal sign distance] (B) to node [swap] {} (D);
    \draw[globjects,->] (C) to node {$!$} (D);
  \end{tikzpicture}
 \end{center} 
 
 The right face is the unit $\zeta_{\pi_{2*}^{F_1}(H)}\colon \pi_{2*}^{F_1}(H) \to \pi_{1*}^{F_1}F_1(H)$. 
 Thus we have that $\Psi$ sends this face to the following commutative square.
 \begin{center}
  \begin{tikzpicture}[node distance=3.25cm, auto]
    \node (A) {$N_H$};
    \node (B) [below of=A] {$\overline{F_1(H)}$};
    \node (C) [right of=A] {$F_2(\overline{H})$};
    \node (D) [right of=B] {$1$};
    \draw[->] (A) to node [swap] {\llap{$\pi_1\Psi(\zeta_{\pi_{2*}^{F_1}(H)})$}} (B);
    \draw[globjects,->] (A) to node {$t_H$} (C);
    \draw[globjects,->] (B) to node [swap] {$!$} (D);
    \draw[->] (C) to node {$!$} (D);
  \end{tikzpicture}
 \end{center}
 
 The pullback of these two faces will then give the image of $(N,H,\ell)$ under $\Psi$.
 The pullback diagram is given by the large cuboid in the diagram below.
 Here we have factored this pullback as in the similar pullback diagram in the proof of \cref{prp:cocomma}.
 
 \begin{center}
   \begin{tikzpicture}[node distance=4.8cm, auto, on grid,
    cross line/.style={preaction={draw=white, -,line width=6pt}},
    small cross line/.style={preaction={draw=white, -,line width=4pt}}
    ]
    
    \node [fill=white,inner sep=1pt, outer sep=2pt] (A) {$F_2(H)$};
    \node [right of=A] (B) {$F_2(H)$};
    \node [below of=A] (C) {$1$};
    \node [right of=C] (D) {$1$};
    
    \begin{scope}[shift=({A})]
        \draw [draw=black!50] +(0.22,-0.78) coordinate (X) -- +(0.78,-0.78) coordinate (Y) -- +(0.78,-0.22) coordinate (Z);
    \end{scope}
    
    \node [above left=1.5cm and 2.0cm of A] (A1) {$N \times_{F_1(H)} F_2(H)$};
    \node [right of=A1] (B1) {$F_2(H)$};
    \node [below of=A1] (C1) {$N$};
    \node [right of=C1] (D1) {$F_1(H)$};
    
    \begin{scope}[shift=({A1})]
        \draw [draw=black!40] +(0.25,-0.75) coordinate (X1) -- +(0.75,-0.75) coordinate (Y1) -- +(0.75,-0.25) coordinate (Z1);
    \end{scope}
    
    \node [below of=C1] (S1) {$\overline{N}$};
    \node [below of=C] (S2) {$1$};
    \node [below of=D1] (S3) {$\overline{F_1(H)}$};
    \node [below of=D] (S4) {$1$};
    
    \node [right of=B1] (S1') {$N_H$};
    \node [right of=B] (S2') {$F_2(\overline{H})$};
    \node [right of=D1] (S3') {$\overline{F_1(H)}$};
    \node [right of=D] (S4') {$1$};
    
    \node [right of=S3] (S3'') {$\overline{F_1(H)}$};
    \node [right of=S4] (S4'') {$1$};
    
    \draw[globjects,->] (S1) to node {$!$} (S2);
    \draw[globjects,->] (S3) to node {$!$} (S4);
    \draw[globjects,->] (S3'') to node {$!$} (S4'');
    \draw[globjects,->] (S3') to node {$!$} (S4');
    \draw[globjects,->] (S1') to node {$t_H$} (S2');
    
    \draw[->] (C1) to node [swap,pos=0.6] {$\pi_1(\alpha_N)$} node [sloped,pos=0.6] {$\sim$} (S1);
    \draw[->] (D1) to node [swap,pos=0.6] {$\pi_1(\alpha_{F_1(H)})$} node [sloped,pos=0.6] {$\sim$} (S3);
    \draw[->] (D1) to node [swap,pos=0.7] {$\pi_1(\alpha_{F_1(H)})$} node [sloped,pos=0.7] {$\sim$} (S3');
    \draw[->] (B1) to node {$\pi_1(\gamma_H)$} node [swap,sloped] {$\sim$} (S1');
    
    \draw[double equal sign distance] (S3) to node {} (S3'');  
    \draw[double equal sign distance] (S4) to node {} (S4'');
    \draw[double equal sign distance] (S3') to node {} (S3'');  
    \draw[double equal sign distance] (S4') to node {} (S4'');
    
    \draw[->] (S1) to node [pos=0.75,swap]{$\pi_1\Psi\pi_{1*}^{F_1}(\ell)$} (S3);
    \draw[double equal sign distance] (S2) to node {} (S4);
    \draw[->] (S1') to node [pos=0.6,swap] {$\pi_1\Psi(\zeta_{\pi_{2*}(H)})$} (S3');
    \draw[->] (S2') to node {$!$} (S4');
    
    \draw[cross line, double equal sign distance] (C) to node {} (S2);
    \draw[cross line, double equal sign distance] (D) to node {} (S4);
    \draw[cross line, double equal sign distance] (D) to node {} (S4');
    \draw[cross line, ->] (B) to node [pos=0.28] {$F_2\pi_2(\gamma_H)$} node [swap,sloped,pos=0.28] {$\sim$} (S2');
    
    \begin{pgfonlayer}{bg}
    \draw[draw=black!40] (X) to node {} (X1);
    \draw[draw=black!40,dashed] (Y) to node {} (Y1);
    \draw[draw=black!40] (Z) to node {} (Z1);
    \end{pgfonlayer}
    
    \draw[small cross line, ->] (A1) to node {$p_{\psi_H}(\ell)$} (B1);
    \draw[->] (A1) to node [swap] {$p_\ell(\psi_H)$} (C1);
    \draw[->] (B1) to node [yshift=-8pt] {$\psi_H$} (D1);
    \draw[->] (C1) to node [swap, pos=0.75] {$\ell$} (D1);
    
    \draw[cross line, double equal sign distance, shorten <=1.2pt] (A) to node [swap, yshift=1pt, xshift=-5pt] {} (B);
    \draw[cross line, ->] (A) to node [swap, yshift=8pt] {$!$} (C);
    \draw[cross line, ->] (B) to node [swap,pos=0.4] {$!$} (D);
    \draw[cross line, double equal sign distance] (C) to node {} (D);
    
    \draw[globjects, small cross line, ->] (A1) to node [xshift=2pt,yshift=-1.5pt, fill=white,inner sep=1.2pt] {$p_{\psi_H}(\ell)$} (A);
    \draw[globjects, double equal sign distance] (B1) to node [swap] {} (B);
    \draw[globjects, ->] (C1) to node {$!$} (C);
    \draw[globjects, ->] (D1) to node [xshift=-4pt,yshift=2pt] {$!$} (D);
   \end{tikzpicture}
 \end{center}

The bottom face of the bottom left cube and the right face of the top right cube are the commutative squares considered above. These have also been extended by the identity maps in the bottom right cube, so that the bottom and right-hand faces of the full cuboid are as required for the pullback in question.

The bottom left cube commutes by the naturality of $\alpha$, while the top right cube commutes by the definition of $\psi = \Gamma_{\H,\Nvar}\inv(\Psi)$ as in \cref{prp:natmorphismassociation}.
Since $\alpha$ and $\gamma$ are isomorphisms, the top left cube is also a pullback.
Recall that the front and back faces are then also pullbacks. Since the top face of the top left cube must commute, we find that the green arrow we seek is given by $p_{\psi_H(\ell)}$. Hence, $\Psi(N, H, \ell)$ is isomorphic to $(N \times_{F_1(H)} F_2(H),H,p_{\psi_H}(\ell))$.

Of course, every natural transformation $\psi\colon F_2 \to F_1$ yields a morphism of extensions defined by such a pullback. For the associated natural isomorphisms we may take $\beta$ to be the identity and $\alpha$ to be $(\widehat{\alpha},\id)$ where $\widehat{\alpha}$ is defined by the diagram below.

\begin{center}
  \begin{tikzpicture}[node distance=3.0cm, auto]
    \node (A) {$\overline{N}$};
    \node (B) [below of=A] {$N$};
    \node (C) [right of=A] {$F_2(1)$};
    \node (D) [right of=B] {$F_1(1)$};
    \draw[->] (A) to node [swap] {$\widehat{\alpha}\inv_N$} (B);
    \draw[->] (A) to node {$!$} (C);
    \draw[->] (B) to node [swap] {$!$} (D);
    \draw[->] (C) to node {$\psi_{1}$} (D);
    \begin{scope}[shift=({A})]
        \draw +(0.25,-0.75) -- +(0.75,-0.75) -- +(0.75,-0.25);
    \end{scope}
  \end{tikzpicture}
\end{center}
Finally, we take $\gamma = (\widehat{\gamma},\id)$ where $\widehat{\gamma}$ is specified by the diagram below.

\begin{center}
  \begin{tikzpicture}[node distance=3.0cm, auto]
    \node (A) {$\overline{F_1(H)}$};
    \node (B) [below of=A] {$F_1(H)$};
    \node (C) [right of=A] {$F_2(H)$};
    \node (D) [right of=B] {$F_1(H)$};
    \draw[->] (A) to node [swap] {} (B);
    \draw[->] (A) to node {$\widehat{\gamma}\inv_H$} (C);
    \draw[double equal sign distance] (B) to node [swap] {} (D);
    \draw[->] (C) to node {$\psi_{H}$} (D);
    \begin{scope}[shift=({A})]
        \draw +(0.25,-0.75) -- +(0.75,-0.75) -- +(0.75,-0.25);
    \end{scope}
  \end{tikzpicture}
\end{center}
It is easy to see that $\epsilon_2 = \epsilon_1(\beta\pi_{2*}^{F_1})(\pi^{F_2}_2\gamma)$ as each factor is just the identity.

We now end this section with what is perhaps a surprising result about morphisms of extensions.

\begin{proposition}\label{prop:morphismofextensionshasleftadjoint}
If $(\Psi,\alpha,\beta,\gamma): \Gamma_{\H,\Nvar}(F_1) \to \Gamma_{\H,\Nvar}(F_2)$ is a morphism of adjoint extensions, then $\Psi\colon \Gl(F_1) \to \Gl(F_2)$ is a geometric morphism of toposes.
\end{proposition}

\begin{proof}
    Let $\psi\colon F_2 \to F_1$ be the natural transformation associated to $\Psi$. We can construct a functor $\Psi^*\colon \Gl(F_2) \to \Gl(F_1)$ which sends $(N,H,\ell)$ to $(N,H,\psi_H\ell)$ and leaves morphisms `fixed' in the sense that $(f,g)\colon (N_1,H_1,\ell_1) \to (N_2,H_2,\ell_2)$ is sent to $(f,g)\colon (N_1,H_1,\psi_{H_1}\ell_1) \to (N_2,H_2,\psi_{H_2}\ell_2)$, which may be seen to be a morphism in $\Gl(F_1)$ using the naturality of $\psi$.
    
    We claim that $\Psi^*$ is left adjoint to $\Psi$. To see this we consider the candidate counit $\epsilon_{N,H,\ell} = (\epsilon_\ell,\id_H)$, where $\epsilon_\ell$ is defined as in the following pullback diagram.
    \begin{center}
      \begin{tikzpicture}[node distance=3.0cm, auto]
        \node (A) {$\overline{N}$};
        \node (B) [below of=A] {$N$};
        \node (C) [right of=A] {$F_2(H)$};
        \node (D) [right of=B] {$F_1(H)$};
        \draw[->] (A) to node [swap] {$\epsilon_\ell$} (B);
        \draw[->] (A) to node {$\overline{\ell}$} (C);
        \draw[->] (B) to node [swap] {$\ell$} (D);
        \draw[->] (C) to node {$\psi_{H}$} (D);
        \begin{scope}[shift=({A})]
            \draw +(0.25,-0.75) -- +(0.75,-0.75) -- +(0.75,-0.25);
        \end{scope}
      \end{tikzpicture}
    \end{center}
    
    We must show that given a morphism $(f,g) \colon (N_1, H_1, \psi_{H_1}\ell_1) = \Psi^*(N_1,H_1,\ell_1) \to (N_2,H_2,\ell_2)$ there exists a unique morphism $(\widehat{f},\widehat{g})\colon (N_1,H_1,\ell_1) \to \Psi(N_2,H_2,\ell_2) = (\overline{N_2},H_2,\overline{\ell_2})$ such that $(\epsilon_{\ell_2},\id_{H_2}) \circ \Psi^*(\widehat{f},\widehat{g}) = (f,g)$.
    We will construct this map using the following diagram.
    
    \begin{center}
      \begin{tikzpicture}[node distance=3.0cm, auto]
        \node (A) {$\overline{N_2}$};
        \node (B) [below of=A] {$N_2$};
        \node (C) [right of=A] {$F_2(H_2)$};
        \node (D) [right of=B] {$F_1(H_2)$};
        \draw[->] (A) to node [swap] {$\epsilon_{\ell_2}$} (B);
        \draw[->] (A) to node {$\overline{\ell_2}$} (C);
        \draw[->] (B) to node [swap] {$\ell_2$} (D);
        \draw[->] (C) to node {$\psi_{H_2}$} (D);
        
        \begin{scope}[shift=({A})]
            \draw +(0.25,-0.75) -- +(0.75,-0.75) -- +(0.75,-0.25);
        \end{scope}
        
        \node (X) [above left=1.2cm and 1.2cm of A.center] {$N_1$};
        \draw[out=-90,->] (X) to node [swap] {$f$} node [anchor=center] (T') {} (B);
        \draw[out=0,->] (X) to node {$F_2(g)\ell_1$} node [anchor=center] (T) {} (C);
        \draw[dashed, ->, pos=0.55] (X) to node {$\widehat{f}$} (A);
      \end{tikzpicture}
    \end{center}
    
    Here the maps out of $N_1$ form a cone as we have $\psi_{H_2}F_2(g)\ell_1 = F_1(g)\psi_{H_1}\ell_1 = \ell_2 f$, where the first equality follows from naturality of $\psi$ and the second from the fact that $(f,g)$ is a morphism in $\Gl(F_1)$.
    
    By the universal property we have that $\overline{l_2}\widehat{f} = F_2(g)\ell_1$, which means that $(\widehat{f},g)$ is a morphism from $(N_1,H_1,\ell_1)$ to $(\overline{N_2},H_2,\overline{\ell_2})$ in $\Gl(F_2)$. It is immediate from the diagram that $(\epsilon_{\ell_2},\id_{H_2}) \circ (\widehat{f},g) = (f,g)$ and it is also not hard to see that this is the unique such morphism.
    Thus, $\Psi^*$ is indeed left adjoint to $\Psi$.
    
    Finally, we must show that $\Psi^*$ preserves finite limits. This follows immediately from the fact that finite limits in the glueing may be computed componentwise.
\end{proof}

\begin{remark}
 Notice that $\Psi^*$ is in fact a morphism of \emph{non-split} extensions in the sense that it commutes with the kernel and cokernel maps up to isomorphism. However, it does not commute with the splittings unless $\Psi$ is the identity.
\end{remark}

\subsection{Colimits in \texorpdfstring{$\ExtOne(\H,\Nvar)$}{Ext(H,N)}}

In \cite{faul2019artin} it was shown for frames $H$ and $N$ that there was something akin to a Baer sum of extensions in $\ExtOne(H,N)$. It is natural to ask if something analogous occurs in the category $\ExtOne(\H,\Nvar)$. Indeed, it is not hard to see via the equivalence with $\Hom(\H,\Nvar)\op$ that $\ExtOne(\H,\Nvar)$ has all finite colimits. The following functor will help us compute these colimits.

Let $M\colon \ExtOne(\H,\Nvar)\op \to \Cat/(\Nvar \times \H)$ be the functor sending extensions $\splitext{\Nvar}{K}{\G}{E}{E_*}{\H}$ to $!_{K^*E_*}^*\colon \Gl(K^*E_*) \to \Nvar \times \H$ where $!_{K^*E_*}^*$ is left adjoint to the universal map $!_{K^*E_*}$ in $\ExtOne(\H,\Nvar)$ out of the 2-initial object $\splitext{\Nvar}{\pi_{1*}}{\Nvar \times \H}{\pi_2}{\pi_{2*}}{\H}$. Explicitly, this adjoint sends $(N,H,\ell)$ to $(N,H)$ and $(f,g)$ to $(f,g)$.

For the action on morphisms, let $(\Psi,\alpha,\beta,\gamma)$ be a morphism of extensions and let $\psi\colon K_2^*E_{2*} \to K_1^*E_{1*}$ be the corresponding natural transformation.

\begin{center}
  \begin{tikzpicture}[node distance=2.75cm, auto]
    \node (A) {$\Nvar$};
    \node (B) [right of=A] {$\G_1$};
    \node (C) [right of=B] {$\H$};
    \node (D) [below of=A] {$\Nvar$};
    \node (E) [right of=D] {$\G_2$};
    \node (F) [right of=E] {$\H$};
    \draw[->] (A) to node {$K_1$} (B);
    \draw[transform canvas={yshift=0.5ex},->] (B) to node {$E_1$} (C);
    \draw[transform canvas={yshift=-0.5ex},->] (C) to node {$E_{1*}$} (B);
    \draw[->] (D) to node [swap] {$K_2$} (E);
    \draw[transform canvas={yshift=0.5ex},->] (E) to node {$E_2$} (F);
    \draw[transform canvas={yshift=-0.5ex},->] (F) to node {$E_{2*}$} (E);
    \draw[->] (B) to node [swap] {$\Psi$} (E);
    \draw[double equal sign distance] (A) to (D);
    \draw[double equal sign distance] (C) to (F);
  \end{tikzpicture}
\end{center}

Then $M$ maps $\Psi$ to the morphism $\overline{\Psi}^*\colon \Gl(K_2^*E_{2*}) \to \Gl(K_1^*E_{1*})$, which sends $(N,H,\ell)$ to $(N,H,\psi_H\ell)$ and $(f,g)$ to $(f,g)$. It is immediate that the necessary diagram for this to be a morphism in the slice category commutes.

\begin{proposition}\label{prp:baerlim}
The functor $M$ above creates (and preserves) finite limits. Thus, finite colimits in $\ExtOne(\H,\Nvar)$ can be computed from limits in $\Cat/(\Nvar \times \H)$.
\end{proposition}

\begin{proof}
Let $D\colon \mathcal{J} \to \ExtOne(\H,\Nvar)$ be a diagram functor with finite domain $\mathcal{J}$. To compute the colimit of $D$ we may compose $D$ with $\Gamma_{\H,\Nvar}\inv$ and compute the limit in $\Hom(\H,\Nvar)$.
Let $R\colon \H \to \Nvar$ be the resulting limit in $\Hom(\H,\Nvar)$ and $(\phi^i\colon R \to \Gamma_{\H,\Nvar}\inv D(i))_{i \in \mathcal{J}}$ be the corresponding projections.
Then $\splitext{\Nvar}{\pi_{1*}}{\Gl(R)}{\pi_2}{\pi_{2*}}{\H}$ is the colimit of $D$, where the morphisms of the colimiting cone are given in the obvious way.

If we consider the diagram functor $MD\colon \mathcal{J} \to \Cat/(\Nvar \times \H)$, then we may again compute the limit with the assistance of the calculation in $\Hom(\H,\Nvar)$. We claim that $!_R^*\colon \Gl(R) \to \Nvar \times \H$ is the required limit with the morphisms of the limiting cone given in the expected way --- that is, if $\phi$ is a morphism of the limiting cone in $\Hom(\H,\Nvar)$ then $\Gamma_{\H,\Nvar}(\phi)^* = \overline{\Phi}^*$ is the associated morphism in $\Cat/(\Nvar \times \H)$.

We must demonstrate that this cone satisfies the universal property. Suppose we have some other cone $(\Xi_i\colon \C \to \Gl(K_i^*E_{i*}))_{i \in \mathcal{J}}$ and consider the following diagram in $\Cat$ where $\Psi = D(f)$ for some morphism $f\colon i \to j$ in $\mathcal{J}$.

\begin{center}
  \begin{tikzpicture}[node distance=3.3cm, auto]
    \node (A) {$\Nvar \times \H$};
    \node (B) [right of=A] {$\Gl(K_j^*E_{j*})$};
    \node (C) [below of=A] {$\Gl(K_i^*E_{i*})$};
    \node (D) [right of=C] {$\C$};
    \draw[->] (B) to node [swap] {$!_{K_j^*E_{j*}}^*$} (A);
    \draw[->] (C) to node {$!_{K_i^*E_{i*}}^*$} (A);
    \draw[->] (B) to node {$\overline{\Psi}^*$} (C);
    \draw[->] (D) to node [swap] {$\Xi_j$} (B);
    \draw[->] (D) to node {$\Xi_i$} (C);
  \end{tikzpicture} 
\end{center}

Since each $\Xi_i$ is a morphism in $\Cat/(\Nvar \times \H)$ we have that it commutes with the $!$ maps. This means that the $\Xi$ maps all agree on the first two components. If we assume that $\Xi_k(C) = (N_C,H_C,\ell^k_C)$, then $\Xi_i = \overline{\Psi}^* \Xi_j$ gives $\ell^i_C = \psi_{H_C}\ell^j_C$. Now consider the following diagram in $\Nvar$ where we make use of the aforementioned limiting cone in $\Hom(\H,\Nvar)$.

\begin{center}
  \begin{tikzpicture}[node distance=3.0cm, auto]
    \node (A) {$R(H_C)$};
    \node (B) [below of=A] {$K^*_iE_{i*}(H_C)$};
    \node (C) [right of=A] {$K^*_jE_{j*}(H_C)$};
    \draw[->] (A) to node [swap,pos=0.43] {$\phi^i_{H_C}$} (B);
    \draw[->] (A) to node {$\phi^j_{H_C}$} (C);
    \draw[->] (C) to node {$\psi_{H_C}$} (B);
    
    \node (X) [above left=1.2cm and 1.2cm of A.center] {$N_C$};
    \draw[out=-90,->] (X) to node [swap] {$\ell^i_C$} node [anchor=center] (T') {} (B);
    \draw[out=0,->] (X) to node {$\ell^j_C$} node [anchor=center] (T) {} (C);
    \draw[dashed, ->, pos=0.55] (X) to node {$\overline{\ell}_C$} (A);
  \end{tikzpicture}
\end{center}

Here we use the universal property of $R$ componentwise at $H_C$ to produce the map $\overline{\ell}_C$. This allows us to construct a map $S \colon \C \to \Gl(R)$ with $S(C) = (N_C,H_C,\overline{\ell}_C)$.
As for morphisms, now note that each $\Xi_k$ sends $f\colon C \to C'$ to the `same' pair $(f_1,f_2)$ and we define $S$ to act on morphisms in the same way. The pair $(f_1,f_2)$ can be seen to be a morphism in $\Gl(R)$ from $S(C)$ to $S(C')$ by considering the above diagram in the functor category and then using the naturality of $\overline{\ell}\colon \pi_1\Xi_k \to R\pi_2\Xi_k$. This morphism $S$ is the desired map and is easily seen to be unique.

From the above it is clear that $M$ preserves limits and that every limiting cone of $MD$ is isomorphic to one of the form $(M\Gamma_{\H,\Nvar}(\phi^i)\colon M \Gamma_{\H,\Nvar}(R) \to M D(i))_{i \in \mathcal{J}}$, where $\phi^i \colon R \to \Gamma_{\H,\Nvar}\inv D(i)$ is the limiting cone in $\Hom(\H,\Nvar)$. For $M$ to create limits, it remains to show that every cone of $D$ which maps to a limiting cone of $MD$ is isomorphic to one of the above form. This follows since $M$ is conservative.
\end{proof}

Notice that the limit diagram was embedded into the slice category so that each $\Xi$ in the proof would agree on the first two components. If the limit diagram is connected, this will happen automatically and so we obtain the following corollary.

\begin{corollary}
The functor $M\colon \ExtOne(\H,\Nvar) \to \Cat$ sending extensions $\splitext{\Nvar}{K}{\G}{E}{E_*}{\H}$ to $\Gl(K^*E_*)$ and acting on morphisms as in \cref{prp:baerlim} creates finite connected limits. 
\end{corollary}

\smallskip
A disconnected (co)limit is the subject of the following example.
\par\vspace{-10pt}
\begin{example}
 Let us consider the coproduct of the extensions $\splitext{\Nvar}{K_1}{\G_1}{E_1}{E_1*}{\H}$ and $\splitext{\Nvar}{K_2}{\G_2}{E_2}{E_2*}{\H}$.
 Since products in a slice category correspond to pullbacks, we may construct this coproduct using the following pullback in $\Cat$.
 \begin{center}
  \begin{tikzpicture}[node distance=3.0cm, auto]
    \node (A) {$\mathcal{P}$};
    \node (B) [below of=A] {$\Gl(K^*_1E_{1*})$};
    \node (C) [right of=A] {$\Gl(K^*_2E_{2*})$};
    \node (D) [right of=B] {$\Nvar \times \H$};
    \draw[->] (A) to node [swap] {} (B);
    \draw[->] (A) to node {} (C);
    \draw[->] (B) to node [swap] {$!^*_{K_1^*E_{1*}}$} (D);
    \draw[->] (C) to node {$!^*_{K_2^*E_{2*}}$} (D);
    \begin{scope}[shift=({A})]
        \draw +(0.25,-0.75) -- +(0.75,-0.75) -- +(0.75,-0.25);
    \end{scope}
  \end{tikzpicture}
\end{center}
If $!_\mathcal{P}$ is the composite morphism from $\mathcal{P}$ to $\Nvar \times \H$, then the coproduct extension may be recovered as $\splitext{\Nvar}{(\pi_1 !_\mathcal{P})_*}{\mathcal{P}}{\pi_2!_\mathcal{P}}{(\pi_2!_\mathcal{P})_*}{\H}$.
\end{example}

\section{The \texorpdfstring{$\ExtOne$}{Ext} functor}\label{sec:ext_functor}

Given that we have established that $\ExtOne(\H,\Nvar)$ is equivalent to $\Hom(\H,\Nvar)\op$, it is natural to ask if $\ExtOne$ can be extended to a 2-bifunctor and if $\ExtOne$ and the $\Homop$ will then be 2-naturally equivalent (where $\Homop = \mathrm{Op} \circ \Hom\co$ and $\mathrm{Op}$ is the opposite category 2-functor).

The answer is of course ``yes'', for if $T\colon \H' \to \H$ and $S \colon \Nvar \to \Nvar'$, all we need do is define $\ExtOne(T,S) = \Gamma_{\H',\Nvar'} \circ \Homop(T,S) \circ \Gamma_{\H,\Nvar}\inv$ (and similarly for natural transformations).
However this is unsatisfactory, as there is already established behaviour for how an $\ExtOne$ functor ought to act on objects and morphisms. In this section, we show that the above definition conforms with the usual expectations of an $\ExtOne$ functor.

We consider each component of our $\ExtOne$ functor separately and begin by describing $\ExtOne(-,\Nvar)$.
In other contexts (for instance, see \cite{borceux2005internal}) the extension functor can be obtained from a fibration. In the protomodular setting, we start from the `fibration of points' sending split epimorphisms to their codomain. In the more general setting of $\mathcal{S}$-protomodularity (see \cite{bourn2015Sprotomodular})
we consider only a certain subclass of split epimorphisms. This suggests we consider a 2-fibration sending open subtopos adjunctions to the codomain of their inverse image functors.

A categorification of the Grothendieck construction (see \cite{buckley2014fibred}) gives that 2-fibrations correspond to 3-functors into $2\Cat$. Fortunately, aside from motivation, we will largely be able to avoid 3-functors for the same reasons that $\ExtTwo(\H,\Nvar)$ is essentially a 1-category (\cref{cor:extensions_only_1_cat}).

While the paradigmatic example of a fibration is the codomain fibration, which maps from the whole arrow category to the base category, the domain of the analogous 2-fibration is restricted to the category of fibrations. See \cite{buckley2014fibred} for more details on 2-fibrations.

The fibre 3-functor $\FLTop\coop \to 2\Cat$ corresponding to the 2-fibration $\mathrm{Cod}\colon \mathrm{Fib}_{\FLTop} \to \FLTop$ can be described as follows (omitting the description of the coherence data for simplicity):
\begin{itemize}
    \item On objects it sends a topos $\mathcal{E}$ to the slice 2-category of finite-limit-preserving fibrations from toposes to $\mathcal{E}$.
    \item On 1-morphisms it sends a finite-limit-preserving functor $T\colon \mathcal{E'} \to \mathcal{E}$ to the 2-functor $T^\star$ corresponding to pulling back along $T$.
    \item On 2-morphisms it sends a natural transformation $\tau\colon T \to S$ to a 2-natural
    transformation $\tau^\star$ from $S^\star$ to $T^\star$. The component $\tau^\star_E\colon S^\star(E) \to T^\star(E)$ indexed by the finite-limit-preserving fibration $E\colon \mathcal{D} \to \mathcal{E}$ can be constructed in three steps.
    \begin{itemize}
        \item First define the morphism of fibrations $(P_S, S)\colon S^\star(E) \to E$ given by the pullback projection.
        \begin{center}
         \begin{tikzpicture}[node distance=3.0cm, auto]
          \node (A) {$\overline{\mathcal{D}}_S$};
          \node (B) [below of=A] {$\mathcal{E'}$};
          \node (C) [right=2.75cm of A] {$\mathcal{D}$};
          \node (D) [below of=C] {$\mathcal{E}$};
          \draw[->] (A) to node [swap] {$\phantom{E}\mathllap{S^\star(E)}$} (B);
          \draw[->] (A) to node {$P_S$} (C);
          \draw[->] (B) to node [swap] {$S$} (D);
          \draw[->] (C) to node {$E$} (D);
          \begin{scope}[shift=({A})]
            \draw +(0.25,-0.75) -- +(0.75,-0.75) -- +(0.75,-0.25);
          \end{scope}
          \node (E) at ($(B)!0.5!(C)$) {\large $\cong$};
         \end{tikzpicture}
        \end{center}
        \item Then we may use the fibration property of $E$ to lift the natural transformation $\tau S^\star(E)$ to a natural transformation into $P_S$. Explicitly, for an object $X \in \overline{\mathcal{D}}_S$, the morphism $\tau_{S^\star(E)(X)}\colon T(S^\star(E)(X)) \to S(S^\star(E)(X))$ can be lifted to a morphism in $\mathcal{D}$ with codomain $P_S(X)$. These lifted morphisms assemble into a natural transformation $\overline{\tau}$ from a new functor $L\colon \overline{\mathcal{D}}_S \to \mathcal{D}$ to $P_S$. This functor sends a morphism $f\colon X \to Y$ to the morphism obtained by factoring $P_S(f) \overline{\tau}_X$ through $\overline{\tau}_Y$ as shown in the diagram below.
        \begin{center}
         \scalebox{0.95}{
         \begin{tikzpicture}[node distance=3.5cm, auto]
          \node (A) {$L(X)$};
          \node (B) [below of=A] {$L(Y)$};
          \node (C) [right=2.75cm of A] {$P_S(X)$};
          \node (D) [below of=C] {$P_S(Y)$};
          \draw[dashed,->] (A) to node [swap] {$L(f)$} (B);
          \draw[->] (A) to node {$\overline{\tau}_X$} (C);
          \draw[->] (B) to node [swap] {$\overline{\tau}_Y$} (D);
          \draw[->] (C) to node {$P_S(f)$} (D);
          
          \node (A') [below=4cm of B] {$T(S^\star(E)(X))$};
          \node (B') [below of=A'] {$T(S^\star(E)(Y))$};
          \node (C') [below=4cm of D] {$S(S^\star(E)(X))$};
          \node (D') [below of=C'] {$S(S^\star(E)(Y))$};
          \draw[->] (A') to node [swap] {$T(S^\star(E)(f))$} (B');
          \draw[->] (A') to node {$\tau_{S^\star(E)(X)}$} (C');
          \draw[->] (B') to node [swap] {$\tau_{S^\star(E)(Y)}$} (D');
          \draw[->] (C') to node {$S(S^\star(E)(f))$} (D');
          
          \coordinate (M1) at ($(A)!0.5!(C)$);
          \coordinate (M2) at ($(B')!0.5!(D')$);
          \draw[semithick,darkgray] (M1) ++(-50:8cm and 5cm) arc (-50:-130:8cm and 5cm);
          \path (M1) ++(-90:8cm and 5cm) coordinate (N1);
          \draw[semithick,darkgray] (M2) ++(50:8cm and 5cm) arc (50:130:8cm and 5cm);
          \path (M2) ++(90:8cm and 5cm) coordinate (N2);
          \draw[|->] ($(N1) - (0,0.25)$) to node {$E$} ($(N2) + (0,0.25)$);
         \end{tikzpicture}
         }
        \end{center}
        \item Finally, consider the 2-pullback of $E$ and $T$ and note that the maps $L\colon \overline{\mathcal{D}}_S \to \mathcal{D}$ and $S^\star(E)\colon \overline{\mathcal{D}}_S \to \mathcal{E'}$ form a cone as shown below.
        \begin{center}
         \begin{tikzpicture}[node distance=3.25cm, auto]
          \node (A) {$\overline{\mathcal{D}}_T$};
          \node (B) [below of=A] {$\mathcal{E'}$};
          \node (C) [right=2.75cm of A] {$\mathcal{D}$};
          \node (D) [below of=C] {$\mathcal{E}$};
          \draw[->] (A) to node [pos=0.45,swap] {$\phantom{E}\mathllap{T^\star(E)}$} (B);
          \draw[->] (A) to node {$P_T$} (C);
          \draw[->] (B) to node [swap] {$T$} (D);
          \draw[->] (C) to node {$E$} (D);
          \begin{scope}[shift=({A})]
            \draw +(0.25,-0.75) -- +(0.75,-0.75) -- +(0.75,-0.25);
          \end{scope}
          \node (E) at ($(B)!0.5!(C)$) {\large $\cong$};
          
          \node (X) [above left=1.2cm and 1.2cm of A.center] {$\overline{\mathcal{D}}_S$};
          \draw[out=-105,in=150,->] (X) to node [pos=0.35,swap] {$S^\star(E)$} node [anchor=center] (T') {} (B);
          \draw[out=15,in=120,->] (X) to node {$L$} node [anchor=center] (T) {} (C);
          \draw[dashed, ->, pos=0.55] (X) to node {$\tau^\star_E$} (A);
          
          \node (E2) at ($(A)!0.6!(T)$) {\large $\cong$};
          \node (E3) at ($(A)!0.5!(T')$) [yshift=5pt,xshift=-5pt] {\large $\cong$};
         \end{tikzpicture}
        \end{center}
        Thus, we may factor these through $P_T$ and $T^\star(E)$ respectively to obtain a functor from $\overline{\mathcal{D}}_S$ to $\overline{\mathcal{D}}_T$.
        This is the desired functor $\tau^\star_E\colon S^\star(E) \to T^\star(E)$.
    \end{itemize}
    The coherent set of 2-isomorphisms for the 2-natural transformation can also be obtained by the cartesian property of the lifted maps and universal property of the 2-pullback.
\end{itemize}
We can then easily modify this to describe the fibre 3-functor for the 2-fibration of (open) points.
Moreover, to obtain $\ExtTwo(-,\Nvar)\colon \FLTop\coop \to 2\Cat$ we restrict to inverse image functors of open subtoposes (equipped with right adjoint splittings) with fixed kernel object $\Nvar$.
The above discussion restricts easily to this case, since these functors are stable under pullback along finite-limit-preserving functors and the relevant morphisms can be shown to be morphisms of extensions.
To see this we will use the following folklore result, which we prove here for completeness.

\begin{proposition}\label{prop:commaobject}
Let $\C$ be a 2-category, $F\colon \H \to \Nvar$ and $T\colon \H' \to \H$ 1-morphisms. Then the comma object $\Gl(FT)$ can be represented as a (strict) 2-pullback (where we draw the 2-cokernels of the extension horizontally).
\begin{center}
  \begin{tikzpicture}[node distance=3.5cm, auto]
    \node (A) {$\Gl(FT)$};
    \node (B) [below of=A] {$\Gl(F)$};
    \node (C) [right of=A] {$\H'$};
    \node (D) [right of=B] {$\H$};
    \draw[->] (A) to node [swap] {$Q$} (B);
    \draw[->] (A) to node {$\pi^{FT}_2$} (C);
    \draw[->] (B) to node [swap] {$\pi^F_2$} (D);
    \draw[->] (C) to node {$T$} (D);
    \begin{scope}[shift=({A})]
        \draw +(0.25,-0.75) -- +(0.75,-0.75) -- +(0.75,-0.25);
    \end{scope}
 \end{tikzpicture}
\end{center}
\end{proposition}

\begin{proof}
We must first describe $Q$. Consider the following comma object diagrams.

\begin{center}
  \begin{tikzpicture}[node distance=3.5cm, auto]
    \node (A) {$\Gl(F)$};
    \node (B) [below of=A] {$\Nvar$};
    \node (C) [right of=A] {$\H$};
    \node (D) [right of=B] {$\Nvar$};
    \node (A') [right of=C] {$\Gl(FT)$};
    \node (B') [below of=A'] {$\Nvar$};
    \node (C') [right of=A'] {$\H'$};
    \node (D') [right of=B'] {$\Nvar$};
    \draw[->] (A) to node [swap] {$\pi^F_1$} (B);
    \draw[->] (A) to node {$\pi^F_2$} (C);
    \draw[double equal sign distance] (B) to node {} (D);
    \draw[->] (C) to node {$F$} (D);
    \begin{scope}[shift=({A})]
        \draw[dashed] +(0.25,-0.75) -- +(0.75,-0.75) -- +(0.75,-0.25);
    \end{scope}

    \draw[->] (A') to node [swap] {$\pi^{FT}_1$} (B');
    \draw[->] (A') to node {$\pi^{FT}_2$} (C');
    \draw[double equal sign distance] (B') to node {} (D');
    \draw[->] (C') to node {$FT$} (D');
    \begin{scope}[shift=({A'})]
        \draw[dashed] +(0.25,-0.75) -- +(0.75,-0.75) -- +(0.75,-0.25);
    \end{scope}
     \draw[double distance=4pt,-implies] ($(B)!0.38!(C)$) to node {$\lambda_1$} ($(B)!0.62!(C)$);
      \draw[double distance=4pt,-implies] ($(B')!0.38!(C')$) to node {$\lambda_2$} ($(B')!0.62!(C')$);
  \end{tikzpicture}
\end{center}
Here $\lambda_1 \colon \pi^F_1 \to F\pi^F_2$ and $\lambda_2\colon \pi^{FT}_1 \to FT\pi^{FT}_2$ are the universal (not necessarily invertible) 2-morphisms. Explicitly, we have $\lambda_1 = \pi_1^F \theta^F$ and $\lambda_2 = \pi_1^{FT} \theta^{FT}$.

Now the map $Q$ is given by the universal property of the comma object $\Gl(F)$ applied to the 2-cone given by $\lambda_2 \colon \pi^{FT}_1 \to FT\pi^{FT}_2$.

\begin{center}
  \begin{tikzpicture}[node distance=3.25cm, auto]
    \node (A) {$\Gl(F)$};
    \node (B) [below of=A] {$\Nvar$};
    \node (C) [right of=A] {$\H$};
    \node (D) [right of=B] {$\Nvar$};
    \draw[->] (A) to node [swap] {$\pi^F_1$} (B);
    \draw[->] (A) to node {$\pi^F_2$} (C);
    \draw[double equal sign distance] (B) to node {} (D);
    \draw[->] (C) to node {$F$} (D);
    
    \begin{scope}[shift=({A})]
        \draw[dashed] +(0.25,-0.75) -- +(0.75,-0.75) -- +(0.75,-0.25);
    \end{scope}
    
    \draw[double distance=4pt,-implies] ($(B)!0.38!(C)$) to node {$\lambda_1$} ($(B)!0.62!(C)$);
    
    \node (X) [above left=1.2cm and 1.2cm of A.center] {$\Gl(FT)$};
    \draw[out=-90,->] (X) to node [swap] {$\pi^{FT}_1$} node [anchor=center] (T') {} (B);
    \draw[out=0,->] (X) to node {$T\pi^{FT}_2$} node [anchor=center] (T) {} (C);
    \draw[dashed, ->, pos=0.55] (X) to node {$Q$} (A);
    
    \draw[double distance=4pt,-implies,shorten <=2.5pt,shorten >=1pt] (A) to node [swap] {$\mu$} (T);
    \draw[double distance=4pt,-implies,shorten <=2.5pt,shorten >=1pt] (T') to node [pos=0.75] {$\nu\inv$} (A);
  \end{tikzpicture}
\end{center}

Here $\nu$ and $\mu$ are 2-isomorphisms and satisfy that $F\mu \circ \lambda_1Q \circ \nu\inv = \lambda_2$. Concretely, we have that $Q(N,H,\ell) = (N,T(H),\ell)$ and $Q(f,g) = (f,T(g))$ and that $\mu$ and $\nu$ are both the identity.

Now we paste our candidate 2-pullback square to our other comma object diagrams as follows.

\begin{center}
  \begin{tikzpicture}[node distance=3.25cm, auto]
    \node (A) {$\Gl(FT)$};
    \node (B) [below of=A] {$\Gl(F)$};
    \node (C) [right of=A] {$\H'$};
    \node (D) [right of=B] {$\H$};
    \node (B') [below of=B] {$\Nvar$};
    \node (D') [below of=D] {$\Nvar$};
    \draw[->] (A) to node [swap] {$Q$} (B);
    \draw[->] (A) to node {$\pi^{FT}_2$} (C);
    \draw[->] (B) to node [swap] {$\pi^F_2$} (D);
    \draw[->] (C) to node {$T$} (D);
    \begin{scope}[shift=({A})]
        \draw +(0.25,-0.75) -- +(0.75,-0.75) -- +(0.75,-0.25);
    \end{scope}
    \draw[double distance=4pt,-implies] ($(B)!0.38!(C)$) to node {$\mu$} ($(B)!0.62!(C)$);
   \draw[->] (B) to node [swap] {$\pi^F_1$} (B');
    \draw[double equal sign distance] (B') to node {} (D');
    \draw[->] (D) to node {$F$} (D');
    \begin{scope}[shift=({B})]
        \draw[dashed] +(0.25,-0.75) -- +(0.75,-0.75) -- +(0.75,-0.25);
    \end{scope}
    \draw[double distance=4pt,-implies] ($(B')!0.38!(D)$) to node {$\lambda_1$} ($(B')!0.62!(D)$);
  \end{tikzpicture}
\end{center}

Since $\pi_1^FQ = \pi_1^{FT}$,
we see that the pasting diagram above is just the comma object diagram corresponding to $\Gl(FT)$.
It follows that the upper square is a 2-pullback in a similar manner to (the converse direction of) the pullback lemma.
\end{proof}

We will now describe the $\ExtOne$ `functor' explicitly and at the same time demonstrate its relationship to $\Hom$.

The 3-functor $\ExtTwo(-,\Nvar)$ sends a topos $\H$ to the 2-category of extensions $\ExtTwo(\H,\Nvar)$ as defined above. In fact, since $\ExtTwo(\H,\Nvar)$ is equivalent to a 1-category, it turns out that $\ExtTwo(-,\Nvar)$ factors through $\Cat \hookrightarrow 2\Cat$ and so can be treated as a 2-functor $\ExtOne(-,\Nvar)\colon \FLTop\coop \to \Cat$.

The following computations will all be performed in the equivalent full subcategory of $\ExtOne(\H,\Nvar)$ whose objects are of the form $\splitext{\Nvar}{\pi_{1*}}{\Gl(F)}{\pi_2}{\pi_{2*}}{\H}$. The fact that this may be done coherently is due to the 2-categorical analogue of the result which says that if $F\colon \B \to \C$ is a functor and for each $B$ we have an isomorphism $F(B) \cong C_B$, then there is a functor $F'$ such that $F'(B) = C_B$ (and which acts on morphisms by conjugating the result of $F$ by the appropriate isomorphisms).

We can use \cref{prop:commaobject} to show how $\ExtOne(-,\Nvar)$ acts on 1-morphisms.
If $T\colon \H' \to \H$ is a finite-limit-preserving functor, then it is enough to describe how $\ExtOne(T,\Nvar)\colon \ExtOne(\H,\Nvar) \to \ExtOne(\H',\Nvar)$ acts on extensions in Artin-glueing form. It should `take the 2-pullback along $T$' and hence it sends the object $\splitext{\Nvar}{\pi^F_{1*}}{\Gl(F)}{\pi^F_2}{\pi^F_{2*}}{\H}$ to $\splitext{\Nvar}{\pi^{FT}_{1*}}{\Gl(FT)}{\pi^{FT}_2}{\pi^{FT}_{2*}}{\H'}$ as in the following diagram.

\begin{center}
  \begin{tikzpicture}[node distance=2.75cm, auto]
    \node (A) {$\Gl(FT)$};
    \node (B) [below of=A] {$\Gl(F)$};
    \node (C) [right=2cm of A] {$\H'$};
    \node (D) [below of=C] {$\H$};
    \node (E) [left=2cm of A] {$\Nvar$};
    \node (F) [below of=E] {$\Nvar$};
    \draw[->] (A) to node [swap] {} (B);
    \draw[transform canvas={yshift=0.5ex}, ->] (A) to node {$\pi^{FT}_2$} (C);
    \draw[transform canvas={yshift=0.5ex}, ->] (B) to node {$\pi^F_2$} (D);
    \draw[->] (C) to node {$T$} (D);
    \draw[->] (E) to node {$\pi^{FT}_{1*}$} (A);
    \draw[->] (F) to node [swap] {$\pi_{1*}^F$} (B);
    \draw[transform canvas={yshift=-0.5ex}, ->] (D) to node {$\pi^F_{2*}$} (B);
    \draw[transform canvas={yshift=-0.5ex}, ->] (C) to node {$\pi^{FT}_{2*}$} (A);
    \draw[double equal sign distance] (E) to node {} (F);
    
    \begin{scope}[shift=({A})]
        \draw +(0.25,-0.75) -- +(0.75,-0.75) -- +(0.75,-0.25);
    \end{scope}
  \end{tikzpicture}
\end{center}

The functor $\ExtOne(T,\Nvar)$ acts on morphisms via the universal property of the 2-pullback. Let $\Psi$ be a morphism of extensions corresponding to the natural transformation $\psi\colon F_2 \to F_1$ and consider the following diagram, noting that $T\pi^{F_1T}_2 = \pi^{F_2}_2\Psi Q$ and hence we have a 2-cone. Also recall that $\Psi$ can always be to chosen to correspond to pulling back along $\psi$.

\begin{center}
  \begin{tikzpicture}[node distance=4.5cm, auto, on grid,
    cross line/.style={preaction={draw=white, -,line width=6pt}}
    ]
    
    \node (A) {$\Gl(F_2T)$};
    \node [right of=A] (B) {$\H'$};
    \node [below of=A] (C) {$\Gl(F_2)$};
    \node [right of=C] (D) {$\H$};
    
    \begin{scope}[shift=({A})]
        \draw [draw=black] +(0.25,-0.75) coordinate (X) -- +(0.75,-0.75) coordinate (Y) -- +(0.75,-0.25) coordinate (Z);
    \end{scope}
    
    \node [above left=1.4cm and 1.9cm of A.center] (A1) {$\Gl(F_1T)$};
    \node [right of=A1] (B1) {$\H'$};
    \node [below of=A1] (C1) {$\Gl(F_1)$};
    \node [right of=C1] (D1) {$\H$};
    
    \begin{scope}[shift=({A1})]
        \draw [draw=black!70] +(0.22,-0.78) coordinate (X1) -- +(0.78,-0.78) coordinate (Y1) -- +(0.78,-0.22) coordinate (Z1);
    \end{scope}
    
    \draw[->] (A1) to node {$\pi^{F_1T}_2$} (B1);
    \draw[->] (A1) to node [swap] {$Q$} (C1);
    \draw[->] (B1) to node [pos=0.65] {$T$} (D1);
    \draw[->] (C1) to node [node on layer=over,fill=white,inner sep=2pt,outer sep=0pt,yshift=1pt] (K) [pos=0.25] {$\pi^{F_1}_2$} (D1);
    
    \draw[cross line, ->] (A) to node {$\pi^{F_2T}_2$} (B);
    \draw[cross line, ->] (A) to node [swap, pos=0.35] {$Q'$} (C);
    \draw[cross line, ->] (B) to node {$T$} (D);
    \draw[cross line, ->] (C) to node {$\pi^{F_2}_2$} (D);
    
    \begin{scope}
    \begin{pgfonlayer}{over}
    \clip (K.south west) rectangle (K.north east);
    \draw[->, opacity=0.7] (A) to (C);
    \end{pgfonlayer}
    \end{scope}
    
    \draw[dashed,->] (A1) to node [yshift=-2pt,pos=0.55] {$\ExtOne(T,\Nvar)(\Psi)$} (A);
    \draw[double equal sign distance] (B1) to node {} (B);
    \draw[->] (C1) to node [swap, yshift=2pt] {$\Psi$} (C);
    \draw[double equal sign distance] (D1) to node {} (D);
   \end{tikzpicture}
\end{center}
We shall now show that we may take $\ExtOne(T,\Nvar)(\Psi)$ to be $(\Gamma_{\H',\Nvar} \circ \Hom(T,\Nvar) \circ \Gamma_{\H,\Nvar}\inv)(\Psi) = \Gamma_{\H',\Nvar}(\psi T)$ --- that is, the latter functor is given by the universal property of the 2-pullback. We must also supply two 2-morphisms corresponding to the left and top faces of the above cube. Both may be taken to be the identity. Note in fact that now each face of the cube commutes strictly.

We only need to check that $\pi_2^{F_2T}\Gamma_{\H',\Nvar}(\psi T) = \pi_2^{F_1T}$ and that $Q'\Gamma_{\H',\Nvar}(\psi T) = \Psi Q$. The former is immediate and the latter follows because $\Psi Q(N,H,\ell) = \Psi(N,T(H),\ell)$ is given by the following pullback.

\begin{center}
  \begin{tikzpicture}[node distance=3.25cm, auto]
    \node (A) {$N'$};
    \node (B) [below of=A] {$N$};
    \node (C) [right of=A] {$F_2T(H)$};
    \node (D) [right of=B] {$F_1T(H)$};
    \draw[->] (A) to node [swap] {} (B);
    \draw[->] (A) to node {$\ell'$} (C);
    \draw[->] (B) to node [swap] {$\ell$} (D);
    \draw[->] (C) to node {$\psi_{T(H)}$} (D);
    \begin{scope}[shift=({A})]
        \draw +(0.25,-0.75) -- +(0.75,-0.75) -- +(0.75,-0.25);
    \end{scope}
  \end{tikzpicture}
\end{center}
This is readily seen to be the same pullback which determines $\Gamma_{\H',\Nvar}(\psi T)(N,H,\ell)$.

Finally, we discuss how $\ExtOne(-,\Nvar)$ acts on 2-morphisms. We follow the construction outlined above for the codomain 2-fibration. Let $\tau \colon T \to T'$ be a natural transformation. Then we describe the natural transformation $\ExtOne(\tau,\Nvar)\colon \ExtOne(T',\Nvar) \to \ExtOne(T,\Nvar)$ componentwise.

Without loss of generality we may describe each component at extensions of the form $\Gamma_{\H,\Nvar}(F)$. Consider the following diagram.

\begin{center}
  \begin{tikzpicture}[node distance=3.25cm, auto]
    \node (A) {$\H'$};
    \node (B) [below of=A] {$\H$};
    \node (C) [right of=A] {$\Gl(FT')$};
    \node (D) [right of=B] {$\Gl(F)$};
    \node (B') [right=3.25cm of D] {$\Gl(F)$};
    \node (A') [above of=B'] {$\Gl(FT)$};
    \node (C') [right of=A'] {$\H'$};
    \node (D') [right of=B'] {$\H$};
    \draw[->] (A) to node [swap] {$T'$} (B);
    \draw[->] (C) to node [swap] {$\pi^{FT}_2$} (A);
    \draw[->] (D) to node {$\pi^F_2$} (B);
    \draw[->] (C) to node [swap] {$P_{T'}$} (D);
    \begin{scope}[shift=({C})]
        \draw +(-0.25,-0.75) -- +(-0.75,-0.75) -- +(-0.75,-0.25);
    \end{scope}
    
    \draw[->] (A') to node {$P_T$} (B');
    \draw[->] (A') to node {$\pi^{FT'}_2$} (C');
    \draw[->] (B') to node [swap] {$\pi^F_2$} (D');
    \draw[->] (C') to node {$T$} (D');
    \begin{scope}[shift=({A'})]
        \draw +(0.25,-0.75) -- +(0.75,-0.75) -- +(0.75,-0.25);
    \end{scope}
    \draw[double equal sign distance] (D) to node [swap] {} (B');
    \draw[dashed, ->] (C) to node {$\ExtOne(\tau,\Nvar)_{\Gamma_{\H,\Nvar}(F)}$} (A');
    \draw[->, highlightArrow] (C) to node [xshift=-2pt,yshift=-2pt] (T'') {$L_\tau$} (B');
    \draw[double distance=4pt,-implies,highlightArrow] ($(T'')!0.3!(D)$) to node [swap,pos=0.3] {$\overline{\tau}$} ($(T'')!0.6!(D)$);
  \end{tikzpicture}
\end{center}

As discussed in the case of the codomain fibration, we may define a functor $L_\tau$ as follows.
We have that $P_{T'}(N,H,\ell) = (N,T'(H),\ell)$ lies above the codomain of $\tau_H \colon T(H) \to T'(H)$ with respect to the fibration $\pi^F_2$ (see \cref{prp:pi2fib}) and so by the universal property of the fibration we get a map $\overline{\tau}_{(N,H,\ell)}\colon (\overline{N},T(H),\overline{\ell}) \to (N,T'(H),\ell)$ as given by the pair of vertical morphisms in the following pullback diagram.

\begin{center}
  \begin{tikzpicture}[node distance=3.25cm, auto]
    \node (A) {$\overline{N}$};
    \node (B) [below of=A] {$N$};
    \node (C) [right of=A] {$FT(H)$};
    \node (D) [right of=B] {$FT'(H)$};
    \draw[->] (A) to node [swap] {$\pi_1^F \overline{\tau}_{(N,H,\ell)}$} (B);
    \draw[->] (A) to node {$\overline{\ell}$} (C);
    \draw[->] (B) to node [swap] {$\ell$} (D);
    \draw[->] (C) to node {$\tau_{F(H)}$} (D);
    \begin{scope}[shift=({A})]
        \draw +(0.25,-0.75) -- +(0.75,-0.75) -- +(0.75,-0.25);
    \end{scope}
 \end{tikzpicture}
\end{center}

These morphisms form a natural transformation $\overline{\tau}\colon L_\tau \to P_{T'}$,
where $L_\tau$ is the functor which sends $(N,H,\ell)$ to $(\overline{N},T(H),\overline{\ell})$.
As above, this functor factors through $P_T$ to give a functor $\ExtOne(\tau,\Nvar)_{\Gamma_{\H,\Nvar}(F)}\colon \Gl(FT') \to \Gl(FT)$, which sends $(N,H,\ell)$ to $(\overline{N},H,\overline{\ell})$.
It remains to show that this gives a morphism of split extensions, but it is clear from the above pullback diagram that this functor is the morphism of split extensions corresponding to $\tau F$ (which itself is equal to $\Hom(\tau,\Nvar)_F$).

The morphisms $\ExtOne(\tau,\Nvar)_{\Gamma_{\H,\Nvar}(F)}$ define the desired natural transformation $\ExtOne(\tau,\Nvar)$. (Naturality follows from the interchange law or from the general theory of the codomain 2-fibration.)

The 2-functor $\ExtOne(-,\Nvar)$ composes strictly, and for the unitors observe that $\ExtOne(\Id_\H,\Nvar)$ is equal to $\Gamma_{\H,\Nvar}\Gamma_{\H,\Nvar}\inv\colon \ExtOne(\H,\Nvar) \to \ExtOne(\H,\Nvar)$ so that we make take as our unitors the unit of the adjunction $\Gamma_{\H,\Nvar}$. The necessary 2-functor axioms can then easily be seen to hold.

\begin{remark}
 The above argument also proves that the 2-functor sending adjoint extensions with fixed kernel $\Nvar$ to their cokernels is a 2-fibration.
\end{remark}

\begin{theorem}\label{family1}
 The 2-functors $\ExtOne(-,\Nvar)$ and $\Homop(-,\Nvar)$ are 2-naturally equivalent via $\Gamma^\Nvar\colon \Homop(-,\Nvar) \to \ExtOne(-,\Nvar)$ defined as follows:
  \begin{enumerate}
    \item for each $\H \in \FLTop$ we have the equivalence $\Gamma^\Nvar_\H = \Gamma_{\H,\Nvar} \colon \Hom(\H,\Nvar)\op \to \ExtOne(\H,\Nvar)$,
    \item for each $T\colon \H' \to \H$ we have the identity $\Gamma_{\H',\Nvar}\Hom(T,\Nvar)\op = \ExtOne(T,\Nvar)\Gamma_{\H,\Nvar}$.
  \end{enumerate}
\end{theorem}

\begin{proof}
The equality in point (2) is clear by inspection of the definition of $\ExtOne(T,\Nvar)$.
The proof of the coherence conditions is easy. In particular, the first coherence condition is satisfied because each morphism of the diagram is the identity. Similarly, for the second condition again all morphisms are the identity (though marginally more work is required to show that the unitor at $H$ whiskered with $\Gamma_{\H,\Nvar}$ is in fact the identity).
\end{proof}

We turn our attention to the functor $\ExtOne(\H,-)\colon \FLTop\co \to \Cat$.
We could not find an elegant description of this in terms of a 2-fibration.
However, we believe a reasonable definition can be given by dualising our arguments for $\ExtOne(-,\Nvar)$.

Naturally, for an object $\Nvar$ we have that $\ExtOne(\H,\Nvar)$ is just the category of adjoint split extensions.

For 1-morphisms, consider $S\colon \Nvar \to \Nvar'$. We would have $\ExtOne(\H,S)$ act on objects by sending $\splitext{\Nvar}{K}{\G}{E}{E_*}{\H}$ to the extension resulting from a pushout of $K$ along $S$ and on morphisms via the universal property of the pushout.

\begin{center}
  \begin{tikzpicture}[node distance=2.75cm, auto]
    \node (A) {$\G$};
    \node (B) [below of=A] {$\mathcal{P}$};
    \node (C) [right of=A] {$\H$};
    \node (D) [right of=B] {$\H$};
    \node (E) [left of=A] {$\Nvar$};
    \node (F) [left of=B] {$\Nvar'$};
    \draw[->] (A) to node [swap] {} (B);
    \draw[transform canvas={yshift=0.5ex}, ->] (A) to node {$E$} (C);
    \draw[transform canvas={yshift=0.5ex}, ->] (B) to node {$\coker(P)$} (D);
    \draw[double equal sign distance] (C) to node {} (D);
    \draw[->] (E) to node {$K$} (A);
    \draw[->] (F) to node [swap] {$P$} (B);
    \draw[transform canvas={yshift=-0.5ex}, ->] (D) to node {$\coker(P)$} (B);
    \draw[transform canvas={yshift=-0.5ex}, ->] (C) to node {$E_*$} (A);
    \draw[->] (E) to node [swap] {$S$} (F);
    
    \begin{scope}[shift=({B})]
        \draw +(-0.25,0.75) -- +(-0.75,0.75) -- +(-0.75,0.25);
    \end{scope}
  \end{tikzpicture}
\end{center}

To see that this is well-defined we prove the following result dual to \cref{prop:commaobject}.

\begin{proposition}
    Let $F \colon \H \to \Nvar$ and $S\colon \Nvar \to \Nvar'$ be finite-limit-preserving functors. Then $\Gl(SF)$ is given by the following 2-pushout.
    \begin{center}
      \begin{tikzpicture}[node distance=3.25cm, auto]
        \node (A) {$\Nvar$};
        \node (B) [below of=A] {$\Nvar'$};
        \node (C) [right of=A] {$\Gl(F)$};
        \node (D) [right of=B] {$\Gl(SF)$};
        \draw[->] (A) to node [swap] {$S$} (B);
        \draw[->] (A) to node {$\pi_{1^*}$} (C);
        \draw[->] (B) to node [swap] {} (D);
        \draw[->] (C) to node {$P$} (D);
        \begin{scope}[shift=({D})]
            \draw +(-0.25,0.75) -- +(-0.75,0.75) -- +(-0.75,0.25);
        \end{scope}
      \end{tikzpicture}
    \end{center}
\end{proposition}

\begin{proof}
By \cref{prp:cocomma}, we know that $\Gl(F)$ is a cocomma object in the 2-category $\FLCat$. Thus it is a comma object in $\FLCat\op$.
Applying \cref{prop:commaobject} and then reversing the arrows gives the desired result.
It is not hard to see that $P(N,H,\ell) = (S(N),H,S(\ell))$ and $P(f,g) = (S(f),g)$.
\end{proof}

Thus, fixing particular pushouts we can describe $\ExtOne(\H,S)$ concretely as sending an extension $\splitext{\Nvar}{K}{\G}{E}{E_*}{\H}$ to $\splitext{\Nvar}{\pi_{1*}}{\Gl(SK^*E_*)}{\pi_2}{\pi_{2*}}{\H}$.

As mentioned above, $\ExtOne(\H,S)$ should act on morphisms by the universal property of the pushout. Let $\Psi = (\Psi, \alpha_1, \beta_1, \gamma_1)$ be a morphism of extensions and consider the 2-cocone given by $P_2\Psi\pi^{F_1}_{1*} \xrightarrow[]{P_2\alpha\inv_1} P_2\pi_{1*}^{F_2} \Id_\Nvar = \pi_{1*}^{SF_2} \Id_{\Nvar'} S$ as in the following pasting diagram (where we omit the 2-morphisms do avoid clutter).

\begin{center}
  \begin{tikzpicture}[node distance=4.5cm, auto, on grid,
    cross line/.style={preaction={draw=white, -,line width=6pt}}
    ]
    
    \node (A) {$\Nvar$};
    \node [right of=A] (B) {$\Gl(F_2)$};
    \node [below of=A] (C) {$\Nvar'$};
    \node [right of=C] (D) {$\Gl(SF_2)$};
    
    \begin{scope}[shift=({D})]
        \draw [draw=black] +(-0.25,0.75) coordinate (X) -- +(-0.75,0.75) coordinate (Y) -- +(-0.75,0.25) coordinate (Z);
    \end{scope}
    
    \node [above left=1.4cm and 1.9cm of A.center] (A1) {$\Nvar$};
    \node [right of=A1] (B1) {$\Gl(F_1)$};
    \node [below of=A1] (C1) {$\Nvar'$};
    \node [right of=C1] (D1) {$\Gl(SF_1)$};
    
    \begin{scope}[shift=({D1})]
        \draw [draw=black!70] +(-0.22,0.78) coordinate (X1) -- +(-0.78,0.78) coordinate (Y1) -- +(-0.78,0.22) coordinate (Z1);
    \end{scope}
    
    \draw[->] (A1) to node {$\pi^{F_1}_{1*}$} (B1);
    \draw[->] (A1) to node [swap] {$S$} (C1);
    \draw[->] (B1) to node [pos=0.65] {$P_1$} (D1);
    \draw[->] (C1) to node [node on layer=over,fill=white,inner sep=2pt,outer sep=0pt,yshift=1pt] (K) [pos=0.28] {$\pi^{SF_1}_{1*}$} (D1);
    
    \draw[cross line, ->] (A) to node [pos=0.4] {$\pi^{F_2}_{1*}$} (B);
    \draw[cross line, ->] (A) to node [swap, pos=0.33] {$S$} (C);
    \draw[cross line, ->] (B) to node {$P_2$} (D);
    \draw[cross line, ->] (C) to node [swap] {$\pi^{SF_2}_{1*}$} (D);
    
    \begin{scope}
    \begin{pgfonlayer}{over}
    \clip (K.south west) rectangle (K.north east);
    \draw[->, opacity=0.7] (A) to (C);
    \end{pgfonlayer}
    \end{scope}
    
    \draw[double equal sign distance] (A1) to node [yshift=-2pt] {} (A);
    \draw[->] (B1) to node {$\Psi$} (B);
    \draw[double equal sign distance] (C1) to node [swap, yshift=2pt] {} (C);
    \draw[dashed,->] (D1) to node [swap,pos=0.35,yshift=2pt] {$\ExtOne(\H,S)(\Psi)$} (D);
  \end{tikzpicture}
\end{center}

Here the front, back and left faces commute strictly and the top face has associated invertible 2-morphism $\alpha_1\inv$.

Let $\psi$ be the natural transformation associated to $\Psi$. We will show that the map given by the universal property is $\Gamma_{\H,\Nvar'}(S\psi) = (\Gamma_{\H,\Nvar'}(S\psi),\alpha_2,\beta_2,\gamma_2)$. We define the associated 2-morphisms for the universal property of the 2-pushout as follows.
For the bottom face we use $\alpha_2\inv\colon \Gamma_{\H,\Nvar'}(S\psi)\pi_{1*}^{SF_1} \to \pi_{1*}^{SF_2}$.

As for the right-hand face, we note that $P_2 \Psi(N,H,\ell)$ is given by the following diagram.

\begin{center}
  \begin{tikzpicture}[node distance=3.0cm, auto]
    \node (A) {$S(\overline{N})$};
    \node (B) [below of=A] {$S(N)$};
    \node (C) [right of=A] {$SF_2(H)$};
    \node (D) [right of=B] {$SF_1(H)$};
    \draw[->] (A) to node [swap] {$S(\overline{\psi_H})$} (B);
    \draw[->] (A) to node {$S(\overline{\ell})$} (C);
    \draw[->] (B) to node [swap] {$S(\ell)$} (D);
    \draw[->] (C) to node {$S\psi_H$} (D);
    \begin{scope}[shift=({A})]
        \draw +(0.25,-0.75) -- +(0.75,-0.75) -- +(0.75,-0.25);
    \end{scope}
  \end{tikzpicture}
\end{center}
On the other hand, we have that $\Gamma_{\H,\Nvar'}(S\psi)P_1(N,H,\ell)$ is given by a similar diagram as follows.
\begin{center}
  \begin{tikzpicture}[node distance=3.0cm, auto]
    \node (A) {$\overline{S(N)}$};
    \node (B) [below of=A] {$S(N)$};
    \node (C) [right of=A] {$SF_2(H)$};
    \node (D) [right of=B] {$SF_1(H)$};
    \draw[->] (A) to node [swap] {$\overline{S(\psi_H)}$} (B);
    \draw[->] (A) to node {$\overline{S(\ell)}$} (C);
    \draw[->] (B) to node [swap] {$S(\ell)$} (D);
    \draw[->] (C) to node {$S\psi_H$} (D);
    \begin{scope}[shift=({A})]
        \draw +(0.25,-0.75) -- +(0.75,-0.75) -- +(0.75,-0.25);
    \end{scope}
  \end{tikzpicture}
\end{center}

Since $S$ preserves pullbacks, there is a natural family of isomorphisms $\mu_{(N,H,\ell)} = (\widehat{\mu}_{(N,H,\ell)},\id_H)$ from $(S(\overline{N}),H,S(\overline{\ell}))$ to $(\overline{S(N)},H,\overline{S(\ell)})$, which in particular satisfies that $\overline{S(\psi_H)}\widehat{\mu}_{(N,H,\ell)} = S(\overline{\psi_H})$. We take $\mu$ to be the 2-morphism associated to the right-hand face.

Now let us discuss the 2-cocone 2-morphism $P_2\alpha_1\inv$ in more detail. The component at $N$ is given by $(S\widehat{\alpha}_{1,N}\inv,\id)$ for $\widehat{\alpha}$ as defined \cref{sec:equivalence}.
Notice that $S(\widehat{\alpha}_{1,N}\inv)$ is given by the following diagram.
\begin{center}
  \begin{tikzpicture}[node distance=3.0cm, auto]
    \node (A) {$S(\overline{N})$};
    \node (B) [below of=A] {$S(N)$};
    \node (C) [right of=A] {$SF_2(1)$};
    \node (D) [right of=B] {$SF_1(1)$};
    \draw[->] (A) to node [swap] {$S(\widehat{\alpha}_{1,N}\inv)$} (B);
    \draw[->] (A) to node {$S(!)$} (C);
    \draw[->] (B) to node [swap] {$S(!)$} (D);
    \draw[->] (C) to node {$S \psi_1$} (D);
    \begin{scope}[shift=({A})]
        \draw +(0.25,-0.75) -- +(0.75,-0.75) -- +(0.75,-0.25);
    \end{scope}
  \end{tikzpicture}
\end{center}
We see that $S(\widehat{\alpha}_1\inv) = S(\overline{\psi_1})$. Similarly, we have $\overline{S(\psi_1)} = \widehat{\alpha}_2\inv S$. Thus, the equation $\overline{S(\psi_H)}\mu_N = S(\overline{\psi_H})$ from above reduces in this case to $\widehat{\alpha}_{2,N}\inv S \circ \widehat{\mu}_{(N,1,!)} = S \widehat{\alpha}_{1,N}\inv$. Consequently, we have $\alpha_2\inv S \circ \mu \pi_{1*}^{F_1} = P_2\alpha_1\inv$,
and hence $\Gamma_{\H,\Nvar}(S\psi)$ equipped with $\alpha_2\inv$ and $(\mu,\id)$ is indeed the desired map from the universal property of the pushout.

In summary, we have $\ExtOne(\H,S)(\Psi) = \Gamma_{\H,\Nvar'}(S\psi)$, which
completes the description of $\ExtOne(\H,S)$.

Finally, let $\sigma\colon S \to S'$ be a natural transformation. We shall describe the natural transformation $\ExtOne(\H,\sigma)\colon \ExtOne(\H,S') \to \ExtOne(\H,S)$ componentwise. We define $\ExtOne(\H,\sigma)_{\Gamma_{\H,\Nvar}(F)}$ through its left adjoint as follows.

\begin{center}
  \begin{tikzpicture}[node distance=3.0cm, auto]
    \node (A) {$\Nvar$};
    \node (B) [above of=A] {$\Nvar'$};
    \node (C) [right of=A] {$\Gl(F)$};
    \node (D) [right of=B] {$\Gl(S'F)$};
    \node (B') [right=2.7cm of D] {$\Gl(SF)$};
    \node (A') [below of=B'] {$\Gl(F)$};
    \node (C') [right of=A'] {$\Nvar$};
    \node (D') [right of=B'] {$\Nvar'$};
    \draw[->] (A) to node {$S'$} (B);
    \draw[->] (A) to node [swap] {$\pi^{F}_{1*}$} (C);
    \draw[->] (B) to node {$\pi^{S'F}_{1*}$} (D);
    \draw[->] (C) to node {$P_{S'}$} (D);
    \begin{scope}[shift=({D})]
        \draw +(-0.25,-0.75) -- +(-0.75,-0.75) -- +(-0.75,-0.25);
    \end{scope}

    \draw[->] (A') to node [swap] {$P_S$} (B');
    \draw[->] (C') to node {$\pi^{F}_{1*}$} (A');
    \draw[->] (D') to node [swap] {$\pi^{SF}_{1*}$} (B');
    \draw[->] (C') to node [swap] {$S$} (D');
    \begin{scope}[shift=({B'})]
        \draw +(0.25,-0.75) -- +(0.75,-0.75) -- +(0.75,-0.25);
    \end{scope}
    \draw[double equal sign distance] (C) to node [swap] {} (A');
    \draw[dashed, ->] (B') to node [swap] {$\ExtOne(\H,\sigma)_{\Gamma(F)}^*$} (D);
    \draw[->, highlightArrow] (A') to node [xshift=-2pt,yshift=-2pt] (T'') [swap] {$L_\sigma$} (D);
    \draw[double distance=4pt,-implies,highlightArrow] ($(T'')!0.3!(C)$) to node [swap,pos=0.3] {$\overline{\sigma}$} ($(T'')!0.6!(C)$);
  \end{tikzpicture}
\end{center}

Notice the resemblance of the above diagram to the one that arose when defining $\ExtOne(\tau,\Nvar)_{\Gamma_{\H,\Nvar}(F)}$. It has the same basic structure, but all 1-morphisms are pointing in the opposite direction. 

First we define $L_\sigma$ and $\overline{\sigma}$. Note that $P_{S'}(N,H,\ell) = (S'(N),H,S'(\ell))$ lies above the codomain of $\sigma_N$, with respect to the fibration $\pi_1^{S'F}$ (see \cref{prp:pi1fib}). Thus by the universal property, we may lift $\sigma_N$ to a map $\overline{\sigma_N}\colon (S(N),H,S'(\ell)\sigma_N) \to (S'(N),H,S'(\ell))$. We may define a functor $L_\sigma\colon \Gl(F) \to \Gl(S'F)$ which sends objects $(N,H,\ell)$ to $(S(N),H,S'(\ell)\sigma_N)$ and which sends morphisms $(f,g)\colon (N,H,\ell) \to (N',H',\ell')$ to $(S(f),g)$.
The pair $(S(f),g)$ can be seen to be a morphism in $\Gl(S'F)$ by considering the following diagram.
\begin{center}
    \begin{tikzpicture}[node distance=2.75cm, auto]
    \node (A) {$S(N)$};
    \node (B) [right of=A] {$S'(N)$};
    \node (C) [right of=B] {$S'F(H)$};
    \node (D) [below of=A] {$S(N')$};
    \node (E) [right of=D] {$S'(N')$};
    \node (F) [right of=E] {$S'F(H')$};
    \draw[->] (A) to node {$\sigma_N$} (B);
    \draw[transform canvas={yshift=0.5ex},->] (B) to node {$S'(\ell)$} (C);
    \draw[->] (D) to node [swap] {$\sigma_{N'}$} (E);
    \draw[transform canvas={yshift=0.5ex},->] (E) to node [swap] {$S'(\ell')$} (F);
    \draw[->] (B) to node [swap] {$S'(f)$} (E);
    \draw[->] (A) to node [swap] {$S(f)$} (D);
    \draw[->] (C) to node {$S'F(g)$} (F);
   \end{tikzpicture}
\end{center}
The left-hand square commutes by naturality of $\sigma$ and the right-hand square commutes since $(f,g)$ is morphism in $\Gl(F)$.
Now the $\overline{\sigma_N}$ arrange into a natural transformation $\overline{\sigma}\colon L_\sigma \to P_{S'}$.

If things are to behave dually, we should have $L_\sigma$ factor through $P_S$. By the naturality of $\sigma$ we have that the following diagram commutes.

\begin{center}
  \begin{tikzpicture}[node distance=3.0cm, auto]
    \node (A) {$S(N)$};
    \node (B) [below of=A] {$S(F(H))$};
    \node (C) [right of=A] {$S'(N)$};
    \node (D) [right of=B] {$S'(F(H))$};
    \draw[->] (A) to node [swap] {$S(\ell)$} (B);
    \draw[->] (A) to node {$\sigma_N$} (C); 
    \draw[->] (B) to node [swap] {$\sigma_{F(H)}$} (D);
    \draw[->] (C) to node {$S'(\ell)$} (D);
  \end{tikzpicture}
\end{center}

Thus, observe that $L_\sigma(N,H,\ell) = (S(N),H,S'(\ell)\sigma_N) = (S(N),H,\sigma_{F(H)}S(\ell))$. This perspective allows us to factor $L_\sigma$ as $L_\sigma = \Gamma_{\H,\Nvar'}(\sigma F)^* \circ P_S$. (Recall that the left adjoint to the functor $\Gamma_{\H,\Nvar'}(\sigma F)$ sends an object $(N,H,\ell)$ to $(N,H,\sigma_{F(H)}\ell)$.)

We might now hope to take $\ExtOne(\H,\sigma)_{F(H)}$ to be this resulting factor $\Gamma_{\H,\Nvar'}(\sigma F)^*$. However, this map goes in the `wrong' direction.
We can remedy this by taking the right adjoint and setting $\ExtOne(\H,\sigma)_{F(H)} = \Gamma_{\H,\Nvar'}(\sigma F)$.

In order to specify $\ExtOne(\H,-)$ completely, it only remains to discuss the compositors and unitors. As before we have that $\ExtOne(\H,-)$ composes strictly and we take the unitors to be the unit of the adjunction $\Gamma_{\H,\Nvar}^* \dashv \Gamma_{\H,\Nvar}$ in \cref{thm:gammaequivalence}.

We can express the relationship between $\ExtOne(\H,-)$ and $\Hom(\H,-)$ as follows.

\begin{theorem}\label{family2}
 $\ExtOne(\H,-)$ and $\Homop(\H,-)$ are 2-naturally equivalent via $\Gamma^\H\colon \Homop(\H,-) \to \ExtOne(\H,-)$ defined as follows
  \begin{enumerate}
    \item for each $\Nvar \in \FLTop$ we have the equivalence $\Gamma^\H_\Nvar = \Gamma_{\H,\Nvar} \colon \Hom(\H,\Nvar)\op \to \ExtOne(\H,\Nvar)$, 
    \item for each $S\colon \Nvar \to \Nvar'$ we have the identity $\Gamma_{\H,\Nvar}\Hom(\H,S)\op = \ExtOne(\H,S)\Gamma_{\H,\Nvar'}$.
  \end{enumerate}
\end{theorem}

\begin{proof}
  Just as before, the equality in point (2) is clear by inspection of the definition of $\ExtOne(\H,S)$ and the necessary coherence conditions hold, because each involved morphism is an identity.
\end{proof}

A bifunctor theorem for 2-functors was discussed in \cite{faul20202dimensional} and gives the precise conditions that allow two families of 2-functors $M_B \colon \C \to \D$ and $L_C\colon \B \to \D$ to be collated into a bifunctor $P\colon \B \times \C \to \D$ for which $P(B,-)$ is isomorphic to $M_B$ and $P(-,C)$ is isomorphic to $L_C$. These conditions are that $L_C(B) = M_B(C)$ and that for each $f\colon B_1 \to B_2$ in $\B$ and $g\colon C_1 \to C_2$ in $\C$ there exists an invertible 2-morphism $\chi_{f,g}\colon L_{C_2}(g)M_{B_1}(f) \to M_{B_2}(f)L_{C_1}(g)$ satisfying certain coherence conditions reminiscent of those for distributive laws for monads. Such families together with $\chi$ are called a \emph{distributive law of 2-functors}.

For the families $\ExtOne(-,\Nvar)$ and $\ExtOne(\H,-)$ the first condition is immediate. Moreover, it is not hard to see that if $T\colon \H' \to \H$ and $S\colon \Nvar \to \Nvar'$ that $\ExtOne(\H,S)\ExtOne(T,\Nvar) = \ExtOne(T,\Nvar')\ExtOne(\H',S)$
so that we might choose $\chi_{T,S}$ to be the identity. The coherence conditions are then immediate.

Thus, we may apply the results of \cite{faul20202dimensional} to arrive at the 2-functor $(\ExtOne,\omega,\kappa)\colon \FLTop\op \times \FLTop \to \Cat$ defined below.

\begin{definition}
Let $(\ExtOne,\omega,\kappa)\colon \FLTop\coop \times \FLTop\co \to \Cat$ be the 2-functor defined as follows.
\begin{enumerate}
    \item $\ExtOne(\H,\Nvar)$ is the category of extensions of $\H$ by $\Nvar$,
    \item $\ExtOne(T,S) = \ExtOne(T,\Nvar')\ExtOne(\H',S)$ for functors $S\colon \Nvar \to \Nvar'$ and $T\colon \H' \to \H$,
    \item $\ExtOne(\tau,\sigma) = \ExtOne(\tau,\Nvar') \ast \ExtOne(\H',\sigma)$ for 2-morphisms $\sigma\colon S \to S'$ and $\tau\colon T \to T'$,
    \item $\omega$ is the identity,
    \item $\kappa_{\H,\Nvar}$ is given by the unit $\Phi$ of the adjunction $\Gamma_{\H,\Nvar}^* \dashv \Gamma_{\H,\Nvar}$ as defined in \cref{thm:glueingequivalence,thm:gammaequivalence}.
\end{enumerate}
\end{definition}

The 2-bifunctor $\Homop$ can be recovered as the collation of the functors obtained by fixing one of its components, $\Hom_{\mathrm{op}}(\H,-)$ and $\Hom_{\mathrm{op}}(-,\Nvar)$.
It is shown in \cite{faul20202dimensional} that `morphisms between distributive laws' can also be collated to give 2-natural transformations between the corresponding bifunctors. The 2-natural equivalences in \cref{family1,family2} can be collected into a 2-natural equivalence $\Gamma\colon \ExtOne \to \Hom_\mathrm{op}$ provided that $\Gamma^\Nvar_\H = \Gamma^\H_\Nvar$ and the Yang--Baxter equation holds. These conditions are immediate in our setting and so we obtain the following theorem.

\begin{theorem}
 The 2-functors $\ExtOne$ and $\Homop$ are 2-naturally equivalent via $\Gamma \colon\ExtOne \to \Homop$ in which $\Gamma_{\H,\Nvar} = \Gamma^\H_\Nvar$ and $\Gamma_{T,S}$ is the identity for all functors $S\colon \Nvar \to \Nvar'$ and $T\colon \H' \to \H$. 
\end{theorem}

\bibliographystyle{abbrv}
\bibliography{bibliography}
\end{document}